\documentclass{amsart}

\usepackage{amssymb, amsthm, amsmath, amscd}

\usepackage[pdftex]{graphicx}
\usepackage{subfig}
\usepackage{float}
\usepackage[below]{placeins}

\newtheorem{lemma}{Lemma}
\newtheorem{prop}[lemma]{Proposition}
\newtheorem{theorem}[lemma]{Theorem}

\newtheorem*{varlemma}{Lemma \ref{lem:7-4}}
\newtheorem*{varprop}{Proposition \ref{prop:A3}}

\newtheorem{definition}[lemma]{Definition}
\newtheorem{cor}[lemma]{Corollary}

\newcommand{\wilde}{\widetilde}

\newcommand{\btau}{\bar{\tau}}
\newcommand{\pd}{ {\partial}}
\newcommand{\R}{\mathbf R}
\renewcommand{\S}{\mathbf S}

\newcommand{\lin}{\mathcal{L}}
\newcommand{\vcal}{\mathcal{V}}

\newcommand{\ical}{\mathcal{I}}
\newcommand{\hcal}{\mathcal{H}}
\newcommand{\scal}{\mathcal{S}}
\newcommand{\rcal}{\mathcal{R}}
\newcommand{\ecal}{\mathcal{E}}
\newcommand{\ncal}{\mathcal{N}}
\newcommand{\F}{\mathcal{F}}

\newcommand{\uvarphi}{\underline{\varphi}}
\newcommand{\utheta}{\underline{\theta}}

\newcommand{\mcal}{\mathcal{M}}

\newcommand{\fcal}{\mathcal{F}}

\newcommand{\eps}{\varepsilon}
\newcommand{\ubar} {\underline}
\begin{document}

\title[Self-Translating Surfaces under MCF]{Complete Embedded Self-Translating Surfaces under Mean Curvature Flow.  }
\author{Xuan Hien Nguyen}
\thanks{This work is partially supported by the National Science Foundation, Grant No. DMS-0908835.}
\address{Department of Mathematics, Kansas State University, Manhattan, KS 66506} 
\email{xhnguyen@math.ksu.edu}

\begin{abstract}
We describe a construction of complete embedded self-translating surfaces under mean curvature flow by desingularizing the intersection of a finite family of grim reapers in general position.
\end{abstract}

\maketitle

\section{Introduction}

Let  $X(\cdot,t):M^d \to \mathbf{R}^{d+1}$ be a one-parameter family of immersions of smooth hypersurfaces into $\mathbf{R}^{d+1}$. The family of hypersurfaces $M_t =X(M^d, t)$ is a \emph{solution to the mean curvature flow} if
\begin{equation}
\label{eq:sol-mcf}
\begin{split}
\frac{\partial}{\partial t}X(p,t)=\mathbf{H}(p,t),&\quad p \in M, t>0,\\
X(p, 0) = X_0(p), &
\end{split}
\end{equation}
where $\mathbf{H}(p,t)$ is the mean curvature vector of the hypersurface $M_t$ at the point $X(p,t)$ for some initial data given by the immersion $X_0$. By the local existence theorem for parabolic equations, the flow can be continued past any time $t$ as long as the norm of the second fundamental form  $|A(p,t)|$ stays bounded on $M_t$. A singularity at time $T$ is classified according to the rate at which $\max_{p \in M_t} |A(p,t)|$ blows up: if $\max_{p \in M_t} |A(p,t)| \sqrt{T-t} \leq C$, we say the singularity is fast forming, otherwise we say the singularity is slow forming. The behavior of the flow near fast forming singularities is modeled by self-similar surfaces, which are surfaces that are rescaled by the flow, while their shape is left unchanged \cite{huisken;asymptotic-behavior}. The study of slow forming singularities is more complex because of the lack of control on the geometry of the solution.

In the present article,  we work in dimension $d=2$ and focus on surfaces that are translated by the mean curvature flow at constant speed. These surfaces are called self-translating surfaces (STS) and can give some insight in  slow forming singularities. For example, if the initial hypersurface has nonnegative mean curvature and the blow up is slow, then the surfaces $\{M_t\}$ tend asymptotically to a strictly convex STS or $\R^{d-k} \times S^d$, where $S^d$ is a lower dimensional strictly convex STS \cite{huisken-sinestrari}. Detailed examples of asymptotic convergence are given in Angenent  \cite{angenent;formation-singularities} and Angenent-Vel\'azquez \cite{angenent-velazquez;cusp-singularities}\cite{angenent-velazquez;degenerate-neckpinches}.  Although the study of STS and singularities of the mean curvature flow are linked, few examples are available. Besides the classic examples of a plane, a grim reaper cylinder, and a rotationally symmetric soliton, Altschuler and Wu  \cite{altschuler-wu;translating-surfaces}  showed the existence of  paraboloid type self-translating surfaces that are graphs over convex domains in $\R^2$ having a prescribed angle of contact with the boundary cylinder.  In  \cite{mine;tridents}, we constructed STS by desingularizing the intersection of a grim reaper and a plane. We present here a more general result with a family of grim reapers as the initial configuration. 

To find an equation for STS, we set  $\frac{\partial}{\partial t}X= \mathbf H= \mathbf a + \mathbf V$, where $\mathbf a$ is the constant velocity of the translation and $\mathbf V$ is a vector field tangent to the surface $M_t$ to account for possible reparametrizations of the surface. Without loss of generality, we can fix $\mathbf a = \vec e_y$ to be the second coordinate unit vector in $\R^3$, and taking the inner product with the normal vector, we get
	\begin{equation}
	\label{eq:self-trans-wilde}
	H - \vec e_y \cdot \nu = 0,
	\end{equation}
where $H$ is the mean curvature of the surface considered and $\nu$ is the unit normal vector such that $\mathbf H = H \nu$. It is well known that a grim reaper  is a self-translating curve, so the cylinder over a grim reaper shifted by $(\tilde b, \tilde c)$,  
	\[
	\wilde\Gamma= \{ (x,y,z) \in (-\frac{\pi}{2} + \tilde b, \frac{\pi}{2} + \tilde b) \times \R^2 \mid (y - \tilde c) = - \log \cos(x -\tilde b)\} 
	\]
is an example in dimension $d=2$. With a slight abuse of language, we will also call $\wilde \Gamma$ a grim reaper throughout the article.


\subsection{Main result} 

Let us consider a finite family of grim reapers, $\{\wilde \Gamma_n\}_{n=1}^{N_{\Gamma}}$. Such a family is said to be in general position if no three $\wilde \Gamma$'s intersect on the same line, and no two $\wilde \Gamma$'s have the same asymptotic plane. For a family in general position,  let us denote 
by $\delta$ the minimum distance between two intersection lines and by $ \delta_{\Gamma}$ the minimum of the measure of the angles formed by the  $\wilde \Gamma$'s with the $yz$-plane at the intersection lines. Note that both $ \delta$ and $\delta_{\Gamma}$ are positive. 

To each intersection line $l_k$, we associate a positive integer $m_k$. The $m_k$'s allow us to take different scales at different  intersection lines, making the construction as general as possible. 

\begin{figure} [ht]
\includegraphics[height=2.5in]{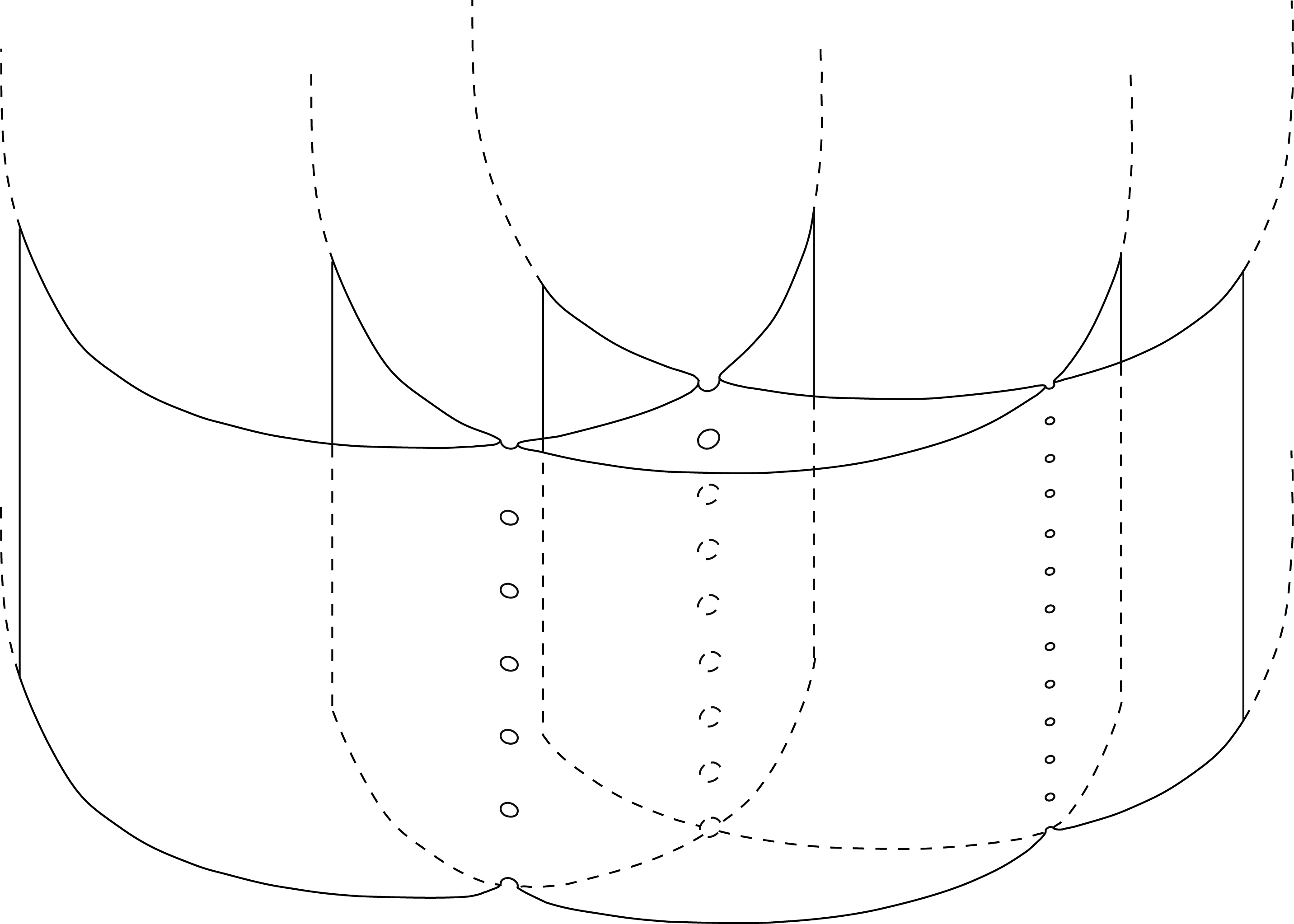}
\caption{A self-translating surface $\wilde{\mcal}_{\bar \tau}$} 
\end{figure}

\begin{theorem}
\label{thm:main}
Suppose $\{\wilde \Gamma_n\}_{n=1}^{N_{\Gamma}}$ is a finite family of grim reapers in general position. There is  a one parameter family of surfaces $\{ \wilde \mcal_{\bar \tau} \}_{\bar \tau \in (0,\delta_{\bar \tau})}$, with $\delta_{\bar \tau}$ depending on $\max_k(m_k)$, $N_{\Gamma}$, $\delta$ and $ \delta_{\Gamma}$ only, satisfying the following properties:
\begin{enumerate}
\item  $\wilde \mcal_{\bar \tau}$ is a complete embedded surface satisfying \eqref{eq:self-trans-wilde}.
\item  $\wilde \mcal_{\bar \tau}$ is invariant under reflection with respect to the $xy$-plane. 
\item  $\wilde \mcal_{\bar \tau}$ is singly periodic of period $2 \pi \bar \tau$ in the $z$-direction. 
\item If $U$ is a neighborhood in $\R^2$ such that $U \times \R$ contains no intersection line, then $\wilde \mcal_{\bar \tau} \cap (U \times \R)$ converges uniformly in $C^{j}$ norm, for any $j <\infty$, to $(\bigcup_{n=1}^{N_{ \Gamma}} \wilde \Gamma_n ) \cap (U \times \R)$ as $\bar \tau \to 0$.  
\item If $T_k$ is the translation that moves the $k$th intersection line to the $z$-axis, then $\bar \tau^{-1} T_k(\wilde \mcal_{\bar \tau})$ converges uniformly in $C^{j}$ norm, for any $j < \infty$, on any compact set of $\R^3$ to a Scherk surface with $m_k$ periods between $z=0$ and $z=2\pi$ as $\bar \tau \to 0$.
\end{enumerate}

\end{theorem}

\FloatBarrier

We call a  \emph{Scherk surface} any singly periodic minimal surface with four ends asymptotic to planes.     
\begin{figure} [h]
\includegraphics[height=2in]{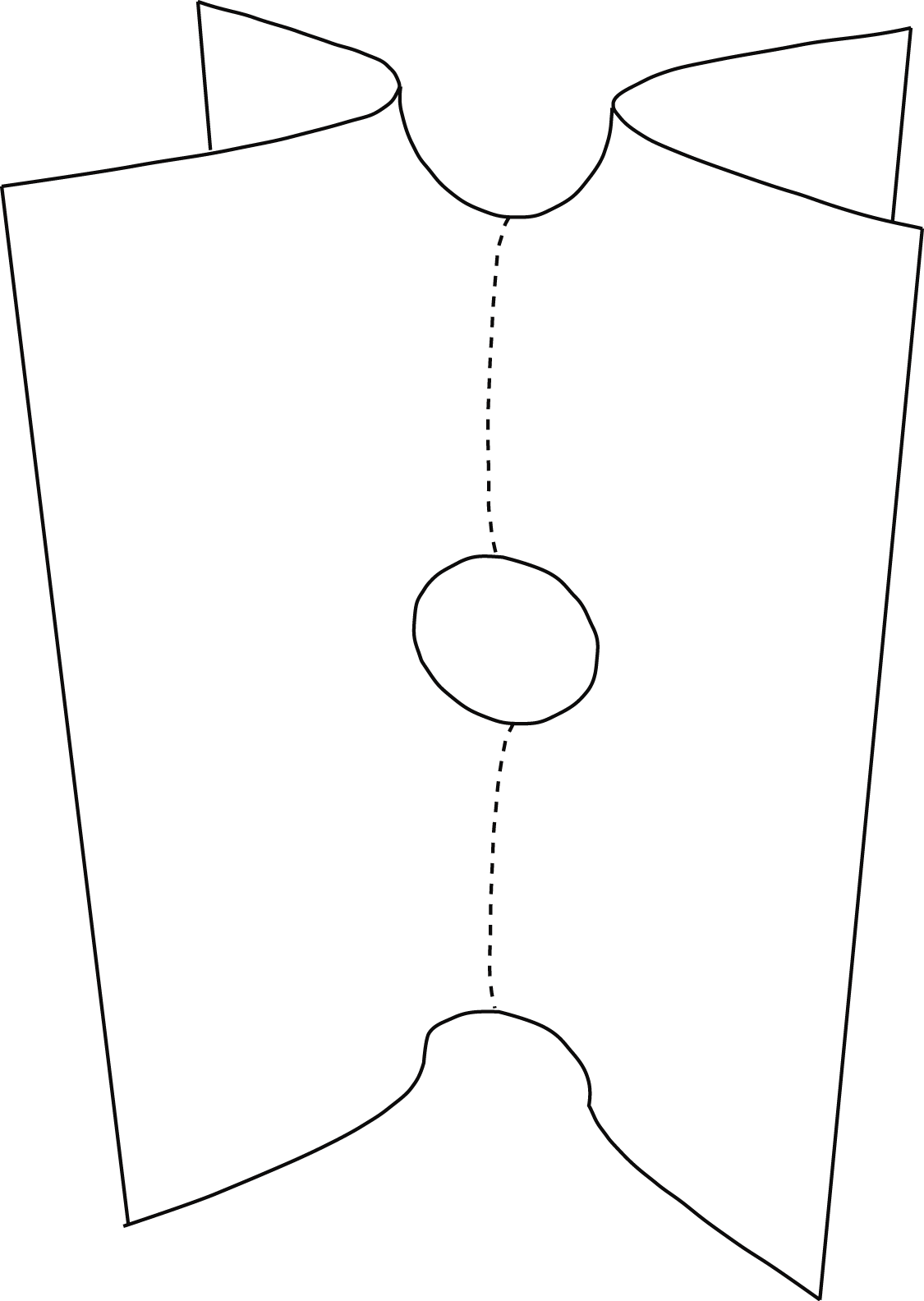}
\hspace*{3cm}
\includegraphics[height=2in]{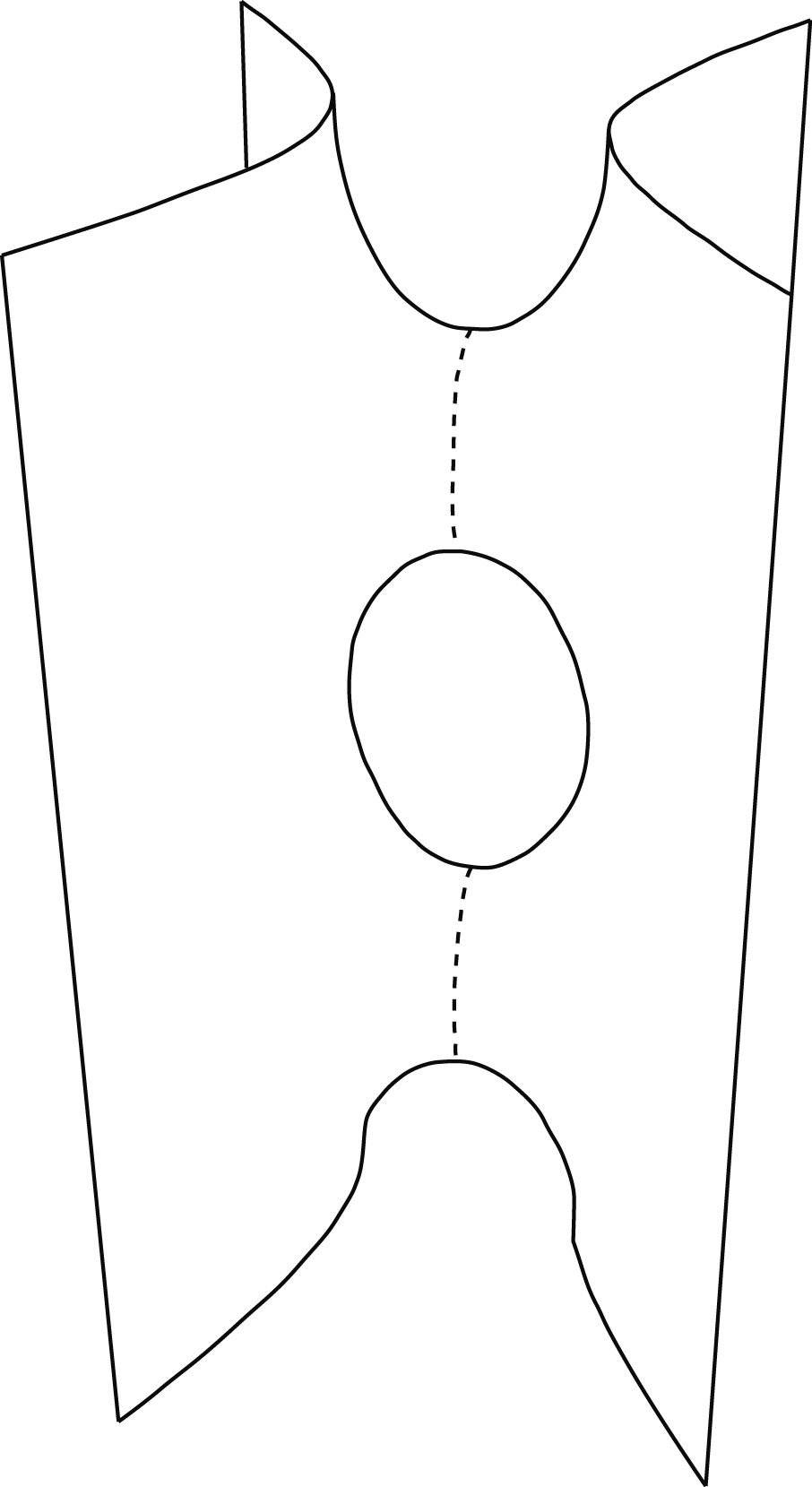}
\caption{Two Scherk surfaces.} 
\end{figure}
%
%
%
%
%

\FloatBarrier

\subsection{Sketch of the proof}

We use a version of \eqref{eq:self-trans-wilde} rescaled by a factor  $1/\bar \tau$, where $\bar \tau$ is a small constant to be determined:
	\begin{equation}
	\label{eq:self-translating}
	H-\bar \tau \vec e_y \cdot \nu = 0.
	\end{equation}

Our result and its proof are inspired by Kapouleas' construction of minimal surfaces by desingularizing a family of coaxial catenoids  \cite{kapouleas;embedded-minimal-surfaces}. Throughout the article, we treat equation \eqref{eq:self-translating} as a perturbation of $H=0$ and show that the term $-\bar \tau \vec e_y \cdot \nu$ can be controlled at every step. Since the intersections we desingularize are lines, as opposed to circles, the construction of the initial approximate solution is simplified. We can therefore  give explicit computations and more details in the proof overall.  Kapouleas mapped the catenoids conformally to cylinders to estimate and control the asymptotic behavior of the minimal surfaces. This step is not necessary in our case since our surfaces are singly periodic and asymptotic to planes. The study of the linear operator is not simpler however (but not more difficult either). As in \cite{kapouleas;embedded-minimal-surfaces}, we have to handle the presence of small eigenvalues for the linear operator and ensure exponential decay asymptotically.

The construction presented here can be easily adapted to desingularize the intersection of vertical planes to obtain minimal surfaces, provided no three planes intersect on the same line, and no two planes are parallel. Borrowing the idea from Kapouleas  \cite{kapouleas;embedded-minimal-surfaces}, we allow different scales at different intersections, so our result is a generalization of the construction of singly periodic minimal surfaces by Traizet \cite{traizet;surfaces-minimales}. 

Let us fix $\bar \tau>0$ and denote the rescaled grim reapers by 
	$
	\Gamma_n = {\bar \tau} ^{-1}\wilde \Gamma_n.
	$

In a first attempt to construct an initial surface, we replace the intersection lines with Scherk surfaces with asymptotic planes parallel to the tangent planes at the intersection, then use cut-off functions to obtain a smooth surface. The surface obtained is a crude first approximation, and the discussion below explains why the actual construction has to be more subtle. 

The initial surface is denoted by $M$, its position vector by $X$ and the unit normal vector with positive $\vec e_y$ component by $\nu$. We are looking for a solution to \eqref{eq:self-translating} among graphs of small functions $v$ with exponential decay over $M$, so we define $X_v = X + v \nu$, and denote by $M_v$ the graph of $v$ over $M$, by $H_v$ its mean curvature and by $\nu_v$ its unit normal vector. We have
	\[
	H_v - \tau \vec e_y \cdot \nu_v = H -\tau \vec e_y \cdot \nu + \Delta v + |A|^2 v + \tau \vec e_y \cdot \nabla v + Q_v,
	\]
where $Q_v$ is at least quadratic in $v$, $\nabla v$ and $\nabla^2 v$, and $A$ is the second fundamental form on $M$.  The surface $M_v$ is a STS if 
	\begin{equation}
	\label{eq:interm}
	\lin v = -H +\tau \vec e_y \cdot \nu - Q_v,
	\end{equation}
where $\lin v =  \Delta v + |A|^2 v + \tau \vec e_y \cdot \nabla v$. An important part of the proof is dedicated to solving the differential equation $\lin v =E$ on $M$. Once we can invert the linear operator $\lin$, we expect the quadratic term to be small  so the solution $v$ to \eqref{eq:interm} could be obtained by iteration. Since $M$ is a complicated surface, we study $\lin$ on different pieces first.

On an edge joining two desingularizing  surfaces, the Dirichlet problem $\lin v =E$ with vanishing boundary conditions has a unique solution by the standard elliptic theory.

On a desingularizing piece of Scherk surface $(\Sigma, g_{\Sigma})$, where $g_{\Sigma}$ is the metric induced by the embedding into $\R^3$,  $\lin$ is a perturbation of the linear operator $L = \Delta +|A|^2$ associated to normal perturbations of the mean curvature.  The mean curvature is invariant under translations, so the kernel of $L$ contains the functions $\vec e_x \cdot \nu$, $\vec e_y \cdot \nu$ and $\vec e_z \cdot \nu$, where $\vec e_x$, $\vec e_y$ and $\vec e_z$ are the three coordinate vectors. By imposing the symmetry with respect to the $xy$-plane, we can eliminate  $\vec e_z \cdot \nu$. The remaining functions $\vec e_x \cdot \nu$ and $\vec e_y \cdot \nu$ do not have the required exponential decay, nonetheless, they indicate the possible presence of small eigenvalues of $L$.  The \emph{approximate kernel} of $L$ is defined to be the span of all the eigenfunctions of $L$ corresponding to eigenvalues in $[-1,1]$. We introduce two linearly independent  functions $w_1$ and $w_2$ that have the important property of not being perpendicular to the approximate kernel (with respect to the inner product of $L^2(\Sigma, g_{\Sigma})$). Given a function $E$, one can  find constants $\theta_1$ and $\theta_2$ for which $E + \theta_1 w_1 + \theta_2 w_2$ is perpendicular to the approximate kernel and a function  $v$ satisfying $L v = E + \theta_1 w_1 + \theta_2 w_2$. Since $\lin$ is a perturbation of $L$, a similar result is true for $\lin$. 
For an exact solution, we must cancel any linear combination of $w_1$ and $w_2$ within the construction. This process is called \emph{unbalancing} and consists in dislocating the original Scherk surface so that the angles formed by its asymptotic planes are changed.

On a grim reaper end, we consider functions with exponential decay so that the asymptotic behavior of our solution matches our initial grim reaper closely. The difficulties do not arise from finding a solution to $\lin v =E $ on the end per se, but from the fact that, if we hope to find a global solution with exponential decay, the solutions on the desingularizing surfaces $\Sigma$ need to have exponential decay as well. 

Let us call {\em wings} of $\Sigma$ the four connected components obtained after removing a large enough cylinder from $\Sigma$. We can illustrate the behavior of a solution to the Dirichlet problem $\lin v =E$ in $\Sigma$, $v=0$ on $\pd \Sigma$ on a wing by considering the case of the Laplace operator on a cylinder. The surface is periodic in the $z$ variable, so the behavior can be reduced to one coordinate, say $s$. Suppose we have a function $E$ with exponential decay and a solution $v$ to $v''(s) = E(s)$ that vanishes at the boundary $s=s_0$. Explicitly, the function $v$ is given by 
	\[
	v(s)  = \int_{s_0}^s \int_{s_0}^t E(r) dr dt + v'(s_0) (s-s_0), \quad s,t \leq s_0.
	\]
The first term on the right hand side has the right decay, but the second term is linear. From the point of view of the construction, we can cancel the term $v'(s_0) (s-s_0)$ by modifying the slope of $\Sigma$ at the boundary of each wing. This involves bending each wing separately by an amount $\{\varphi_i\}_{i=1}^4$ to generate new functions $\{\bar w_i\}_{i=1}^4$ to achieve exponential decay along the wings. 

We construct a global solution for $\lin v=E$ on $M$ by partitioning the support of the inhomogeneous term and solving on each piece inductively. The error generated is small thanks to the exponential decay, and a contraction principle gives us convergence. Finally, we use a fixed point theorem to find  a self-translating surface.

\subsection{Outline of the article}

In Section \ref{sec:initial-conf}, we study families of grim reapers and how the unbalancing at the intersection affects the initial configuration. 

In Section \ref{sec:desingularizing-surfaces}, we turn our attention to Scherk surfaces and construct the desingularizing surfaces. We carefully study the interaction between the bending of the wings and the unbalancing, then  introduce the functions $w$'s and $\bar w$'s and establish key estimates in Section \ref{sec:estim}.

In Section \ref{sec:initial-surfaces}, we describe how to replace the lines of intersection by the desingularizing surfaces and construct  smooth initial surfaces. 

Section \ref{sec:lin-op} is dedicated to the study of the linear operator, on each piece separately, then globally. In Sections \ref{sec:quadratic} and \ref{sec:fpt}, we finish the proof by estimating the quadratic term and applying the Schauder Fixed Point Theorem.

Some sections of this article follow the exposition of Kapouleas in \cite{kapouleas;embedded-minimal-surfaces}, and we try to use the same notations, whenever possible. 

\subsection{Notations} 
\label{ssec:notations}
	\begin{itemize}
	\item $E^3$ is the Euclidean three space equipped with the usual metric.
	\item $\vec e_x$, $\vec e_y$ and $\vec e_z$ are the three coordinate vectors of $E^3$.
	\item $\S^2$ refers to the standard unit sphere of dimension $2$. 
	\item Throughout this article, a surface with a tilde $\wilde{S}$ is a surface in the original scale, with curvature comparable to $1$. We use a notation without tilde for its rescaling $S = \frac{1}{\tau} \wilde S=\{ (x,y,z) \in E^3 \mid (\tau x, \tau y, \tau z) \in \wilde S\}$, with curvature comparable to $\tau$. 

	\item We fix once and for all a smooth cut-off function $\psi$ which is increasing, vanishes on $(-\infty, 1/3)$ and is equal to $1$ on $(2/3, \infty)$. We define the functions $\psi[a,b]: \R \to [0,1]$ which transition from $0$ at $a$ to $1$ at $b$ by 
	\[
	\psi[a,b] (s) = \psi \left(\frac{s-a}{b-a}\right).
	\]
	
	\item We often have a function $s$ defined on the surfaces with values in $\R \cup \{ \infty\}$. If $V$ is a subset of such a surface, we use the notations 
	\begin{equation}
	\label{notation-s}
	V_{\leq a} := \{ p \in V: s(p) \leq a \}, \quad V_{\geq a}:= \{ p \in V :s(p) \geq a\}.
	\end{equation}
	\item $\nu, g, A$, and $H$ denote respectively the oriented unit normal vector, the induced metric, the second fundamental form, and the mean curvature of an immersed surface in the Euclidean space $E^3$. When we want to emphasize the surface $S$, we write these quantities with a subscript, for example $g_S$ denotes the metric of $S$. 
	\item Given a surface $S$ in $E^3$, which is immersed by $X: S \to E^3$, and a $C^1$ function ${\sigma}: S \to \R$, we call the graph of ${\sigma}$ over $S$ the surface given by the immersion $X+{\sigma} \nu$ and denote it by $S_{\sigma}$. We often use $X+{\sigma} \nu$ and its inverse to define projections from $S$ to $S_{\sigma}$, or from $S_{\sigma}$ to $S$ respectively. When we refer to projections from $S$ to $S_{\sigma}$ or from $S_{\sigma}$ to $S$, we always mean these projections. 
	 \item	We work with the following weighted H\"older norms:
	\begin{equation}
	\label{eq:def-weighted-holder}
	\| \phi: C^{k, \alpha}(\Omega, g, f)\| := \sup_{x\in \Omega} f^{-1}(x) \| \phi:C^{k, \alpha}(\Omega\cap B(x), g)\|,
	\end{equation}
where $\Omega$ is a domain, $g$ is the metric with respect to which we take the $C^{k, \alpha}$ norm, $f$ is the weight function, and $B(x)$ is the geodesic ball centered at $x$ of radius $1$.
\end{itemize}

\subsection{Thanks} The author would like to thank Sigurd Angenent for his encouragement. 


\section{Grim Reapers and Initial Configuration}
\label{sec:initial-conf}


We discuss dislocations at the intersection of the grim reapers, and how these affect the tangent vectors and the position of the grim reapers.

\subsection{Grim reapers}
\label{ssec:grimreapers}

A grim reaper  $\wilde \Gamma_n$ is a self-translating solution to the mean curvature flow given by the equation
	\[
	 y = - \log(\cos  (x - \tilde b_n)) +   \tilde c_n, \qquad   |x-\tilde b_n| < \pi/2,
	\]
or, in arc length parametrization, 
	\begin{equation}
	\label{eq:arclength}
	(\gamma_1(s) + \tilde b_n, \gamma_2(s) + \tilde c_n) :=(\arctan(\sinh s)+\tilde b_n, \ln(\cosh s)+\tilde c_n, ), \quad s \in \R.
	\end{equation}
Since we will be using the word vertex in another context below, we call the point $(\tilde b_n, \tilde c_n)$ the \emph{center} of the grim reaper to avoid confusion. With a slight abuse of language, we call any surface that is a rescaling of $\wilde \Gamma$, or a cylinder over $\wilde \Gamma$, a grim reaper also. The cylinders $\wilde \Gamma_n \times \R$ ($\Gamma_n \times \R$) will be denoted $\wilde \Gamma_n$ ($\Gamma_n$ resp.) as well in this article. In this section however, we work in the $xy$-plane exclusively since the $z$-coordinate does not play any role. 

	\begin{definition} 
	\label{def:conditions}
We say that a finite family of grim reapers  $\{\wilde  \Gamma_n\}_{n = 1}^{N_{\Gamma}}$ is in general position if  it satisfies the following conditions
	\begin{enumerate}
	\item For some $\eps>0$, $ |\tilde b_n - \tilde b_m - k \pi | > \eps$ for $n\neq m$ and for all integers $k$.
	\item No three grim reapers intersect in one point. 
	\end{enumerate}
	\end{definition}
The first condition ensures that two grim reapers do not share the same asymptote, therefore any two grim reapers that intersect do so transversally. Moreover, we have a lower and upper bound on the angle of intersection. 

Since the result would be trivial otherwise, we assume that at least two grim reapers intersect. The two conditions (i) and (ii) above imply the following properties:

\begin{lemma}
\label{lem:cor-cond}
Suppose $\{\wilde  \Gamma_n\}_{n = 1}^{N_{\Gamma}}$ is a family of grim reapers in general position. There are two real numbers $\delta_{\Gamma}>0$ and  $\delta>0$ such that
	\begin{enumerate}
	\item  the measure of the smallest of the two angles at the intersection of any two grim reapers is between $30 \delta_{\Gamma}$ and $ \pi/2 -30\delta_{\Gamma}$. 
	\item Any tangent vector to a grim reaper at an intersection forms an angle greater than $30 \delta_{\Gamma}$ with  $\vec e_y$. 
	\item the arc length or distance on the grim reapers between any two intersection points is greater than $2\delta$. 
	\end{enumerate}
\end{lemma}
\begin{proof}
The first result is immediate from the explicit formula for the grim reaper and Definition \ref{def:conditions}. We have the second  and third properties since the number of grim reapers and therefore the number of intersection points is finite. 
\end{proof}	

We now work on a larger scale: we fix $\tau$ a small positive constant and consider $\Gamma_n := \frac{1}{\tau} \wilde \Gamma_n$ with center $(b_n, c_n):= \frac{1}{\tau}(\tilde b_n, \tilde c_n)$. Let  $G$ be the graph of $\bigcup_{n=1}^{N_{\Gamma}} \Gamma_n$ and  
	\begin{itemize}
	\item $V(G)$ the set of vertices of $ G$ (intersections of two grim reapers),
	\item $E( G)$ the set of edges of $ G$ (pieces of grim reapers connecting two intersections),
	\item $R( G)$ the set of rays of $ G$ (ends of grim reapers starting at an intersection),
	\item $R_l( G)$ the set of ``left" rays of $G$ (rays that start at an intersection and move toward the negative $x$-coordinate).
	\end{itemize}

Unless otherwise specified, $p$ denotes a vertex, $e$ an edge and $r$ a ray. For each vertex $p$, there are exactly four unit tangent vectors emanating from $p$, denoted by $v_{p1}, v_{p2}, v_{p3}, v_{p4}$, where the number refers to the order in which they appear as we rotate from $\vec e_y$ counterclockwise. We say a graph is {\em balanced} if the four directing vectors cancel, $\sum_{i=1}^4 v_{pi} =0$, at every vertex $p$. Note that the graph  $G$  is balanced.

\subsection{Construction of the initial configuration $\overline G$}
\label{ssec:construction-G-bar}

Our goal is to perturb $G$ into a graph $\overline G$ for which the sum of the directing vectors at each vertex $p$  is given by a small vector $\vec \zeta(p)$,
	\[
	\sum_{i=1}^4 v_{pi} =\vec \zeta(p), \quad p \in V(\overline G).
	\] 
The process is called {\em unbalancing} and is necessary for tackling the approximate kernel of the linear operator $L=\Delta +|A|^2$ on each desingularizing surface.

The equation \eqref{eq:self-translating} we set out to solve  is a perturbation of $H=0$. The mean curvature is invariant under translations, therefore the functions $\vec e_x \cdot \nu$, $\vec e_y \cdot \nu$ and $\vec e_z \cdot \nu$ are in the kernel of the linear operator $L$ associated to normal perturbations of $H$.  Imposing  a symmetry (invariance of our surfaces with respect to reflection across the $xy$-plane), we can rule out $\vec e_z \cdot \nu$. The remaining functions $\vec e_x \cdot \nu$ and $\vec e_y \cdot \nu$ do not have the required exponential decay, however, they indicate  that $L$ has small eigenvalues.  We call the span of eigenfunctions of $L$ corresponding to these small eigenvalues the {\em approximate kernel} of $L$.  One can only solve the differential equation $L v =E $ with a reasonable estimate on $v$ if $E$ is perpendicular to the approximate kernel. We do not have such control over the inhomogeneous term, so we  introduce two functions $w_1$ and $w_2$ to cancel any component  parallel to the approximate kernel. Roughly speaking, $w_1$ has to be in the direction of $\vec e_x \cdot \nu$, in the sense that $\int w_1 (\vec e_x \cdot \nu) \neq 0$, and similarly, $\int w_2 (\vec e_y \cdot \nu) \neq 0$.

Let $S$ be one period of the desingularizing surface $\Sigma$. According to the balancing formula from \cite{korevaar-kusner-solomon} (see also Lemma \ref{lem:balancing}), the mean curvature of $S$ satisfies
	\[
	\int_{S}  H  \vec e_x \cdot \nu dg_{S} = 2 \pi \sum_{i=1}^4 v_i \cdot \vec e_x, \quad \int_{S}  H  \vec e_y  \cdot \nu dg_{S} =  2\pi \sum_{i=1}^4 v_i \cdot \vec e_y, 
	\]
where each vector $v_i$ is the direction of the  half-plane asymptotic to the $i$th end of $\Sigma$. The idea is to define $w_1$ and $w_2$ as derivatives of $H$ and use dislocations that move  $v_1$ and $v_3$, or $v_2$ and $v_4$, away from being parallel to generate linear combinations of $w_1$ and $w_2$. The role of the new graph $\overline G$ is to determine the position of the vertices of the initial surface depending on an imposed unbalancing.

We define the notion of a tetrad to keep track of the different vectors and angles at the vertices. 
We use the notations 
	\begin{equation}
	\label{def:e-e-prime}
	\vec e[\theta] = \cos \theta \vec e_x + \sin \theta \vec e_y, \quad 
	\vec e\ ' [\theta] =-\sin \theta \vec e_x + \cos \theta \vec e_y.
	\end{equation}

	\begin{definition}
	\label{def:tetrad}	
An acceptable tetrad of vectors, or  ``tetrad" for short, is defined to be a tetrad of vectors $T = (v_1, v_2, v_3, v_4)$ such that $v_i = \vec e[\beta_i]$ for some $\beta_i \in \R$ satisfying
	\[
	0 < \beta_2 - \beta_1 < \beta_3 - \beta_1< \beta_4 - \beta_1<2 \pi.
	\]
Moreover, for such a tetrad $T$ we define 
	\begin{align*}
	\theta(T)&: = \frac{ \beta_1-\beta_2 + \beta_3 - \beta_4+2\pi}{4}, &\theta_1(T):= \frac{\beta_3 - \beta_1-\pi}{2},\\
	\theta_r(T)&:=\frac{\beta_1 + \beta_2 + \beta_3 +\beta_4-4\pi}{4}, &\theta_2(T):= \frac{\beta_4 - \beta_2-\pi}{2},
	\end{align*}
and require 
	\begin{equation}
	\label{eq:require-T}
	\theta(T) \in [20 \delta_{\theta}, \frac{\pi}{2} - 20 \delta_{\theta}], \quad \theta_1(T), \theta_2(T) \in [-2 \delta_{\theta}, 2 \delta_{\theta}],
	\end{equation}
with $\delta_{\theta}$ to be determined at the end of the section.
		\end{definition}
\begin{figure} [h]
\includegraphics[height=1.8in]{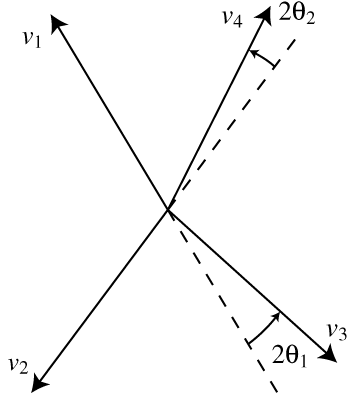}
\caption{The angles $\theta_1$ and $\theta_2$.} 
\end{figure}
We characterize the unbalancing of a tetrad $(v_1, v_2, v_3, v_4)$ with the angles $\theta_{j}$, $j=1,2$, which measure how much the vectors $v_j$ and $v_{j+2}$ fail to point in opposite directions, instead of using the vector $\vec \zeta = \sum_{i=1}^4 v_i$. The angles $\theta(T)$ and $\theta_r(T)$ will not be used here but will come into play in Section \ref{sec:desingularizing-surfaces}.

\FloatBarrier

Let $N_I$ be the number of intersection points, and recall that $N_{\Gamma}$ is the number of grim reapers. The unbalancing restrictions do not determine $\overline G$ completely, and we have enough degrees of freedom to impose a slight shift of each left ray.  The graph $\overline G$ constructed below is therefore a perturbation of $G$ depending on $\{(\theta_{k,1}, \theta_{k,2})\}_{k=1}^{N_I}$ for the unbalancing and on $\{(b_n', c_n')\}_{n=1}^{N_{\Gamma}}$ for the position of the left rays.

The general construction does not depend on the first simpler two cases, but we describe them so the reader can understand the general case more easily. 

1) We first consider the case of the intersection of two grim reapers $\Gamma_1$ and  $\Gamma_2$, where $\Gamma_1$ is the grim reaper that has a tangent at $p$ in the $v_{p1}$ direction.  We translate $\Gamma_1$ to a grim reaper whose center is at $(b_1+b'_1, c_1 +c'_1)$ and similarly, we translate $\Gamma_2$ to one centered at $(b_2+ b'_2, c_2+ c'_2)$. We denote by $\bar p$ the intersection of the two perturbed grim reapers, and  the four directing vectors $\{ v_{\bar p i} \}$ by $v_1, v_2, v_3$ and $v_4$, where the number refers to the order in which they appear as we rotate from $\vec e_y$ counterclockwise. We fix $\bar v_1=v_1$ and $\bar v_2=v_2$, and we find $\bar v_3$ and $\bar v_4$ so that 
 	\[
	\theta_1(\overline T) = \theta_{1,1}, \textrm{ and } \theta_2(\overline T)  = \theta_{1,2}, \textrm{ for } \overline T=(\bar v_1, \bar v_2, \bar v_3, \bar v_4)
	\]
We now join at $\bar p$ two pieces of grim reapers with tangent vectors $\bar v_3$ and  $\bar v_4$ respectively. Note that this joining is possible since by equation \eqref{eq:arclength}, the tangent unit vectors to a grim reaper map onto the set of unit vectors in $\R^2$ with positive $x$-coordinate.  This completes the construction of an initial configuration for the simplest case.
\begin{figure} [ht]
\includegraphics[height=2.15in]{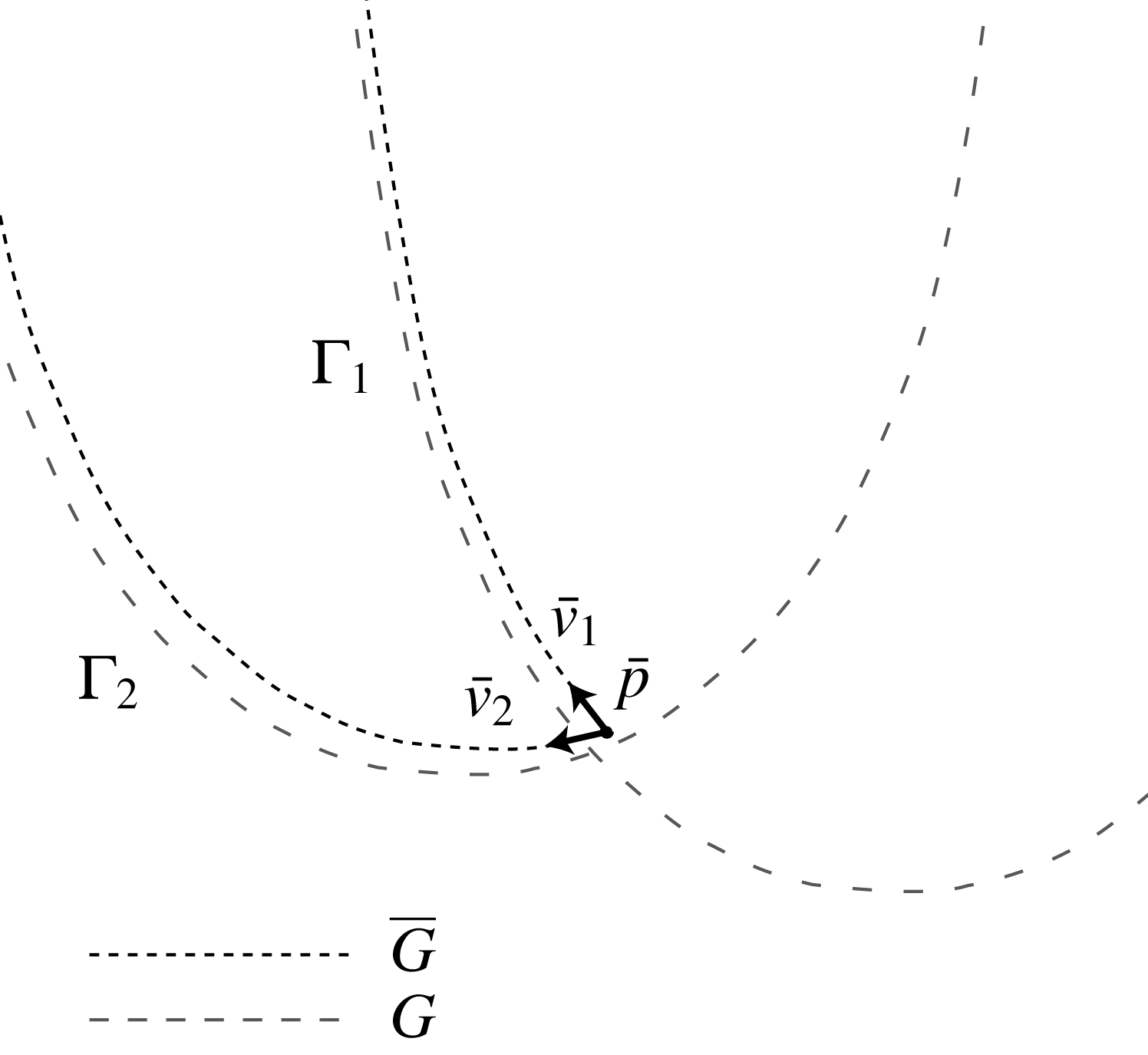}
\hspace*{0cm}
\includegraphics[height=2.15in]{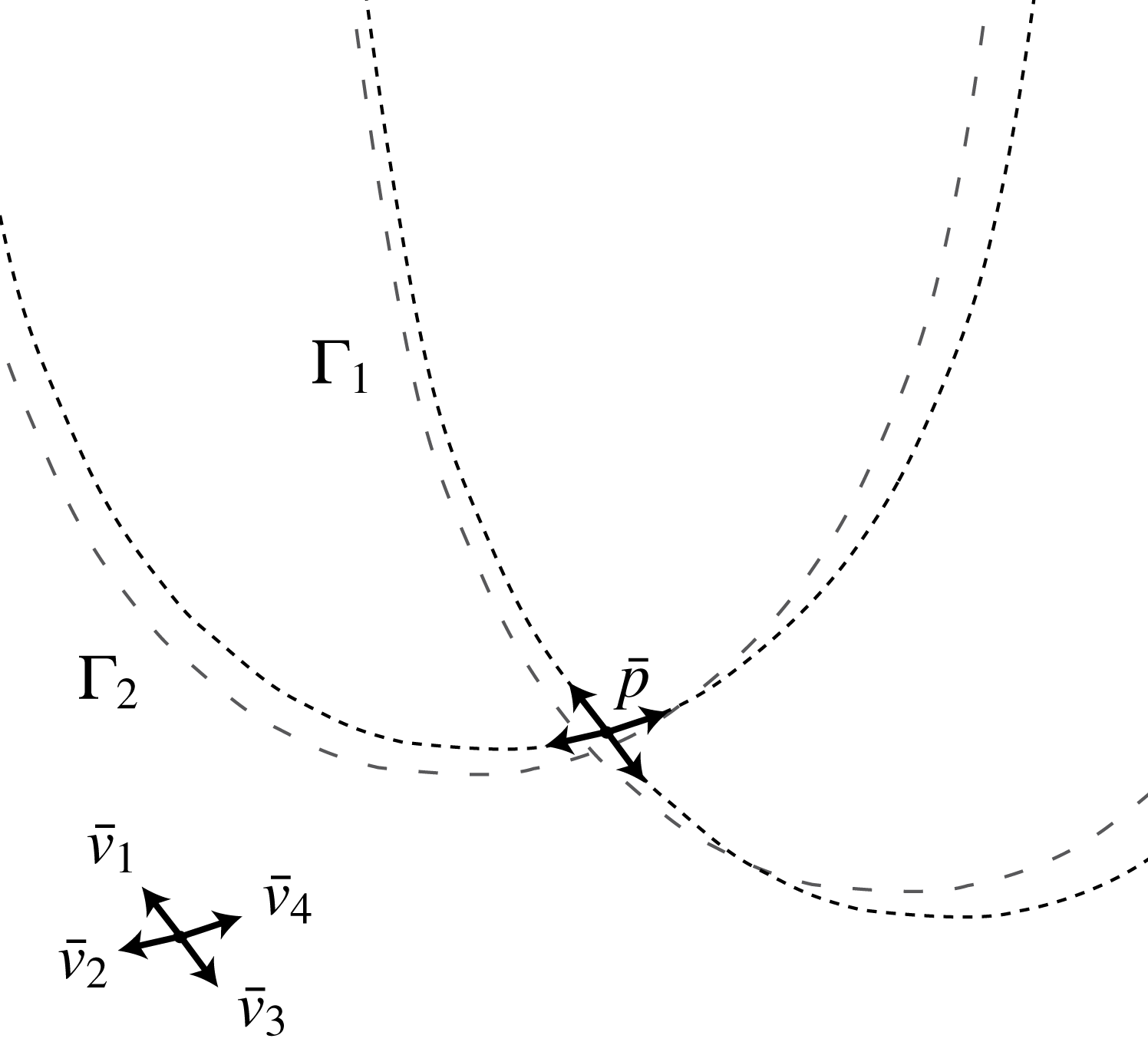}
\caption{Construction of $\overline G$ for two intersecting grim reapers.} 
\end{figure}

2) The second case consists of three grim reapers intersecting in three points $p_k=(x_k, y_k)$, $k=1,2,3$. We number the points so that $x_j < x_k $ for $j<k$, and we number the grim reapers so that $\Gamma_1$ has tangent at $p_1$ in the $v_{p_11}$ direction, $\Gamma_2$ is the other grim reaper through $p_1$, and $\Gamma_3$ is the last grim reaper.

As in the previous case, we translate $\Gamma_1$ and $\Gamma_2$ by $(b'_1, c'_1)$ and $(b'_2, c'_2)$ respectively.  We denote the intersection of the perturbed curves by $\bar p_1$ and the four directing vectors  by $v_{\bar p_11}, v_{\bar p_12}, v_{\bar p_13}$ and $v_{\bar p_14}$ in the same order as before. We impose $\bar v_{11} = v_{\bar p_11}$ and $\bar v_{12} = v_{\bar p_12}$, and find $\bar v_{13}$ and $\bar v_{14}$ such that $\theta_1(\overline T) = \theta_{1,1}, \textrm{ and } \theta_2(\overline T)  = \theta_{1,2},$ for $\overline T=(\bar v_{11}, \bar v_{12}, \bar v_{13}, \bar v_{14})$.

At the point $\bar p_1$, we consider the two rays (which are pieces of  grim reapers) emanating from $\bar p_1$ with tangent directions $\bar v_{13}$ and $\bar v_{14}$ respectively. Both of these rays intersect the translation of the third grim reaper by $(b'_3, c'_3)$; $\bar p_2$ is the leftmost intersection point. We impose $\bar v_{21} = v_{\bar p_2 1}$ and $\bar v_{22} = v_{\bar p_2 2}$. The two other directions $\bar v_{23}$ and $\bar v_{24}$ are determined by $(\theta_{2,1},\theta_{2,2})$.

We now take  the two rays emanating from $\bar p_2$ with tangent directions $\bar v_{23}$ and $\bar v_{24}$ respectively. One of these intersects one of the rays from $\bar p_1$ and we denote the intersection by $\bar p_3$. 
 
We fix the two left directions $\bar v_{31}=v_{\bar p_3 1}$ and $\bar v_{32} = v_{\bar p_3 2}$ and the two other directions $\bar v_{33}$ and $\bar v_{34}$ are determined by $(\theta_{3,1}, \theta_{3,2})$. We now take the two rays emanating from $\bar p_3$ with tangent directions $\bar v_{33}$ and $\bar v_{34}$ to complete the construction.
\FloatBarrier

3) For the general case, let us denote the intersection points $p_k = (x_k, y_k)$, $k=1, \ldots, N_I$ where $x_j \leq x_k$ for $j\leq k$ and if $x_j = x_k$, $y_j>y_k$ for $j<k$. In other words, we number our intersection points from left to right, and if two points have the same abscissa, we take the one with higher ordinate first. 

We now number the grim reapers. $\Gamma_1$ is the grim reaper through $p_1$ with tangent direction $v_{p_11}$, and $\Gamma_2$ is the grim reaper through $p_1$ with tangent in the $v_{p_12}$ direction. Note that the abscissa of the center of $\Gamma_1$ is greater than the abscissa of the center of $\Gamma_2$, in other words, $b_1 >b_2$. We then proceed with the $p_k$'s by increasing $k$. If $p_k$ is on an as yet unnumbered grim reaper, we just give the grim reaper the next available number. In the case $p_k$ is on two unnumbered grim reapers, we number the one with the rightmost center first.

We translate $\Gamma_1$ by $(b'_1, c'_1)$ and $\Gamma_2$ by $(b'_2, c'_2)$ respectively, and denote the intersection by $\bar p_1$. We impose $\bar v_{11} =v_{\bar p_1 1}$ and $\bar v_{12} = v_{\bar p_1 2}$ and determine the directions $\bar v_{13}$ and $\bar v_{14}$ using  $(\theta_{1,1}, \theta_{1,2})$. We now modify the edges or rays on the right of $\bar p_1$ so that they have unit tangent vectors $\bar v_{13}$ and $\bar v_{14}$ respectively.  

We suppose the point $\bar p_{k-1}$ and the vectors $\{\bar v_{(k-1)i}\}_{i=1}^4$ are constructed and give a procedure for the $k$th intersection. Consider the two edges or rays on the left of $p_{k}$ and intersecting at $p_{k}$. If both of the pieces on the left of $p_{k}$ are edges, they have been modified already and their intersection gives us $\bar p_{k}$. If we have one edge and one ray, or two rays, then we modify the ray(s) in the following way. Each ray is supported on a grim reaper $\Gamma_n$. We translate each $\Gamma_n$ by $(b_n',c_n')$ to obtain a modified ray. The intersection of the edge and the modified ray, or of the two modified rays, gives us $\bar p_{k}$. We impose the two left vectors $\bar v_{k1} =v_{\bar p_k 1}$ and $\bar v_{k2} = v_{\bar p_k 2}$ and determine $\bar v_{k3}$ and $\bar v_{k4}$ from $(\theta_{k,1}, \theta_{k,2})$. We now modify the edges or rays on the right of $\bar p_{k}$ so that their unit tangent vectors are $\bar v_{k3}$ and $\bar v_{k4}$ respectively. 
\begin{figure} [h]
\includegraphics[height=3in]{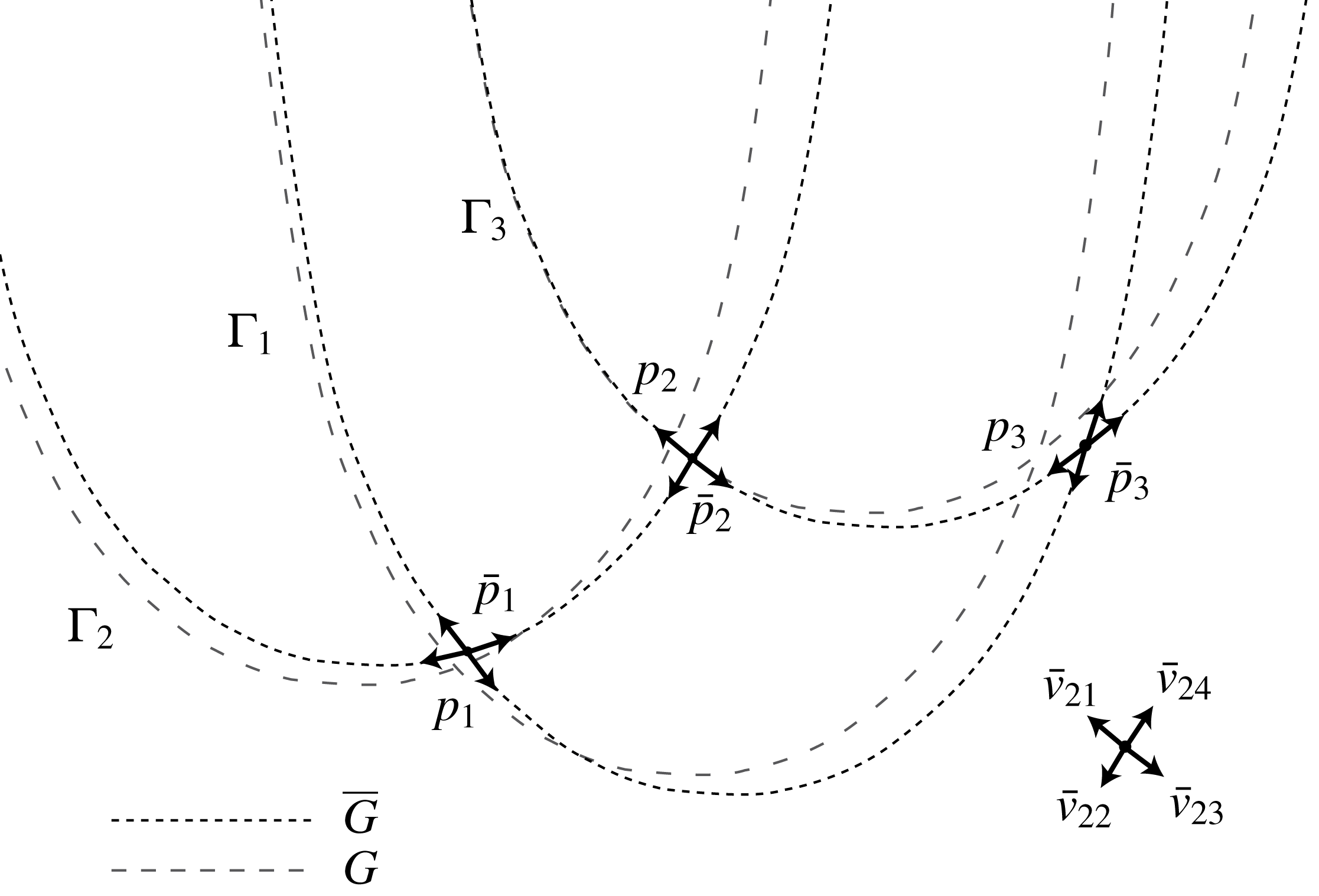}
\caption{Construction of $\overline G$ for three grim reapers, with $(b'_3, c'_3)=(0,0)$.} 
\end{figure}
\FloatBarrier

Let us fix some notations for the new graph $\overline G$. To each vertex $\bar p_k$ of  $\overline G$ corresponds the tetrad $(\bar v_{k1}, \bar v_{k2}, \bar v_{k3}, \bar v_{k4})$ of unit tangent vectors of rays or edges emanating from $\bar p_k$. Similarly, we have the tetrad $(v_{k1}, v_{k2}, v_{k3}, v_{k4})$ at each vertex $p_k$ of the original graph $G$. When we want to emphasize the dependence of $\overline G$ on its parameters, we write $\overline G (\ubar b',\ubar c', \ubar \theta, \tau)$, where $\ubar b' = \{b'_n\}_{n=1}^{N_{\Gamma}}$, $\ubar c' = \{c'_n\}_{n=1}^{N_{\Gamma}}$, $\ubar \theta=\{\theta_{k,1}, \theta_{k,2}\}_{k=1}^{N_I}$, and $\tau$ is related to the scaling.

\begin{prop}
\label{prop:init-conf} There exist a $\delta'_{\theta}>0$ and a constant $C$ depending only on  $\delta_{\Gamma}, \delta$ and $N_{\Gamma}$ with the following property: if  $\max(\tau |(\ubar b',\ubar c')| , | \ubar \theta|) \leq  \delta'_{\theta}$, the vertices $\bar p_k$ and the directing vectors $\bar v_{ki}$ of the graph $\overline G (\ubar b',\ubar c', \ubar \theta, \tau)$ satisfy
	\[
	\tau |p_k-\bar p_k| \leq C\tau |(\ubar b',\ubar c')| + C | \ubar \theta|, \quad |\angle (v_{ki}, \bar v_{ki})| \leq C \tau |(\ubar b',\ubar c')| + C |\ubar \theta|,
	\]
for $i=1,\ldots, 4$ and $ k=1,\ldots, N_I$. 

\end{prop}

\begin{proof} We start by studying how the location of the intersection changes as two grim reapers are translated by different vectors. Without loss of generality, we can assume that one of them stays centered at the origin. 

Let $\Gamma_0$ and $\Gamma_1$ be two grim reapers given by the position vectors
	\begin{gather*} 
	\mathbf r_0 (s) =  \frac{1}{\tau}(\gamma_1(\tau s) , \gamma_2(\tau s)),\\
	\mathbf r_1 (t) =  \frac{1}{\tau}(\gamma_1(\tau t) + \gamma_1(\tau s_0) - \gamma_1(\tau s_1), \gamma_2(\tau t) + \gamma_2(\tau s_0) - \gamma_2(\tau s_1),
	\end{gather*}
where 
	\begin{equation}
	\label{eq:def-gamma}
	\gamma_1 (s) = \arctan(\sinh s), \quad \gamma_2(s) = \ln (\cosh s),
	\end{equation} 
as in \eqref{eq:arclength}. The two grim reapers intersect at $\mathbf r_0 (s_0) = \mathbf r_1(s_1)$ and the tangent vectors at the intersection are 
	\begin{equation}
	\label{eq:tangent}
	\mathbf r_i'(s_i) = \left(\frac{1}{\cosh(\tau s_i)}, \tanh(\tau s_i) \right), \quad i=0,1.
	\end{equation}
The coordinates of the center of  $\Gamma_1$ are 
	\begin{align}
	\label{eq:F-center}
	b=\tau^{-1}(\gamma_1(\tau s_0) - \gamma_1(\tau s_1)), \quad c=\tau^{-1} (\gamma_2(\tau s_0) - \gamma_2(\tau s_1)).
	\end{align}
If the angles between the tangent vectors $\mathbf r_i'(s_i)$  and $\vec e_y$ are larger than $20 \delta_{\Gamma}$,  then  $|\tau s_0| +|\tau s_1|<C$, with $C$ depending on $\delta_{\Gamma}$. Therefore $\cosh(\tau s_0)$ and $\cosh (\tau s_1)$ are bounded above by a constant. In addition, if $|\tau b| =|\gamma_1(\tau s_0) - \gamma_1(\tau s_1)| >\eps/2$, where $\eps$ is given in Definition \ref{def:conditions}, we have
	\[
	 \frac{\eps}{2}<\left|\frac{1}{\cosh (\tau \bar s)} (\tau s_0-\tau s_1) \right| \leq \tau |s_0-s_1|, 
	 \]
where $\bar s$ is the number in $(s_0,s_1)$ given by the Mean Value Theorem.

We define the function $F: \R^2 \to \R^2$, $F(s_0, s_1)= (b,c)$ using \eqref{eq:F-center}.  According to the Inverse Function Theorem, $F$ has an inverse if the determinant of its Jacobian does not vanish. Indeed,
	\[
	|\det [DF]| = \frac{|\sinh (\tau s_1)- \sinh (\tau  s_0) |}{\cosh (\tau s_0)  \cosh (\tau s_1)} \geq  \frac{\tau |s_1-s_0|}{\cosh (\tau s_0) \cosh (\tau s_1)} 
	\geq   \frac{\eps}{2 C^2}>0.
	\]
$F^{-1}$ has  bounded derivatives and the quantities $\tau s_0$ and $\tau s_1$ are bounded, therefore, if $p$ and $\bar p$ are the intersections of $\Gamma_0$ and  grim reapers centered at $(b,c)$ and $(\bar b, \bar c)$ respectively, we have 
	\begin{equation}
	\label{eq:delta-p}
	|p-\bar p| \leq C |(b-\bar b, c-\bar c)|
	\end{equation}
for $(b,c)$ close enough to $(\bar b,\bar c)$.

Let us now fix the intersection point $\tau^{-1}(\gamma_1(\tau s_0), \gamma_2(\tau s_0))$ and study how a change in the tangent vector $ \mathbf r_1'(s_1)$ at the intersection moves the center $(b,c)$. The angle $\alpha$ between  $ \mathbf r_1'(s_1)$ and $\vec e_y$ satisfies
	\begin{equation}
	\label{eq:angle-arclength}
	\tan \alpha=  \sinh (\tau s_1).
	\end{equation}
From an earlier discussion, $|\tau s_0|+|\tau s_1|$ is bounded, so $|\sinh (\tau s_1)|$ is bounded. Hence, 	\begin{equation}
	\label{eq:delta-angle}
	C^{-1} \tau \leq  \left| \frac{d\alpha}{ds_1}\right| = \left| \frac{\tau}{\cosh (\tau s_1)}\right| \leq C \tau 
	\end{equation}
for some constant $C$. 

Starting with our initial configuration where all the angles are bounded below by $30 \delta_{\Gamma}$, the change of position of an intersection point and the changes in the tangent vectors are propagated to the next intersection points, but for  $\tau |(\ubar b', \ubar c')|$ and $|\ubar \theta|$ small enough depending on $N_{\Gamma}$, $\delta$ and $\delta_{\Gamma}$, the perturbed configuration still has  the properties of Lemma \ref{lem:cor-cond}, with $20 \delta_{\Gamma}$ instead of $30\delta_{\Gamma}$ in (i) and (ii), and $\delta$ instead on of $2 \delta$ in (iii). The result follows from \eqref{eq:delta-p} and \eqref{eq:delta-angle} for $\tau |(\ubar b',\ubar c')| $ and $  | \ubar \theta|$ small enough.
\end{proof}

\begin{definition}
\label{def:delta-theta}
We fix $\delta_{\theta} =\min(C\delta'_{\theta}, \delta_{\Gamma})$ for the rest of the article, where $C$ and $\delta'_{\theta}$ are as in the previous proposition. 
\end{definition}


\section{Desingularizing Surfaces}
\label{sec:desingularizing-surfaces}


We now construct the surfaces that will replace the lines of intersection of  the grim reapers. As mentioned in the introduction, we have to allow some flexibility so that opposite wings can fail to have opposite directions  (unbalancing) and for each wing to be bent independently further along (bending).


\subsection{Scherk surfaces}


The one parameter family of Scherk surfaces $\Sigma (\theta)$ is a family of singly periodic minimal surfaces. Only the most symmetric of them, $\Sigma(\pi/4)$, is due to Scherk  and the rest of the family was discovered by Karcher \cite{karcher;minimal-surfaces}. They are often called Scherk's fifth surfaces or Scherk's (saddle) towers but we will refer to them as Scherk surfaces for simplicity.  $\Sigma(\theta)$ is given by the equation
	\[
	\cos^2\theta \cosh \frac{x}{\cos \theta} - \sin^2 \theta \cosh \frac{y}{\sin \theta} = \cos z.
	\]
The surfaces $\Sigma(\theta)$ become degenerate as $\theta \to 0$ or $\theta \to \pi/2$ so we will restrict ourselves to $\theta \in [10 \delta_{\theta}, \frac{\pi}{2} - 10 \delta_{\theta}]$ for $\delta_{\theta}$ as in Definition \ref{def:delta-theta}. 
\begin{figure} [h]
\includegraphics[height=3in]{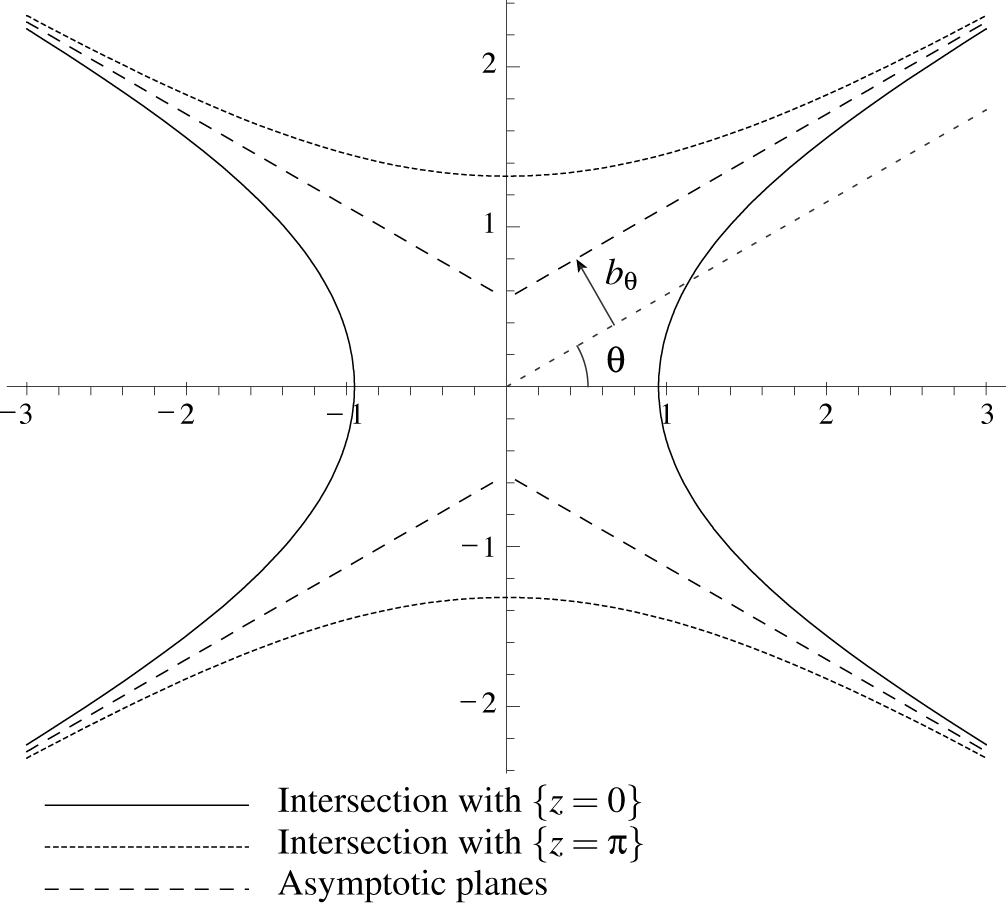}
\caption{Sections of the Scherk surface $\Sigma(\theta)$.} 
\end{figure}
\subsection*{Notations} 
The Scherk surface enjoys many symmetries. In order to refer to them easily in the future, we define the following isometries of the Euclidean space:
	\begin{itemize}
	\item $\rcal_1$ is the identity
	\item $\rcal_2$ is the reflection with respect to the $yz$-plane.
	\item $\rcal_3$ is the reflection with respect to the $z$-axis.
	\item $\rcal_4$ is the reflection with respect to the $xz$-plane.
	\end{itemize}

We denote by $H^+$ the closed half-plane $H^+ = \{ (s,z) \in \R^2 \mid s \geq 0\}$. The vectors $\vec e[\theta]$ and $\vec e\ ' [\theta]$ are defined in equation \eqref{def:e-e-prime}.

We quote Proposition 2.4 from \cite{kapouleas;embedded-minimal-surfaces} for some properties of the Scherk surfaces. 
	\begin{prop}
	\label{prop:scherk}
	$\Sigma(\theta)$ is a singly periodic embedded complete minimal surface which depends smoothly on $\theta$ and has the following properties:
	\begin{enumerate}
	\item $\Sigma(\theta)$ is invariant under the $\mathcal R_i$'s above and also under reflection with respect to the planes $\{ z = k\pi\}$ ($k\in \mathbf Z$).
	\item For given $\eps \in (0, 10^{-3})$, there is a constant $a = a(\delta_{\theta}, \eps)>0$ and smooth functions $f_{\theta}:H^+ \to \R, A_{\theta}:H^+ \to E^3$, and $F_{\theta}:H^+ \to E^3$, such that $W_{\theta}: = F_{\theta}(H^+) \subset \Sigma(\theta)$ and 
		\begin{gather*}
		A_{\theta}(s,z) = (a+s) \vec e[\theta] + z \vec e_z + b_{\theta} \vec e\ '[\theta],\\
		F_{\theta} (s,z) = A_{\theta}(s,z) + f_{\theta} (s,z) \vec e\ '[\theta],
		\end{gather*}
	where $b_{\theta} =\sin (2 \theta) \log(\cot \theta)$. Moreover $f_{\theta}$ and $F_{\theta}$ depend smoothly on $\theta \in [10\delta_{\theta}, \pi/2 - 10\delta_{\theta}]$ and (iii)-(vi) are satisfied. 
	\item $\Sigma(\theta) \setminus \bigcup_{j=1}^4 \mathcal R_j(W_{\theta})$ is connected and lies within distance $a+1$ from the $z$-axis.
	\item $W_{\theta} \subset \{ (r \cos \phi, r \sin \phi, z): r>a, \phi \in [9 \delta_{\theta}, \pi/2 -9\delta_{\theta}]\}.$
	\item $\| f_{\theta}: C^5(H^+, e^{-s})\| \leq \eps $ and $\| df_{\theta}/d\theta:C^5(H^+, e^{-s})\| \leq \eps.$
	\item $|b_{\theta}| + |db_{\theta}/d\theta| <\eps a$. 
	\end{enumerate}
	\end{prop}

For the rest of the article,  $\eps$ is a fixed  small constant so that $a$ is fixed also.

$W_{\theta}$ is called the {\em first wing} of the Scherk surface, and the image of $W_{\theta}$ under $\rcal_i$ is called the {\em $i$th wing}.

We consider as standard coordinates on the $i$th wing the coordinates $(s,z)$ defined by $(s,z) = (\rcal_i \circ F_{\theta})^{-1} (s,z)$ and extend the function $s$ to be zero on the rest of $\Sigma(\theta)$. Using the notation \eqref{notation-s}, we call $\Sigma_{\leq 0} (\theta)$ the \emph{core} of the Scherk surface. Note that the boundary of the core has four connected components, each of which is the boundary of a wing. The numbering of the wings can be reconciled with the numbering of the vectors $v_1, v_2, v_3$ and $v_4$ in Section \ref{ssec:construction-G-bar} by taking $\beta_1$ from Definition \ref{def:tetrad} to be in the second quadrant. 


\FloatBarrier


\subsection{Construction of the core}


The goal of this section is to unbalance a Scherk surface so that its wings are tangent to asymptotic planes determined by a possibly unbalanced tetrad $T$. The dislocations are necessary for dealing with the approximate kernel, as discussed in Section \ref{ssec:construction-G-bar}.

Let us  examine more closely the angles in Definition \ref{def:tetrad} when  the tetrad is formed of the directing vectors of the planes asymptotic to $\Sigma(\theta)$, rotated by an angle $\beta_r$ around the $z$-axis. We have $\beta_1 = \theta+\beta_r$, $\beta_2 = \pi-\theta+\beta_r$, $\beta_3 = \pi+\theta+\beta_r$, $\beta_4 = 2 \pi-\theta+\beta_r$, $\theta(T) = \theta$, $\theta_1 =0$, $\theta_2=0$, and $\theta_r=\beta_r$ is the angle of rotation. 

Given a tetrad $T$ for which $\theta_1=0$ and $\theta_2=0$, we can find a Scherk surface that has the vectors of $T$ as directing vectors: it suffices to take  $\Sigma(\theta(T))$ rotated around the $z$-axis by an angle $\theta_r$. In general however, it is not enough to rotate one of the original Scherk surfaces $\Sigma(\theta)$, we need transformations $Z_1$ and $Z_2$ to change the respective directions of the vectors. 

\begin{definition} We define a family of diffeomorphisms $Z_{1}(\phi):E^3 \to E^3$ parametrized by $\phi \in [-2 \delta_{\theta}, 2 \delta_{\theta}]$ such that:
	\begin{enumerate}
	\item $Z_1$ is the identity on the second and fourth quadrants $\{(x,y,z) \mid xy \leq 0\}$ and in the unit ball. 
	\item on $\{(r \cos \theta', r\sin \theta', z) : r >2, \theta' \in [9 \delta_{\theta}, \pi/2 - 9\delta_{\theta}]\}$, $Z_1$ is a rotation of angle $\phi$ clockwise around the $z$-axis.
	\item on $\{(r \cos \theta', r\sin \theta', z) : r >2, \theta' \in [\pi + 9 \delta_{\theta}, 3\pi/2 - 9\delta_{\theta}]\}$, $Z_1$ is a rotation of angle $\phi$ counterclockwise around the $z$-axis.	
	\end{enumerate}
\end{definition}

\begin{definition} The family of diffeomorphisms $Z_{2}(\phi):E^3 \to E^3$ parametrized by $\phi \in [-2 \delta_{\theta}, 2 \delta_{\theta}]$  is defined by $Z_2 (\phi) = \rcal_2 \circ Z_1(-\phi) \circ \rcal_2$. 
\end{definition}

The transformation $Z_1$ rotates points in the first and third quadrants by $\phi$ toward the fourth quadrant, and $Z_2$ rotates points in the second and fourth quadrants by $\phi$ toward the first quadrant.

For a tetrad $T$,  we define 
	\[
	Z[T] := \mathcal R \circ Z_{1}(\theta_1) \circ Z_{2}(\theta_2),
	\]
where $\mathcal R$ denotes the rotation around the $z$-axis with angle $\theta_r(T)$ counterclockwise. We define the surface 	
	\[
	\Sigma[T]:= Z[T](\Sigma(\theta(T)).
	\]
 By construction, each plane asymptotic to the surface $\Sigma[T]$ is parallel to a vector in $T$. We will not touch the core $\Sigma_{\leq 0}[T]$ in the rest of the construction.


\subsection{Construction of the wings} 


With dislocations, we can solve the linear operator on the desingularizing surface, but we do not control the asymptotic behavior of the solutions as       we move away from the $z$-axis. The additional bending of each wing independently will help us achieve exponential decay. In addition, we use the bending to fit the wings smoothly in the construction of the initial surface in Section \ref{sec:initial-surfaces}.

Let $\tau>0$ be a small constant, $\{\varphi_i\}_{i=1}^4$ four real numbers such that $|\varphi_i| \leq \delta_{\theta}$, and $T$ a tetrad as in the previous section. We describe the construction of the $i$th wing, with a bending of angle $\varphi_i$ below. 

\begin{definition}
	\label{def:pivot}
We call the line $ Z[T] \circ \mathcal R_i \circ A_{\theta(T)} (\pd H^+)$ the \emph{$i$th pivot}  and denote its intersection with the $xy$-plane by $(x_i, y_i)$. 	The $i$th pivot is the boundary of the core projected perpendicularly onto the $i$th asymptotic plane.
\end{definition}

We define the map $\kappa[\tau, x_{i}, y_{i}, s_i]: H^+ \to E^3$ by
	\begin{multline}
	\label{eq:kappa}
	\kappa[\tau, x_{i}, y_{i}, s_i] (s,z)\\
		= \frac{1}{\tau} (\gamma_1(\tau (s +s_i))-\gamma_1(\tau s_i)+ \tau x_i, \gamma_2(\tau (s +s_i))-\gamma_2(\tau s_i)+\tau y_i, \tau z).
	\end{multline}
Note that the line $\kappa[\tau, x_i, y_i,s_i] (\pd H^+)$ is the $i$th pivot and the graph of $\kappa$ is a piece of grim reaper. The constant $s_i$ is chosen so that the conormal tangent vector to $\kappa[\tau, x_i, y_i, s_i]$ at $s=0$ is $v_i=\vec e[\beta_i]$ rotated by an angle $\varphi_i$ counterclockwise around the $z$-axis, i. e.  $\tan(\beta_i + \varphi_i) = \sinh (\tau s_i)$ by \eqref{eq:tangent}. We define the immersion of the asymptotic grim reaper to the $i$th wing by 
	\[
	A_i[T, \varphi_i, \tau] = \kappa[\tau, x_{i}, y_{i}, s_i]
	\]
and define $\nu_i[T,\varphi_i, \tau](s,z)$ to be the normal unit vector to $A_i[T, \varphi_i, \tau](H^+)$ at the point $A_i[T, \varphi_i, \tau] (s,z)$ oriented such that $\nu_i[T,\varphi_i,\tau] (0,0) = (-1)^{i-1} \vec e\ '[\beta_i+\varphi_i]$.

Roughly speaking, the bent wing is the graph of $f_{\theta(T)}$ over the asymptotic grim reaper. For a smooth transition, we need to cut off  $f_{\theta(T)}$ far enough so that the error generated is not too big while keeping the desingularizing surface small enough in the scale of the grim reapers. For this reason, we introduce a small constant $\delta_s$ that will be determined later. The function $F_i[T,\varphi_i, \tau]$ defined below is the immersion of the $i$th wing.

	\begin{definition}
	For given $T$ and $\varphi_i \in [-\delta_{\theta}, \delta_{\theta}]$, we define $F_i[T,\varphi_i, \tau]: H^+ \to E^3$ by 
	\begin{multline*}
	F_i[T,\varphi_i, \tau] (s,z)= \psi[1,0](s) Z[T]\circ \rcal_i \circ F_{\theta(T)} (s, z) \\
				+ \big(1 - \psi[1,0](s)\big) \big( A_i[T,\varphi_i, \tau](s,z) + \psi_s(s)f_{\theta(T)}(s,z) \nu_i[T,\varphi_i, \tau] (s,z) \big) 
	\end{multline*}
where $\psi_s$ is defined by $\psi_s(s) = \psi[4 \delta_s /\tau, 3 \delta_s /\tau](s)$.
	\end{definition}

\begin{figure} [H]
\includegraphics[height=3in]{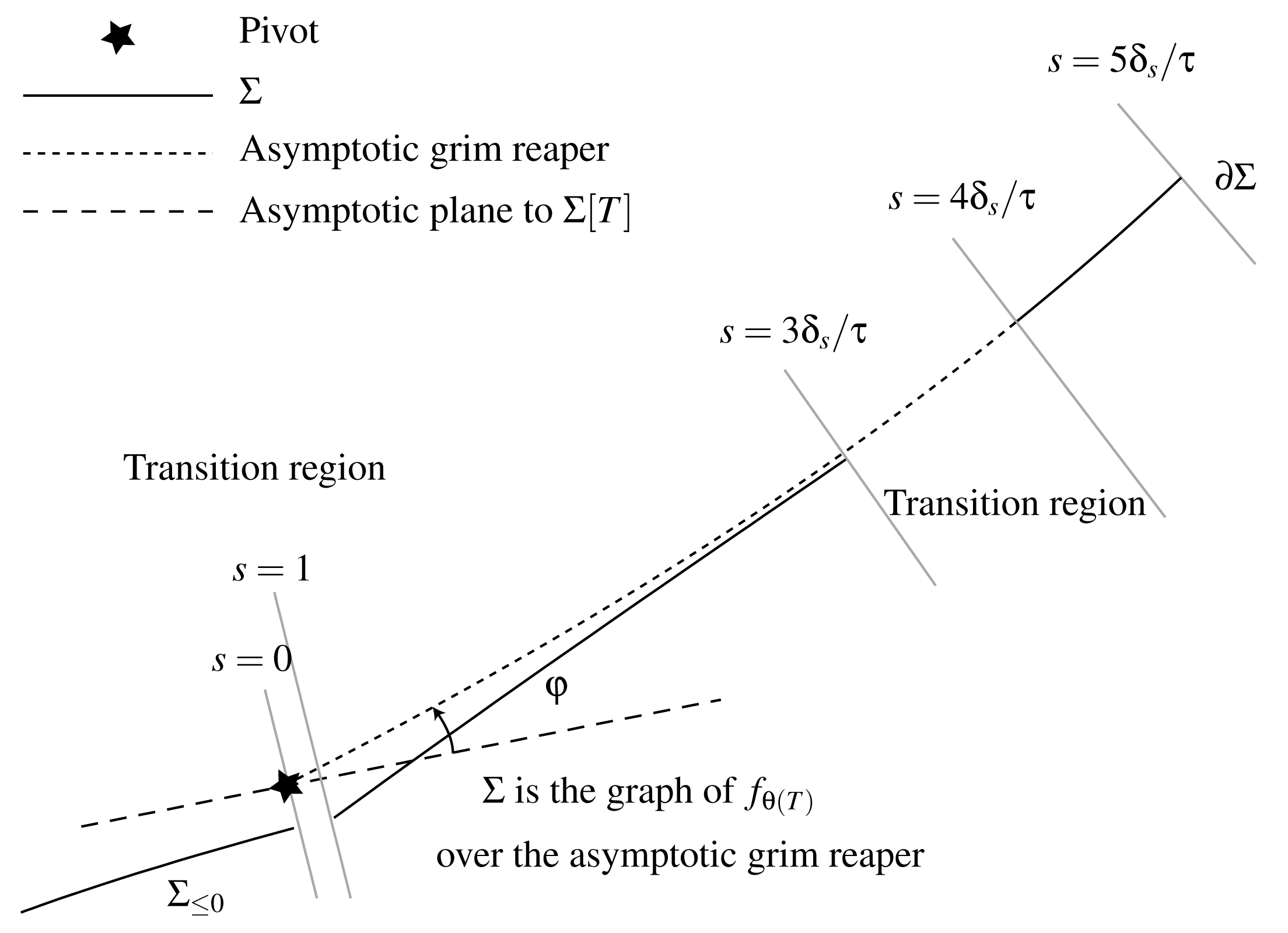}
\caption{Construction of a wing.} 
\end{figure}

\FloatBarrier


\subsection{The desingularizing surfaces $\Sigma[T,\uvarphi, \tau]$}
\label{ssec:desing-surf}


\begin{definition}
For a tetrad $T$ and $\uvarphi = \{ \varphi_i\}_{i=1}^4$ such that $|\uvarphi|\leq \delta_{\theta}$, we define a map $Z[T, \ubar \varphi, \tau]: \Sigma(\theta(T)) \to E^3$  to be $Z[T]$ on the core, and $F_i [T,\varphi_i, \tau] \circ F_{\theta}^{-1} \circ \mathcal R_i^{-1}$ on the $i$th wing of $\Sigma(\theta(T))$. The desingularizing surface $\Sigma[T,\uvarphi, \tau]$ is given by
	\[
	\Sigma=\Sigma[T,\uvarphi, \tau] := Z[T,\uvarphi, \tau](\Sigma_{\leq 5\delta_s/\tau}(\theta(T))).
	\]
\end{definition}

The coordinates $(s,z)$ on $\Sigma(\theta(T))$ are pushed forward by $Z[T,\uvarphi, \tau]$ to coordinates on $\Sigma$. The desingularizing surface is divided in five regions:
	\begin{itemize}
	\item When $s \leq 0$, we are on the core of $\Sigma$. The surface is dislocated here but the bending related to $\uvarphi$ does not affect this region.
	\item  $s \in [0,1]$ is a transition region.
	\item For $s \in [1, 3\delta_s/\tau]$, the wings are  graphs of $f_{\theta(T)}$ on the asymptotic grim reapers; the $i$th asymptotic grim reaper makes an angle $\varphi_i$ with the plane asymptotic to the $i$th wing of $\Sigma[T]$. 
	\item $s \in [3 \delta_s/\tau, 4 \delta_s/\tau]$ is a second transition region where  the function $f_{\theta(T)}$ is cut off. 
	\item For $s \in [4 \delta_s/\tau, 5 \delta_s/\tau]$, the wings are just asymptotic grim reapers.
	\end{itemize}

Note that the desingularizing surfaces are truncated at $s=5 \delta_s /\tau$. The next proposition collects the properties of $\Sigma[T,\uvarphi, \tau]$.

	\begin{prop}
	\label{prop:smooth-dependence}
	There is a constant $\delta'_{\tau} = \delta'_{\tau}(\delta_{\theta})>0$ such that for $T$ satisfying \eqref{eq:require-T},  $|\uvarphi| \leq \delta_{\theta}$, and $\tau \in (0,\delta'(\tau))$, the function $Z[T, \ubar \varphi, \tau]$ satisfies the following properties:
	\begin{enumerate}
	\item $Z[T, \ubar \varphi, \tau]$ is a smooth embedding depending smoothly on its parameters. 
	\item For each $n \in \mathbf Z$, the map $Z[T, \ubar \varphi, \tau]$ is invariant under reflections of the domain and  range  with respect to the plane $\{ z = n \pi\}$.
	\end{enumerate}
	\end{prop}


\section{Estimates on the desingularizing surfaces}
\label{sec:estim}


In the previous section, the dislocations and the bending are constructed independently. It is in our interest to keep the dislocations and the bending small  to control the error generated in the core and the first transition region. If a dislocation is moving a wing in one direction, and the bending moves it back by some amount, we can let the bending cancel part of the dislocation by changing our angles from the onset. This process is called {\em straightening} and is studied in Lemma \ref{lem:4-1}. 

The main part of this section is dedicated to estimating $H-\tau \vec e_y \cdot \nu$ on the desingularizing surface and studying the impact of the dislocations, bending and straightening. For this, we define functions $\{ w_j\}_{j=1}^2$ generated by the dislocations, and functions $\{\bar w_i\}_{i=1}^4$ generated by the straightening. In Lemma \ref{lem:w}, we describe how the functions $\{ w_j\}_{j=1}^2$  can be used to cancel any linear combination of eigenfunctions $\vec e_x \cdot \nu$ and $\vec e_y \cdot \nu$. Finally, we approximate $H-\tau \vec e_y \cdot \nu$ by a linear combination of the $w$'s and the $\bar w$'s with an error of  second order in $\theta_1$, $\theta_2$ and $\uvarphi$ in Proposition \ref{prop:4-20}. We follow the exposition of \cite{kapouleas;embedded-minimal-surfaces} and adapt all the proofs to our case, in particular, we show that the term $-\tau \vec e_y \cdot \nu$ can be controlled throughout. 

	
\subsection{Straightening}


\begin{lemma}
\label{lem:4-1}
There are constants $\delta_{\varphi} = \delta_{\varphi}(\delta_{\theta}) \in (0, \delta_{\theta})$ and $\delta_{\tau}=\delta_{\tau}(\delta_{\theta}) \in (0,\delta'_{\tau})$ such that for a given tetrad $T$ as in Definition \ref{def:tetrad}, $\uvarphi\in \R^4$, and $\tau \in (0, \delta_{\tau}]$ satisfying
	\[
	\theta(T) \in [30 \delta_{\theta}, \frac{\pi}{2}-30\delta_{\theta}], \quad 
	\theta_1(T), \theta_2(T) \in[-\delta_{\theta}, \delta_{\theta}], \quad
	|\ubar \varphi|\leq \delta_{\varphi},
	\]
we have for each $\ubar \varphi' = \{ \varphi'_i\}_{i=1}^4$ with $|\ubar \varphi'| \leq \delta_{\varphi}$, a tetrad $T'$ which depends smoothly on $T, \ubar \varphi, \tau$ and $\ubar \varphi'$, and is characterized by the following properties: 

\begin{enumerate}
\item $(T', \uvarphi-\uvarphi',\tau)$ satisfies the conditions of Proposition \ref{prop:smooth-dependence}. 
\item $T'=T$ when $\uvarphi'=0$.
\item $T'=\{ \vec e[\beta_i']\}_{i=1}^4$ where each $\beta_i'$ depends smoothly on $\ubar \varphi'$ and
	\[
	\left|\frac{\pd \beta_i'}{\pd{\varphi_j'}} -\delta_{ij}\right| \leq C\tau.
	\]
\item There is a smooth function $f_{\ubar \varphi'}$ on $\Sigma[T,\ubar \varphi,\tau]$ which depends smoothly on $T,\ubar \varphi, \tau$ and $\ubar \varphi'$, satisfies $f_{\ubar \varphi'}\equiv 0$ on $\pd\Sigma[T,\ubar \varphi,\tau]$, and whose graph over $\Sigma[T,\ubar \varphi, \tau]$ is contained in the image of $Z[T',\ubar \varphi-\ubar \varphi',\tau]$.
\end{enumerate}

\end{lemma}
\begin{figure} [h]
\includegraphics[height=2.8in]{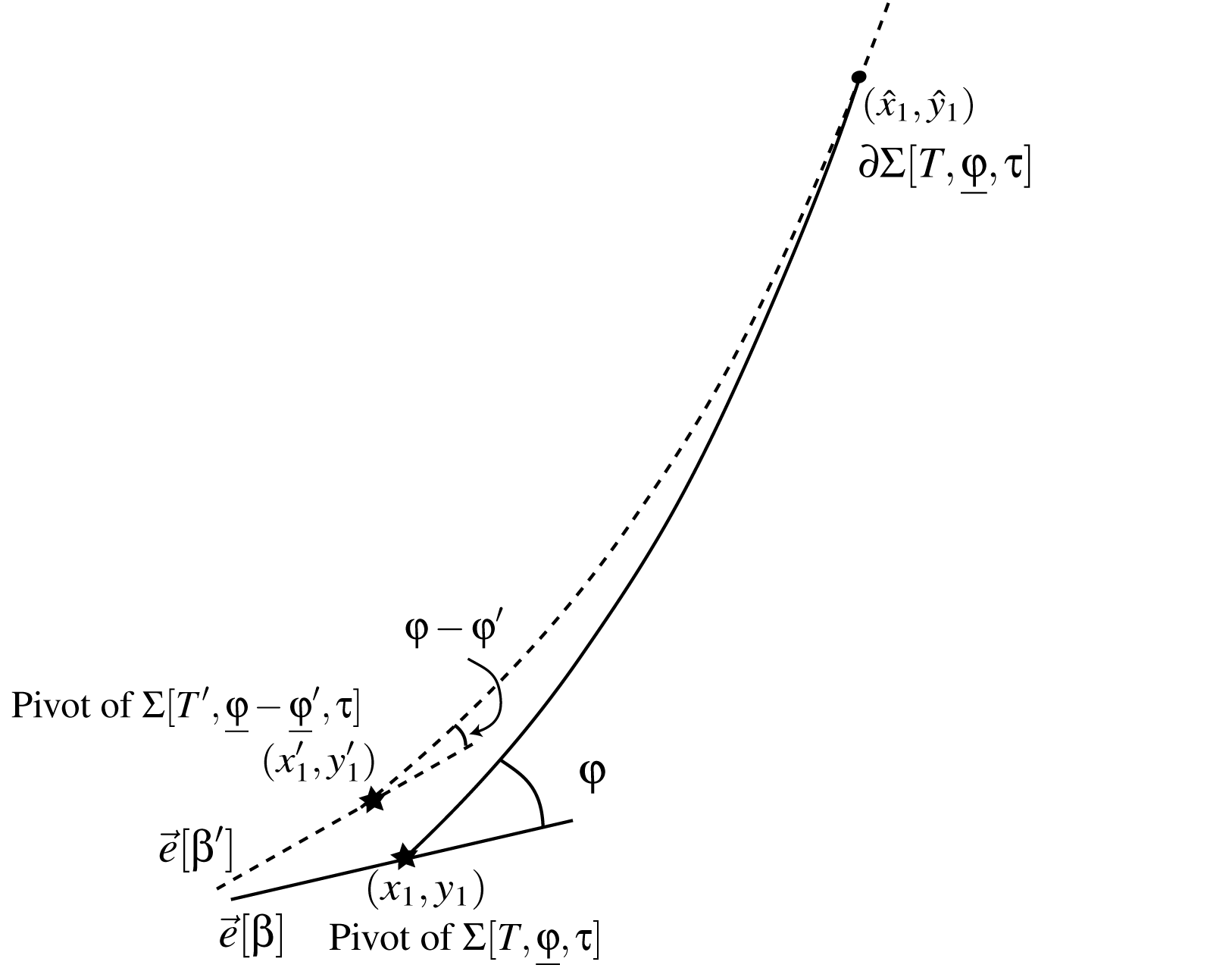}
\caption{The image of $Z[T', \uvarphi-\uvarphi',\tau]$ passes through $\pd \Sigma[T, \uvarphi, \tau]$.} 
\end{figure}
\begin{proof}
We fix $T$, $\ubar \varphi$ and $\tau$. Without loss of generality, we can assume that $\theta_r(T)=0$, otherwise we rotate the whole configuration by $-\theta_r(T)$. For small variations $|\beta_i-\beta_i'|\leq \delta_{\theta}$ and $|\ubar \varphi'|\leq \delta_{\theta}$, the image of $Z[T', \ubar \varphi-\ubar \varphi', \tau]$ is the graph of a function $f$ over $\Sigma[T, \ubar \varphi, \tau]$. On the component of  $\pd \Sigma[T, \ubar \varphi, \tau]$ on the $i$th wing, the function $f$ is a constant, which we denote by $f_i(\ubar \varphi', T')$. Clearly, if $T'=T$ and $\ubar \varphi'=0$, $ (f_1, f_2, f_3, f_4)=0$. The Implicit Function Theorem on $  (f_1, f_2, f_3, f_4)(\ubar \varphi', T')=0$ will give us $T'$ as a function of $\uvarphi'$ for $\uvarphi'$ small enough, provided we show the matrix $\left[{\pd f_i}/{\pd \beta'_j}\right]_{i,j=1, \ldots, 4}$ is invertible.

We study how  changes in  the tetrad and  in the bending affect the first wing in detail. The variation of the three other wings can be obtained similarly. Given $T'=\{ \vec e[\beta_i']\}_{i=1}^4$ and $\uvarphi'$ with $|\beta_i' - \beta_i|$ and $|\uvarphi'|$ small, the asymptotic grim reapers to the first wings of $\Sigma[T, \ubar \varphi, \tau]$ and $\Sigma[T',\ubar \varphi-\ubar \varphi', \tau]$ are parametrized  by $\kappa[\tau, x_1, y_1, s_1] (s,z)$ and $\kappa[\tau, x'_1, y'_1, s'_1] (s,z)$ respectively, where $\kappa[\tau, x_1, y_1, s_1] (s,z)$ is as in \eqref{eq:kappa}, which is the version of the equation below without the primes, 
	\begin{multline*}
	\kappa'[\tau, x'_1, y'_1, s'_1] (s,z)\\ = \frac{1}{\tau} (\gamma_1(\tau (s +s'_1))-\gamma_1(\tau s'_1)+\tau x'_1, \gamma_2(\tau (s +s'_1))-\gamma_2(\tau s'_1)+\tau y'_1, \tau z)
	\end{multline*}
with  $\gamma_1$ and $\gamma_2$ as in \eqref{eq:def-gamma}, and
	\begin{gather}
	\notag (x_1, y_1) =  (a \cos \beta_1, a \sin \beta_1) +  b_{\theta(T)}(-\sin \beta_1, \cos \beta_1),\\
	\notag (x_1', y_1') = (a \cos \beta_1', a \sin \beta_1') + b_{\theta(T')}(-\sin \beta_1', \cos \beta_1'),\\
	\label{eq:def-s1}
	 \sinh (\tau s_1)=\tan (\beta_1 + \varphi_1), \quad \sinh(\tau s_1') =\tan(\beta_1' + \varphi_1 - \varphi'_1).
	\end{gather}

As in the proof of Proposition \ref{prop:init-conf}, the angles $(\beta_1 + \varphi_1)$ and $(\beta_1 +\varphi_1-\varphi'_1)$ stay away from  $\pi/2 +k\pi$, $k\in \mathbf{Z}$, by a fixed amount, so $|\tau s_1|$ and $|\tau s_1'|$ are bounded. Using Proposition \ref{prop:scherk}, we bound the derivatives of $x_1', y_1'$, and $s_1'$ with respect to $\beta_i'$  by
	\begin{gather}
	\label{eq:estim-dxdb}
	\left|\frac{\pd x_1'}{\pd \beta_i'}\right|  \leq C, \qquad \left|\frac{\pd y_1'}{\pd \beta_i'} \right| \leq C, \\
	\label{eq:estim-ds'db}
	\frac{1}{C\tau}  \delta_{1i} \leq \left|\frac{\pd s_1'}{\pd \beta_i'} \right|= \left|-\frac{\pd s_1'}{\pd \varphi'_i}\right|= \left|\frac{1}{\tau \cosh (\tau s_1')} \delta_{1i}\right| \leq \frac {C} {\tau} \delta_{1i} .
	\end{gather}

To simplify notations, let us work in the $xy$-plane. The boundary  $\pd\Sigma[T, \ubar \varphi, \tau]$ is at the point $(\hat x_1, \hat y_1) := \kappa[\tau, x_1, y_1, s_1](5 \delta_s /\tau)$.
The line orthogonal to $\Sigma[T,\ubar \varphi, \tau]$ at the point $(\hat x_1, \hat y_1)$ intersects the surface  $\Sigma[T', \ubar \varphi-\ubar \varphi', \tau]$ at a distance 	
	\begin{equation}
	\label{eq:t=}
	f_1(\uvarphi', T')= t	=  \frac{ \gamma_2(\tau (s+s_1')) - \gamma_2(\tau s_1') + \tau y_1'-\tau\hat y_1}{\tau\gamma_1' (5 \delta_s +\tau s_1) },
	\end{equation}
where $\gamma_1'$ and $\gamma_2'$ are the derivatives of  $\gamma_1$ and $\gamma_2$ respectively, and $s$ is the coordinate on $\Sigma[T', \ubar \varphi-\ubar \varphi', \tau]$ of the intersection point. Note that since $|\tau s_1|$ is  bounded, $\gamma'_1( 5 \delta_s +\tau s_1)$ is bounded away from $0$. The value of $s$ is given implicitly by the equation 
	\begin{multline}
	\label{eq:s}
	\left(\gamma_1(\tau s_1') - \tau  x_1' + \tau \hat x_1\right) \gamma_1'(5 \delta_s +\tau s_1) + \left( \gamma_2(\tau s_1') -  \tau y_1' + \tau \hat y_1\right) \gamma_2'(5 \delta_s +\tau s_1) \\
	= \gamma_1(\tau (s+s_1')) \gamma_1'(5 \delta_s +\tau s_1) + \gamma_2(\tau (s+s_1')) \gamma_2'(5 \delta_s +\tau s_1).
	\end{multline}

Hence, given $\uvarphi'$ and $T'$, we can find $x_1, y_1, s_1$ and $x'_1, y'_1, s'_1$. Solving the system of equations \eqref{eq:t=} \eqref{eq:s} above, we get $s$ and $t$. We first study the dependence of $s$ on $x'_1, y'_1, s'_1$ and estimate its derivatives with respect to each of the $x'_1, y'_1$ and $s'_1$. 

Consider $x_1', y_1', s_1'$ as independent variables and define the function $F$ below related to \eqref{eq:s}, 
	\begin{multline*}
	F(x_1', y_1', s_1', \sigma) \\= \left(\gamma_1(\tau s_1') - \tau x_1' + \tau\hat x_1\right) \gamma_1'(5 \delta_s +\tau s_1) + \left( \gamma_2(\tau s_1') - \tau y_1' + \tau\hat y_1\right) \gamma_2'(5 \delta_s +\tau s_1) \\
		-\gamma_1(\tau \sigma) \gamma_1'(5 \delta_s +\tau s_1) - \gamma_2(\tau \sigma) \gamma_2'(5 \delta_s +\tau s_1).
	\end{multline*}
Note that $F(x_1, y_1, s_1, 5\delta_s/\tau+s_1) = 0$. Its  derivative with respect to $\sigma$ satisfies
	\[
	\frac{\pd F}{\pd \sigma} (x_1, y_1, s_1, 5\delta_s/\tau+s_1)=- \tau \gamma_1'^2(5 \delta_s +\tau s_1)-\tau \gamma_2'^2(5 \delta_s +\tau s_1)=-\tau.
	\]
Let us denote by $B_r(x_1, y_1,s_1)$ the ball of radius $r$ centered at $(x_1, y_1,s_1)$.
By the Implicit Function Theorem, there are constants $r>0$, $\tilde r>0$ and a function $h: B_r (x_1, y_1, s_1) \to  (5\delta_s+\tau s_1-\tilde r, 5\delta_s+\tau s_1+\tilde r)$ such that $F(x_1', y'_1, s_1', h(x_1', y'_1, s_1')) =0$. Since $s= h-s_1'$,
	\begin{gather}
	\label{eq:estim-dsdx-dsdy}\left|\frac{\pd s}{\pd x_1'} (x_1', y_1', s_1')\right| \leq C, \qquad \left|\frac{\pd s}{\pd y_1'} (x_1', y_1', s_1') \right| \leq  C,\\
	\label{eq:dsds1}
	\frac{\pd s}{\pd s_1'} =  \frac{\gamma_1'(\tau s_1') \gamma_1'(5 \delta_s+\tau s_1) +\gamma_2'(\tau s_1') \gamma_2'(5 \delta_s+\tau s_1)}{\gamma_1'(\tau (s+s_1')) \gamma_1'(5 \delta_s+\tau s_1) +\gamma_2'(\tau (s+s_1')) \gamma_2'(5 \delta_s+\tau s_1)} -1.
	\end{gather}

We now show that the determinant of $\left[\frac{\pd f_i}{\pd \beta'_j}\right]_{i,j=1, \ldots, 4}$ does not vanish. Since the  functions $\{f_i\}_{i=1}^4$ play a similar role, it suffices to study one of them, say $f_1$, in detail.   We have 
 	\[
	\frac{\pd s}{\pd \beta_i'}=\frac{\pd s}{\pd x_1'} \frac{\pd x_1'}{\pd \beta_i'} + \frac{\pd s}{\pd y_1'} \frac{\pd y_1'}{\pd \beta_i'}+\frac{\pd s}{\pd s_1'} \frac{\pd s_1'}{\pd \beta_i'},
	\] 
and, from \eqref{eq:t=}, 
	\begin{equation}
	\label{eq:pdf/pdbeta}
	\gamma_1'(5 \delta_s + \tau s_1) \frac{\pd f_1}{\pd \beta'_i} = I + I\!I,
	\end{equation}
where 
	\begin{gather*}
	I:= \left(\gamma_2'(\tau (s+ s_1')) \left(\frac{\pd s}{\pd s_1'} +1\right)-  \gamma_2'(\tau s_1')\right) \frac{\pd s_1'}{\pd \beta_i'}, \\
	I\!I :=  \gamma_2'(\tau (s+ s_1')) \left( \frac{\pd s}{\pd x_1'} \frac{\pd x_1'}{\pd \beta_i'} + \frac{\pd s}{\pd y_1'} \frac{\pd y_1'}{\pd \beta_i'} \right)  +  \frac{\pd y_1'}{\pd \beta_i'}.
	\end{gather*}
	
 From a previous discussion, $\gamma_1'(5 \delta_s + \tau s_1) $ is bounded away from $0$, so it is enough to estimate $I$ and $I\!I$. The bounds  \eqref{eq:estim-dxdb}  and \eqref{eq:estim-dsdx-dsdy}  imply  $|I\!I| \leq C$. For  $I$, we use the explicit formula  \eqref{eq:dsds1}, 
	\begin{equation}
	\label{eq:I}
	I =    \frac{\gamma_1'(5 \delta_s+\tau s_1)\big(\gamma_1'(\tau s_1') \gamma_2'(\tau (s +  s_1'))-\gamma_1'(\tau (s+ s_1'))  \gamma_2'(\tau s_1')\big)}{\gamma_1'(\tau (s+ s_1')) \gamma_1'(5 \delta_s+\tau s_1) +\gamma_2'(\tau (s+ s_1')) \gamma_2'(5 \delta_s+\tau s_1)} \frac{\pd s_1'}{\pd \beta_i'}.
	\end{equation}
We can assume without loss of generality that $\tilde r \leq \delta_s$ so the denominator is close to $1$ since $\tau (s+ s_1') = \tau \sigma \in (4 \delta_s + \tau s_1, 6 \delta_s +\tau s_1)$. The second factor of the numerator is equal to 
	\begin{equation}
	\label{eq:I-num}
	\frac{\sinh(\tau (s +  s_1')) - \sinh (\tau s_1')}{\cosh (\tau (s +  s_1')) \cosh (\tau s_1')} = \frac{\cosh (\tau \bar s)}{\cosh (\tau (s +  s_1')) \cosh (\tau s_1')} \tau s
	\end{equation}
for some $\bar s \in (s_1', s_1' + s_1)$ given by the Mean Value Theorem. Combining \eqref{eq:estim-ds'db}, \eqref{eq:I} and \eqref{eq:I-num}, we have
	\begin{equation}
	\label{eq:estim-I}
	\frac{1}{C\tau}  (4 \delta_s + \tau (s_1 -s_1')) \delta_{i1}\leq |I|\leq \frac{C}{\tau}  (6 \delta_s +\tau (s_1 - s_1')) \delta_{i1}.
	\end{equation}
Choosing  $\delta_{\varphi}$  small enough,  we can ensure that $4 \delta_s - \tau |s_1 -s_1'| >0$ by \eqref{eq:def-s1}. From \eqref{eq:pdf/pdbeta}, \eqref{eq:estim-I},  the fact that $|I\!I|\leq C$, and similar estimates for $f_2, f_3$ and $f_4$, we have for $\tau$ small enough,
	\begin{equation}
	\label{eq:est-dfdb}
	\det \left(\left[\frac{\pd f}{\pd \beta'}\right]\right) \geq \frac{1}{ C \tau},
	\end{equation}
 therefore ${\pd f}/{\pd \beta'}$ has an inverse with norm bounded by $C \tau$. By the Implicit Function Theorem, for every $\uvarphi'$ small enough, there is a tetrad $T'$ such that $(f_1, f_2, f_3, f_4)(T', \uvarphi') =0$. To get the estimate (iii), we write 
	$
	\frac{\pd f}{\pd \beta'}\frac{\pd \beta'}{\pd \ubar \varphi'}+\frac{\pd f}{\pd \ubar \varphi'} =0,
	$
or equivalently,
	\begin{equation}
	\label{eq:dbdp}
	\frac{\pd f}{\pd \beta'}\left(\frac{\pd \beta'}{\pd \ubar \varphi'} -Id\right) = -\frac{\pd f}{\pd \ubar \varphi'}-\frac{\pd f}{\pd \beta'}.
	\end{equation}
To estimate ${\pd f}/{\pd \beta'}+ {\pd f}/{\pd \ubar \varphi'}$, note that the contribution of order $\tau^{-1}$ in ${\pd f}/{\pd \beta'}$ comes from the derivative ${\pd s_1'}/{\pd \beta_i'}$. Adding the derivative with respect to $\varphi'$, we get
	$
	{\pd s_1'}/{\pd \beta_i'}+ {\pd s_1'}/{\pd  \varphi_i'}=0,
	$
 by \eqref{eq:def-s1}.
Therefore $
	\left| \frac{\pd f}{\pd \beta'}+ \frac{\pd f}{\pd \ubar \varphi'} \right| \leq C, 
	$
  and the result (iii) follows from \eqref{eq:est-dfdb} and \eqref{eq:dbdp}.
\end{proof}

We may need the values of $\delta_{\varphi}$ and $\delta_{\tau}$ to be smaller than the ones given in Lemma \ref{lem:4-1} for later estimates. When we write ``for $\tau$ small enough" in the rest of the article, we mean that the value of $\delta_{\tau}$ has to be adjusted accordingly, and similarly for $\delta_{\varphi}$.  


\subsection{Graphs of functions on a surface}
\label{ssec:perturbation}


The wings are the graphs of small functions over the asymptotic grim reapers. In order to estimate the mean curvature and the second fundamental form, we take a brief detour and discuss some standard facts about normal perturbations of surfaces (see Appendix C of \cite{kapouleas;surfaces-euclidean}, or Appendix B of \cite{kapouleas;embedded-minimal-surfaces}).

Suppose we have a surface $M$ in $E^3$, immersed by a $C^2$ map $X: M \to E^3$. We write $g, A, H$ and $\nu$ for the first and second fundamental forms, the mean curvature and the Gauss map of $M$ respectively. For a $C^2$ function $\sigma$ on $M$, we define $X_{\sigma}:M \to \R^3$ by $X_{\sigma} \equiv X+{\sigma} \nu$. When $X_{\sigma}$ is an immersion, we denote by $M_{\sigma}$ the graph of ${\sigma}$ over $M$ and by  $g_{\sigma}$, $A_{\sigma}$, $H_{\sigma}$, and $\nu_{\sigma}$ the first and second fundamental forms, the mean curvature, and the Gauss map of $X_{\sigma}(M)$ pulled back to $M$. 

We use $\Phi$ to denote a term which can be either ${\sigma} A$ or $\nabla {\sigma}$. We use an $\ast$ to denote a contraction with respect to $g$. Also, all the $G$'s below stand for linear combinations with universal coefficients of terms which are contractions with respect to $g$ of at least two $\Phi$'s. The $\hat G$'s  denote linear combinations (with universal coefficients) of terms which are contractions of a number of - possibly none - $\Phi$'s with one of the following: 
	\begin{enumerate}
	\item $A\ast \Phi \ast \Phi$,
	\item ${\sigma} \nabla A \ast \Phi$,
	\item $ {\sigma} A \ast \nabla^2 {\sigma}$,
	\item $\nabla^2 {\sigma} \ast \Phi \ast\Phi$.
	\end{enumerate}

 Let $e_1, e_2, \nu$ be a local orthonormal frame of $E^3$ whose restriction to $M$ has $e_1$ and $e_2$ tangent to $M$. If $|{\sigma} A|<1$, then $X_{\sigma}$ is an immersion and we have 
	\begin{gather}
	\label{eq:perturb-g} g_{{\sigma}ij} = g_{ij} - 2 {\sigma} A_{ij} + G_{ij},\\
	\nu_{\sigma}	 = \nu - \nabla {\sigma} + \mathbf Q^{\nu}_{\sigma} \label{eq:perturb-nu},
	\end{gather}
where $\mathbf Q^{\nu}_{\sigma}=G_1 e_1 + G_2 e_2 +G_3 \nu + \frac{G_ 5 e_1 + G_6 e_2 + G_7 \nu}{1+G_4+\sqrt{1+G_4}} $.
	\begin{align}
	H_{\sigma}			 &= H+(\Delta {\sigma} + |A|^2 {\sigma})  + Q_{\sigma} \label{eq:perturb-H},
	\end{align}
where  $Q_{\sigma} = \frac{ \hat G_1}{\sqrt{1+G_8}} + \frac{ \hat G_2}{1+G_8 +\sqrt{1+G_8}}$. 
Therefore 
	\begin{equation}
	\label{eq:perturb-st}
	H_{\sigma} - \tau \vec e_y \cdot \nu_{\sigma}= H-\tau \vec e_y \cdot \nu+\Delta_g \sigma + |A|^2 \sigma    + \tau \vec e_y \cdot \nabla \sigma + Q_{\sigma}+ \tau\vec  e_y \cdot \mathbf{Q}^{\nu}_{\sigma}.
	\end{equation}


\subsection{Notations} 
\label{ssec:notations2}


We will use the same notation for functions, tensors, and operators on the asymptotic grim reaper and their pushforwards by $F_i \circ A_i^{-1}$ to $\Sigma_{\geq 1}[T,\uvarphi, \tau]$, and vice versa. 
To avoid confusion, we use symbols without subscripts for the geometric quantities considered on the asymptotic grim reapers; we use symbols with subscripts $\Sigma$ for their counterparts on  $\Sigma_{\geq 1}$.  
For example, $g$ denotes the metric on the asymptotic grim reaper (induced by its immersion in $E^3$) and it also denotes the pushforward of this metric  to $\Sigma_{\geq 1}$, while $g_{\Sigma}$ denotes the metric on $\Sigma_{\geq 1}$ induced by the metric in $E^3$  or its pullback to the asymptotic grim reaper.  

For $\uvarphi'$ as in Lemma \ref{lem:4-1} and a fixed $i \in \{1, \ldots, 4\}$, we use a dot $\dot{}$ to denote the differentiation $\pd/\pd \varphi_i'|_{\uvarphi'=0}$.


\subsection{Estimates on the desingularizing surface $\Sigma[T,\ubar \varphi, \tau]$}
\label{ssec:estimates}


\begin{lemma}
\label{lem:4-3} $|(\theta(T'))^{\cdot}|\leq C$ and the following are valid on $\Sigma_{\geq1}[T, \uvarphi, \tau]$:
	\begin{enumerate}
		\item $\| \dot \kappa: C^k(g)\| \leq C$,
		\item $\| A: C^k(g) \| \leq C\tau$,
		\item $\| \dot A: C^k(g) \| \leq C\tau^2$,
		\item $\| \nu: C^k(g) \| \leq C$,
		\item $\| \dot \nu : C^k (g) \| \leq C\tau$,
	\end{enumerate}
where $\kappa$ is as in \eqref{eq:kappa}, $A$  and $\nu$ are as in  Section \ref{ssec:notations} and Notations \ref{ssec:notations2}, and all the constants $C$ depend only on $k$.
\end{lemma}

\begin{proof} 
From \eqref{eq:kappa}, $\kappa$ is an isometry of $H^+$ to an asymptotic grim reaper.
Keeping in mind that $\dot f = \sum_{j=1}^4\frac{\pd f }{\pd \beta_j'}\frac{\pd \beta_j'}{\pd \varphi_i '}+\frac{\pd f}{\pd \varphi_i '} |_{\uvarphi'=0}$ for any function $f(T', \uvarphi')$, the lemma follows immediately from the explicit formula for the position $\kappa$ and the  estimates \eqref{eq:estim-dxdb} and  (iii) of Lemma \ref{lem:4-1}. \end{proof}

\begin{cor}
\label{cor:4-4}
The following estimates are valid on $\Sigma_{\geq 1}[T,\uvarphi,\tau]$, where $l=5 \delta_s/\tau$:
	\begin{gather*}
	\| g_{\Sigma} - g :C^3(\Sigma_{\geq 1}, g, e^{-s})  \| \leq C\eps,  \\
	\| |A_{\Sigma}|^2 - |A|^2 : C^3(\Sigma_{\geq 1}, g, e^{-s})\| \leq C\eps,\\
	\| |A_{\Sigma}|^2: C^3(\Sigma_{\geq 1}, g, e^{-s} +l^{-2})\| \leq C \eps + C \delta^2_s.
	\end{gather*}
In particular, $g$ and $g_{\Sigma}$ are uniformly equivalent on  $\Sigma_{\geq 1}[T,\uvarphi,\tau]$ by assuming without loss of generality that  $\eps$ is small enough. 
\end{cor}

	\begin{proof} The variation of a metric under a normal perturbation $ \sigma= \psi_s f_{\theta(T)}$ is given by \eqref{eq:perturb-g}. The first estimate follows from (ii) in Lemma \ref{lem:4-3} and Proposition \ref{prop:scherk}.

Similarly, we prove the bound on $\||A_{\Sigma}|^2 - |A|^2\|$ using the fact that the perturbation of $|A|^2$ is at least linear in $\sigma$, and the second fundamental form is controlled by the previous lemma. 

For the last estimate, we write
	\[
	\| |A_{\Sigma}|^2:C^3(g, e^{-s}+l^{-2})\| \leq \| |A_{\Sigma}^2 - |A|^2 :C^3(g, e^{-s}) \| + l^{-2} \| |A|^2:C^3(g) \|.
	\]
The first term on the right hand side is controlled by $C \eps$, while the second term is bounded by $C \delta_s^2$ using Lemma \ref{lem:4-3}. 
	\end{proof}

\begin{lemma}
\label{lem:4-5}
Given  $\gamma \in (0, 1)$, we have 
	\[
	\| H_{\Sigma} - \tau \vec e_y \cdot \nu_{\Sigma}: C^2(\Sigma_{\geq 1}[T, \uvarphi, \tau], g, e^{-\gamma s}) \| \leq C\tau,
	\]
where $H_{\Sigma}$ is the mean curvature of $\Sigma[T,\uvarphi, \tau]$.
\end{lemma}

The proof is the same as the proof of Lemma 10 in \cite{mine;tridents}. We reproduce it here for the reader's convenience. 

\begin{proof}  First, note that the estimate is true for $s \geq 4 \delta_s/ \tau$. Let us now work in the region $s \in [ 3 \delta_s/ \tau, 4 \delta_s/\tau]$. We have $H - \tau \vec e_y \cdot \nu \equiv 0$ on the grim reaper cylinder, so by the variation formulas in Section \ref{ssec:perturbation} , $H_{\Sigma} - \tau \vec e_y \cdot \nu_{\Sigma}$ has terms at least linear involving ${\sigma} A_{ij}$, $\nabla {\sigma}$ and $\nabla^2_{ij}{\sigma}$ (with ${\sigma}= \psi_s f_{\theta(T)}$). We are on the support of the derivative of the cut-off function $\psi_s$,  so these terms, their first and second derivatives behave like $(\frac{\tau}{\delta_s})^k  e^{-s}$, $0 \leq k \leq 4$.  For $s \geq 3 \delta_s/\tau$, we can arrange $(\frac{\tau}{\delta_s})^k  e^{-s} \leq e^{-s} \leq \tau e^{-\gamma s}$ to be true by taking $\tau$ small enough in terms of $\gamma$. 

In the region $s \leq 3 \delta_s /\tau$, we have $\psi_s \equiv 1$ and $\sigma=f_{\theta(T)}$. The plane and the original Scherk surface are minimal surfaces, so
	\[
	0 =  \Delta_{g_{\R^2}} \sigma +  Q',
	\]
where $ Q'$ is the $Q_{\sigma}$ defined by equation \eqref{eq:perturb-H}, and  every geometric quantity and tensor in $Q'$ is taken with respect to the flat metric on the asymptotic plane (to the original Scherk surface). Since $H - \tau \vec e_y \cdot \nu$ vanishes on the grim reaper cylinder, 
	\[
	H_{\Sigma} - \tau \vec e_y \cdot \nu_{\Sigma}= \Delta_g \sigma + |A|^2 \sigma  + Q  + \tau \vec e_y \cdot \nabla \sigma + \tau\vec  e_y \cdot \mathbf{Q}^{\nu},
	\]
where $Q$ is $Q_{\sigma}$ from equation \eqref{eq:perturb-H} and where  every geometric quantity and tensor in $Q$ is taken with respect to the metric on the asymptotic grim reaper. The term $\mathbf{Q}^{\nu}$ is  an expression at least quadratic in $\nabla \sigma$ and $\sigma A$ given by \eqref{eq:perturb-nu}. Since the asymptotic grim reaper is isometric to a plane with flat metric,
	\begin{align*}
	H_{\Sigma}  - \tau \vec e_y \cdot \nu_{\Sigma}&= - Q' + |A|^2 \sigma + Q + \tau \vec e_y \cdot \nabla \sigma + \tau \vec e_y \cdot \mathbf{Q}^{\nu}.
	\end{align*}
By Lemma \ref{lem:4-3},  the fact that $\| \sigma: C^5(\Sigma_{\geq 1}, g, e^{-s}) \| \leq \eps$ and  equation \eqref{eq:perturb-nu}, we have 
	\[
	\|  |A|^2 \sigma  + \tau \vec e_y \cdot \nabla \sigma + \tau \vec e_y \cdot \mathbf{Q}^{\nu}: C^2(\Sigma_{\geq 1}, g, e^{-\gamma s}) \| \leq C\tau.
	\]
We are left with
	\[
	Q-Q' = \frac{\hat G'_1}{\sqrt{1+G'_8}}-\frac{\hat G_1}{\sqrt{1+G_8}} + \frac{\hat G'_2}{1+G'_8 +\sqrt{1+G'_8}}- \frac{\hat G_2}{1+G_8 +\sqrt{1+G_8}}.
	\]
We can reduce the fractions to the same denominator and expand the numerators and the square roots in Taylor series. 
The expressions for  $G$'s, $G'$'s, $\hat G$'s  and $\hat G'$'s comprise terms involving $\sigma$, $\nabla \sigma$, $\nabla^2 \sigma $, $A$, $\sigma A$ and $\sigma \nabla^2 A$ and at least quadratic in $\sigma$. If a term in the expansion of the numerators  has either $A$, $\sigma A$ or $\sigma \nabla^2 A$, it can be bounded by $C \tau$  thanks to Lemma \ref{lem:4-3}. We now claim that there is no term involving only $\nabla \sigma$ and $\nabla^2 \sigma $. The grim reaper is isometric to a flat plane, so $\hat G_1'$ only differs from $\hat G_1$ by terms involving the second fundamental form $A$. In other words, setting $A=0$ in the expression for $\hat G_1$ gives us $\hat G_1'$. The same property is true for any $\hat G$ ($G$) and its corresponding $\hat G'$ ($G'$ respectively). Therefore, they contain exactly the same terms involving only $\nabla \sigma$ and $\nabla^2 \sigma$, and these terms can  be paired and cancelled. 
\end{proof}

\begin{lemma}
\label{lem:4-6}
	\[
	\left \|\left.\frac{\pd}{\pd \varphi_i'}\right|_{\uvarphi'=0} (H_{\Sigma}-\tau \vec e_y\cdot \nu_{\Sigma}) :C^1(\Sigma_{\geq 1}[T, \uvarphi,\tau], e^{-\gamma s} ) \right\| \leq C \tau.
	\]
\end{lemma}

\begin{proof}
From  equation \eqref{eq:perturb-st}, we have
	\[
	H_{\Sigma} - \tau \vec e_y\cdot \nu_{\Sigma}= \Delta_g \sigma + |A|^2 \sigma  + Q +\tau \vec e_y \cdot \nabla \sigma + \tau \vec e_y \cdot \mathbf Q^{\nu},
	\]
where $\sigma = \psi_s f_{\theta}$. Therefore
	\[
	(H_{\Sigma} - \tau \vec e_y\cdot \nu_{\Sigma})^{\cdot} = (|A|^2)^{\cdot} {\sigma}+ |A|^2 \psi_s \frac{\pd f_{\theta}}{\pd \theta} (\theta(T'))^{\cdot} + \dot Q + \tau \psi_s \vec e_y \cdot \frac{\pd \nabla f_{\theta}}{\pd \theta} (\theta(T'))^{\cdot} + \tau \vec e_y \cdot \dot{\mathbf Q}^{\nu}
	\]
and we have the result from the estimates in Proposition \ref{prop:scherk}, Lemma \ref{lem:4-3} and the definitions of $Q'$ and $\mathbf Q^{\nu}$ given in Section \ref{ssec:perturbation}.
\end{proof}


\subsection{The functions $\bar u_i$ and $\bar w_i$}


\begin{definition} Let $Y$ be the variation vector field on $\Sigma_{\geq 1}[T, \uvarphi, \tau]$ due to changing $\varphi_i'$ that is on the component contained in the $j$th wing,
	\[
	Y=(F_j[T', \uvarphi-\uvarphi',\tau])^{\cdot},
	\]

and let $Y_{\parallel} := Y-(Y\cdot \nu_{\Sigma}) \nu_{\Sigma}$ be the tangential component of $Y$. 
\end{definition}

\begin{lemma}
\label{lem:4-8}
  $\| Y_{\parallel}:C^1({\Sigma_{\geq 1}[T,\uvarphi,\tau]})\|\leq C$ and $\| Y\cdot \nu_{\Sigma}:C^1({\Sigma_{\geq 1}[T,\uvarphi,\tau])}\|\leq C$.
 \end{lemma}
 \begin{proof}
 We have 
 	\[
	Y=(\kappa + \psi_s f_{\theta} \nu)^{\cdot} = \dot \kappa + \psi_s \frac{\pd f_{\theta}}{\pd \theta} (\theta(T'))^{\cdot}+ \psi_s \sigma_{\theta} \dot \nu
	\]
and the result follows from Proposition \ref{prop:scherk} and Lemma \ref{lem:4-3}. \end{proof}
 
\begin{definition}
We define  the functions $\bar u_i'$ for $ i=1,\ldots,4$ on $\Sigma[T,\uvarphi,\tau]$ by 
	\[
	\bar u_i':= \left.\frac{\pd}{\pd \varphi_i'} \right|_{\uvarphi'=0} f_{\uvarphi'},
	\]
where $f_{\uvarphi'}$ is as in Lemma \ref{lem:4-1}.
\end{definition}

\begin{definition}
\label{def:linop}
Given any smooth surface $S$ in $E^3$, we define the linear differential operators
	\begin{gather*}
	L_S:=\Delta_S+ |A_S|^2,\\
	\lin_S:=\Delta_S+ |A_S|^2 + \tau \vec e_y\cdot \nabla,
	\end{gather*}
 where $g_S$, $|A_S|^2$, and $\Delta_S$  denote the first fundamental form of $S$, the square norm of the second fundamental form of $S$, and the Laplacian with respect to $g_{S}$ on $S$. 
\end{definition}

\begin{lemma}
\label{lem:4-11}
On $\Sigma_{\geq 1}[T,\uvarphi,\tau]$,   $\bar u'=Y\cdot \nu_{\Sigma}$ and 	
	\[
	(H_{\Sigma}-\tau\vec e_y\cdot \nu_{\Sigma})^{\cdot}=Y_{\parallel} (H_{\Sigma}-\tau \vec e_y\cdot \nu_{\Sigma})+\lin_{\Sigma} \bar u_i'.
	\]
\end{lemma}

\begin{proof}
As $\uvarphi'$ changes, one can see the change in the surface as a reparamatrization of the surface followed by $X+f_{\uvarphi'} \nu_{\Sigma}$. The derivative $^{\cdot}$ of the variation of the surface is $Y=Y_{\parallel} + \bar u_i' \nu_{\Sigma}$. Using the fact that differentiation is linear, we get
	\[
	(H_{\Sigma}-\tau \vec e_y\cdot \nu_{\Sigma})^{\cdot} = Y_{\parallel} (H_{\Sigma}- \tau \vec e_y\cdot \nu_{\Sigma}) + Y_{\perp}(H_{\Sigma}-\tau \vec e_y\cdot \nu_{\Sigma}),
	\]
and the derivative of the normal variation is $\lin_{\Sigma} \bar u'$, from Section \ref{ssec:perturbation}.\end{proof}

\begin{cor}
\label{cor:4-12} 
$ \| \lin_{\Sigma} \bar u_i' : C^1(\Sigma_{\geq 1} [T,\uvarphi, \tau], e^{-\gamma s} )\| \leq C\tau. $
\end{cor}
	
\begin{proof} By Lemma \ref{lem:4-11},  $\lin_{\Sigma} \bar u_i'  = Y_{\perp}(H_{\Sigma}-\tau\vec  e_y\cdot \nu_{\Sigma})$, so the estimate follows from Lemmas \ref{lem:4-5} and \ref{lem:4-8}.
\end{proof}

We correct $\bar u_i'$ to $\bar u_i$ so that $\lin_{\Sigma} \bar u_i =0$ on $\Sigma_{\geq 2}$ by solving a differential equation on each asymptotic grim reaper $\Gamma^{j}=A_j[T,\varphi_j,\tau](H^+)$, $j=1, \ldots, 4$.  The region $\Gamma^j_{\geq 1}$ is isometric to a cylinder $\Omega = [1, 5 \delta_s /\tau]/G'$, where $G'$ is the group generated by $(s,z) \to (s,z+2\pi)$. We consider the linear operator $\lin_{\Sigma}$ on the asymptotic grim reaper, where $g_{\Sigma}$ and $A_{\Sigma}$ are the metric and the second fundamental form on $\Sigma$ pulled back to the asymptotic grim reaper, following the Notations \ref{ssec:notations2}. By Corollary \ref{cor:4-4}, we can choose $\eps$ and $\delta_s$ small enough so that Proposition \ref{prop:A3} applies to $\lin_{\Sigma}$. Hence, we  obtain a solution $v_i$  to the linear equation $\lin_{\Sigma} v_i = - \lin_{\Sigma} \bar u'_i$ that vanishes on $s=5 \delta_s/\tau$, is given up to a constant on $s=1$, and has exponential decay on each asymptotic grim reaper.

We fix now once and for all $\alpha \in (0,1)$, and take  $v_i$ to be the solution described above pushed forward to $\Sigma$. More precisely, $v_i := \ubar{\rcal} ( 0, -\lin_{\Sigma} \bar u_i')$, with $\ubar{\rcal}$ as in Proposition \ref{prop:A3}, on each component of $\Sigma_{\geq 1}[T, \uvarphi, \tau]$. We define 
	\[
	\bar u_i= \bar u_i' + (\psi[1,2] \circ s) v_i, \quad \bar w_i = \lin_{\Sigma} \bar u_i.
	\]
\begin{lemma}
\label{lem:4-17}
The functions $\bar u_i$ and $\bar w_i$ satisfy the following properties:
	\begin{enumerate}
	\item They depend continuously on $(T, \uvarphi, \tau)$.
	\item $\bar w_i$ is supported on $\Sigma_{\leq 2}$, $\| \bar w_i: C^{0,\alpha} (\Sigma)\| \leq C$, and $\| \bar w_i:C^{0,\alpha}(\Sigma_{\geq 1})\| \leq C\tau$.
	\item $\| \bar u_i: C^{2,\alpha}(\Sigma) \| \leq C$.
	\item $\lin_{\Sigma} \bar u_i =0$ on $\Sigma_{\geq 2}[T,\uvarphi, \tau]$, $\bar u_i =0$ on $\pd \Sigma$ and $|\bar u_i -(a+2)\delta_{ij}|\leq C a \eps$ on  the component of $\pd \Sigma_{\leq 2}$ contained in the $i$th wing.
	\end{enumerate}
\end{lemma}

\begin{proof}
By definition of $\bar u_i$ and $\bar w_i$, it is immediate that $\bar w_i$ is supported on $\Sigma_{\leq 2}$, and (i) follows from the construction. From Corollary \ref{cor:4-12} and Proposition \ref{prop:A3}, we have 	
	$
	\| v_i : C^{2,\alpha} (\Sigma_{\geq 1} [T,\uvarphi, \tau], g,e^{-\gamma s}) \| \leq C\tau. 
	$
Since the metrics $g$ and $g_{\Sigma}$ are equivalent, 
	\begin{equation}
	\label{eq:psi-v}
	\| (\psi[1,2] \circ s) v_i : C^{2,\alpha} (\Sigma_{\geq 1} [T,\uvarphi, \tau], g_{\Sigma},e^{-\gamma s} )\| \leq C\tau. 
	\end{equation}
Corollary \ref{cor:4-12} and  $\bar w_i = \lin_{\Sigma} \bar u'_i + \lin_{\Sigma} ( (\psi[1,2] \circ s) v_i)$ imply $\| \bar w_i:C^{0,\alpha}(\Sigma_{\geq 1})\| \leq C\tau$. 

The function $\bar u'_i=Y\cdot \nu_{\Sigma}$ is bounded in $C^1$ norm from Lemma \ref{lem:4-8}. The rest of part  (ii) and part (iii) follow from \eqref{eq:psi-v}.  

The function $\bar u_i$ vanishes on the boundary $\pd \Sigma$ because both $\bar u_i'$ and $v_i$ are zero on the boundary by Lemma \ref{lem:4-1} and by the definition of $v_i$ respectively, so the first assertion of (iv) is proved. 

Finally, we estimate $|\bar u_i - (a+2) \delta_{ij}|$ at $s=2$. The function $(\psi[1,2] \circ s) v_i$ brings at most a contribution of order $\tau$, which can be chosen to be smaller than $\eps$, so we only need to study  $|\bar u_i' - (a+2) \delta_{ij}|$ at $s=2$. The curvature of the grim reaper is of order $\tau$, therefore, it suffices to approximate the behavior of $\bar u_i'$ up to first order, which means we can consider planes instead of grim reapers. Given $(T', \uvarphi', \tau)$, the position of the point on the asymptotic plane corresponding to $s=2$ is given by 
	$
	(x_j',y_j') = ((a+2) \cos \beta_j', (a+2) \sin \beta_j') + b_{\theta(T')} (-\sin \beta_j', \cos \beta_j').
	$
 Differentiating with respect to $\varphi_i'$, using (iii) in Lemma \ref{lem:4-1} and (vi) from Proposition \ref{prop:scherk}, we get
	 \[
	 \left | \frac{\pd }{\pd \varphi_i'} (x_j',y_j') - (a+2) (-\sin \beta_j', \cos \beta_j') \delta_{ij}\right| \leq C\tau + C \eps a.
	 \]
We can approximate the position $F_j(T',\uvarphi- \uvarphi' , \tau)\mid_{s=2}$ by $(x_j', y_j')$ and the normal $\nu_{\Sigma}$ by $(-\sin \beta_j', \cos \beta_j')$ committing an error of order $\eps$  so we have the desired result for $\tau$ small enough. 
\end{proof}

\subsection{The functions $w_1$ and $w_2$.}

In this section, we describe how the unbalancing generates functions $w_1$ and $w_2$ close to being in the approximate kernel of  $\lin_{\Sigma}$.
\begin{definition}
Let $H_{\phi}$ be the mean curvature on the surface $Z_1[\phi] (\Sigma(\theta(T)))$ and let $w_1: \Sigma(\theta(T)) \to \mathbf R$ be defined by 
	\[
	w_1 := \left. \frac{d}{d\phi} \right|_{\phi=0} H_{\phi} \circ Z_1[\phi].
	\]
\end{definition}

We also denote by $w_1$ the pushforward  to $\Sigma[T, \uvarphi, \tau]$  by $Z[T,\uvarphi, \tau]$ of the  function above. Similarly, we define $w_2$ to be   $\left. \frac{d}{d\phi} \right|_{\phi=0} H_{\phi} \circ Z_2[\phi]$ and its pushforward by $Z[T,\uvarphi, \tau]$ to $\Sigma[T, \uvarphi, \tau]$. 

\begin{lemma}
\label{lem:w}
The functions $w_1$ and $w_2$ depend continuously on $(T, \uvarphi, \tau)$. They are supported on $\Sigma_{\leq 0}$ and $\| w_j: C^0 (\Sigma) \| \leq C$, for $j=1,2$. 

Let $V$ be the span of  the pushforwards by $Z[T,\uvarphi, \tau]$ of the functions $\vec e_x \cdot \nu$ and $\vec e_y \cdot \nu $ on $\Sigma(\theta(T))$ and let us define $P: L^2(\Sigma[T,\uvarphi, \tau], |A_{\Sigma}|^2 g_{\Sigma} /2) \to V$ to be the orthogonal projection of functions  onto $V$. Every function in $V$ is the projection of a linear combination of $w_1/|A_{\Sigma}|^2$ and $w_2/|A_{\Sigma}|^2$. More precisely, for every $(\mu_1,\mu_2) \in \R^2$, there is a pair $(\eta_1,\eta_2) \in \R^2$ such that 
	\[
	P((\eta_1 w_1 + \eta_2 w_2)/|A_{\Sigma}|^2 ) = \mu_1 \vec e_x \cdot \nu+ \mu_2\vec e_y \cdot \nu,
	\]
with $| (\eta_1,\eta_2)| \leq C( |(\mu_1,\mu_2)|)$.

\end{lemma}

\begin{proof}
The assertions about the continuous dependence and the supports of $w_1, w_2$ are easily derived from the definitions of $Z_1$, $Z_2$, $w_1$ and $w_2$. 

To prove the second part, we recall the following balancing formula from \cite{korevaar-kusner-solomon} and sketch its proof. 

\begin{lemma} 
\label{lem:balancing}
If $S$ is a surface with four Scherk ends, then $\int_S  H \nu d\mu = \int_{\partial S} \vec n$, where $\vec n$ is the unit outward conormal vector at the boundary  $\pd S$. 
\end{lemma}

If $S$ is a surface with boundary embedded in $E^3$ and $Y$ is a vector field that perturbs $S$, the change in area of $S$ is given by 
	\begin{align*}
	\delta (\textrm{area} (S)) &= \int_S \textrm{div} Y = \int_S \textrm{div} (Y^{\parallel}) + \textrm{div} (Y^{\perp}),\\
				&= \int_{\partial S} \vec n \cdot Y - \int_S H \nu \cdot Y
	\end{align*}
where $Y^{\perp} = Y \cdot \nu$ and $Y^{\parallel} = Y- Y^{\perp}$ are the orthogonal and parallel components of $Y$, and $\vec n$ is the unit outward conormal at $\pd S$. If $Y=\vec e_x, \vec e_y$, or $\vec e_z$, the area of $S$ does not change, so $\int_S H  \nu = \int_{\partial S} \vec n $. 

If the surface $S$ is one period of  $\Sigma_{\leq 0}$, we have $\int_{\pd S} \vec n = 2 \pi \sum_{i=1}^4 v_i$, where the $v_i$'s are the vectors of the tetrad $T$. Differentiating the balancing formula, we obtain  
	\[
	\frac{\pd}{\pd \phi} \int_{\Sigma_{\leq 0}/G'}  H_{\phi} \circ Z_1[\phi] \nu d\mu = 2 \pi \frac{\pd}{\pd \phi}\sum_{i=1}^4 v_i.
	\]
Under $Z_1(\phi)$, the first and third wings rotate by an angle $\phi$ towards the fourth quadrant,  so $\frac{\pd}{\pd \phi}\sum_{i=1}^4 v_i=(2 \sin \theta(T), -2 \cos \theta(T))$. Similarly, the rate of change of the sum of the conormals under $Z_2[\phi]$ is $(2 \sin \theta(T), 2 \cos \theta(T))$. The two vectors $(2 \sin \theta(T), \pm 2 \cos \theta(T))$ are linearly independent since  $\theta(T) \in (20\delta_{\theta}, \pi/2 -20 \delta_{\theta})$ in \eqref{eq:require-T}, hence  the norm of $P^{-1}$ is bounded by a constant $C$ depending on $\delta_{\theta}$. 
\end{proof}


\subsection{The decomposition of $H_{\Sigma}-\tau \vec e_y \cdot \nu_{\Sigma}$}


\begin{prop}
\label{prop:4-20}
For $T, \uvarphi, \tau$ as in Lemma \ref{lem:4-1},
	\[
	\| H_{\Sigma}- \tau \vec e_y \cdot \nu_{\Sigma} - \sum_{j=1}^2 \theta_{j,\Sigma} w_j - \sum_{i=1}^4 \varphi_i \bar w_i: C^{0,\alpha} (\Sigma,g_{\Sigma}, e^{-\gamma s}) \| \leq C(\tau+|\ubar{\theta}_{\Sigma}|^2 + |\uvarphi|^2),
	\]
where $\ubar{\theta}_{\Sigma}  = (\theta_1(T) + \frac{\varphi_3}{2} -\frac{\varphi_1}{2}, \theta_2(T) + \frac{\varphi_4}{2} -\frac{\varphi_2}{2})$. 	
\end{prop}

\begin{proof} Without loss of generality, we can assume that $\theta_r(T)=0$. On $\Sigma_{\geq 1}$, $w_1, w_2=0$ and the result follows from Lemmas \ref{lem:4-5} and \ref{lem:4-17}.

Let  $T'$ be the tetrad given by Lemma \ref{lem:4-1} when $\uvarphi ' = \uvarphi$,  $\hat T:=\{\vec e(\beta_i + \varphi_i)\}_{i=1}^4$, and let $T_0$ be the balanced tetrad of the  directing vectors of the Scherk surface $\Sigma(\theta(T))$. We have
	\begin{align*}
	H_{\Sigma[T,\uvarphi, \tau]}-\tau \vec e_y \cdot \nu_{\Sigma[T,\uvarphi, \tau]}= I+I\!I+I\!I\!I,
	\end{align*}
with 
	\begin{align*}
		I:= & H_{\Sigma[T,\uvarphi, \tau]}-\tau \vec e_y \cdot \nu_{\Sigma[T,\uvarphi, \tau]} - (H_{\Sigma[ T', \ubar 0, \tau]}-\tau \vec e_y \cdot \nu_{\Sigma[T', \ubar 0,\tau]}), \\
		I\!I:= & H_{\Sigma[ T', \ubar 0, \tau]}-\tau \vec e_y \cdot \nu_{\Sigma[T', \ubar 0,\tau]}- (H_{\Sigma[\hat T, \ubar 0, \tau]}-\tau\vec e_y \cdot \nu_{\Sigma[\hat T, \ubar 0,\tau]} ),\\
		I\!I\!I:= & H_{\Sigma[\hat T, \ubar 0, \tau]}-\tau \vec e_y \cdot \nu_{\Sigma[\hat T, \ubar 0,\tau]}.
	\end{align*}
On $\Sigma_{\leq 1}$, $\bar w_i = \lin_{\Sigma} \bar u_i'$, so $I-\sum_{i=1}^4 \varphi_i \bar w_i$ is of order $|\uvarphi|^2$ by definition of $\bar u'_i$. For $I\!I$, Lemma \ref{lem:4-1} implies $\beta'_i \sim \beta_i + \varphi_i$, with an error of order $\tau$. The smooth dependence of the construction on $T$ gives $\|I\!I : C^{0,\alpha}(\Sigma_{\leq 1}) \| \leq C\tau$. 
For $I\!I\!I$, note that $\theta_{j,\Sigma} = \theta_j(\hat T)$, $j=1,2$ and, by the definition of $w_1, w_2$, 
	\[
	\| H_{\Sigma[\hat T, \ubar 0, \tau]} - H_{[T_0,\ubar 0,\tau]} -\sum_{j=1}^2 \theta_{j,\Sigma} w_j:C^{0,\alpha} (\Sigma_{\leq 1})\| \leq C |\ubar{\theta}_{\Sigma}|^2.
	\]
Using the facts that $H_{[T_0,\ubar 0,\tau]} \equiv 0$ and $\|\tau \vec e_y \cdot \nu_{\Sigma[\hat T, \ubar 0,\tau]}:C^{0,\alpha} (\Sigma_{\leq 1})\| \leq C\tau$, we finish the proof.
\end{proof}


\section{Initial Surfaces}
\label{sec:initial-surfaces}


We construct smooth initial surfaces by fitting the desingularizing surfaces at the vertices of the flexible graph $\overline G$ from Section \ref{sec:initial-conf}. The ends of the grim reapers have to be slightly adjusted, and the wings of the desingularizing surfaces bent to attach them smoothly; we keep track of these variations with the variables $(\tilde{\ubar b}, \tilde{\ubar c})$ and $\tilde{\ubar \varphi}$ respectively.

Let $\{ \wilde \Gamma_n\}_{n=1}^{N_{\Gamma}}$ be a finite family of grim reapers in general position as in Section  \ref{ssec:grimreapers} and recall that $N_I$ is the number of intersection points. To each intersection point $p_k$, $k=1, \ldots, N_I$, we assign a positive integer $m_k$ and define
	\begin{equation}
	\label{eq:tauk}
	 \tau_k:= \frac{\bar \tau}{m_k}, \quad \ubar m:= \{m_k\}_{k=1}^{N_I}.
	\end{equation}
The parameter $\bar \tau$ controls the overall scaling for the surface and the $m_k$'s allow for different scalings at different intersection points:  the desingularizing surface for the $k$th intersection will have a period of $2 \pi/m_k$ in the $z$-direction. We choose $\bar \tau  \in (0, \delta_{\tau})$ small enough so that Proposition \ref{prop:smooth-dependence} applies and the desingularizing surfaces are embedded. Unless otherwise stated, our constants $C$ will depend on $\delta_{\Gamma}$, $N_{\Gamma}$, $\delta$ from Lemma \ref{lem:cor-cond}. 

To each intersection point  corresponds some unbalancing and bending of the wings measured by $\theta_{k,1}, \theta_{k,2}$ and $\{ \varphi_{k,i}\}_{i=1}^4$ respectively. We collect all of the angles for each intersection, along with the perturbations of the left rays for an initial configuration in the next definition.

\begin{definition}
\label{def:vcal}
  Let $\vcal_{\varphi}:= (\R^4)^{N_I}$, $\vcal_{\theta}:= (\R^2)^{N_I}$, $\vcal_{b,c} := (\R^2)^{N_{\Gamma}}$, and $\vcal := \vcal_{\theta} \times \vcal_{b,c} \times \vcal_{\varphi}$. We fix $\zeta>0$, which will be determined later in the proof of Theorem \ref{thm:8-2} and define 
	\[
	\Xi_{\vcal}:=\{ \xi \in \vcal: |\xi| \leq \zeta \bar \tau\}.
	\]
\end{definition}

We fix now $\xi = (\ubar \theta,\bar  \tau{\ubar b}', \bar \tau{ \ubar c}', \uvarphi) \in \Xi_{\vcal}$, where 
	\begin{itemize}
	\item $\ubar \theta = \{ \theta_{k,1}, \theta_{k,2}\}_{k=1}^{N_I}$ and $(\theta_{k,1},\theta_{k,2})$ determines the dislocation of the tetrad at the $k$th intersection.
	\item ${\ubar  b}' := \{  b_n'\}_{n=1}^{N_{\Gamma}}$ and ${\ubar c}' := \{ c_n'\}_{n=1}^{N_{\Gamma}}$. The point $(b_n + b'_n, c_n +  c'_n)$  is the center of the grim reaper on which the $n$th  left ray of $\overline G$ lies (see Proposition \ref{prop:init-conf}).
	\item $\ubar \varphi :=\{ \{ \varphi_{k,i}\}_{i=1}^4 \}_{k=1}^{N_I}$ relates to the bending of the wings at each intersection.
	\end{itemize}

Given $\bar \tau$ small and extracting $\ubar \theta$ and $(\ubar b', \ubar c')$ from $\xi$, we construct the initial configuration $\overline G(\ubar b', \ubar c', \ubar \theta,\bar \tau)$ as in Section \ref{sec:initial-conf}. 

Let us recall that $\{ \bar v_{ki} \}_{i=1}^4$ are the directing vectors of the edges or rays emanating from the $k$th vertex of $\overline G$. We define the angles $\beta_{ki}$ by the relation $\bar v_{ki} = \vec e[\beta_{ki}]$ and the tetrad $T_k(\xi)$ by 
	\begin{equation}
	\label{eq:Tk}
	T_k(\xi) = \{\vec e[\beta_{ki} + \varphi_{k,i}]\}_{i=1}^4.
	\end{equation}
We take the mapping $\hcal_k$ to be the scaling by $1/m_k$ followed by the translation that sends the origin to $\bar p_k$, the $k$th vertex of $\overline G$.   The surface $\scal_k(\bar \tau, \xi)$, which desingularizes the $k$th intersection,  is defined by  
	\[
	\scal_k(\btau, \xi): = \hcal_k(\Sigma[T_k(\xi), \tilde \uvarphi_k, \tau_k])
	\]
with $\tilde \uvarphi_{k}$ to be determined later and $\tau_k$ as in \eqref{eq:tauk}. The pivots of $\scal_k$ are defined to be the image under $\hcal_k$ of the pivots given in Definition \ref{def:pivot}. Note that they do not depend on $\tilde \uvarphi_k$.

Let us now construct the edges and the rays of the initial surface. For an edge in $\overline G$ connecting the two vertices $\bar p_k$ and $\bar p_k'$, we consider the pivots of $\scal_k$ and $\scal_{k'}$ on the appropriate wings, say the $i$th wing of $\scal_k$ and the $i'$th wing of $\scal_{k'}$. There is a unique  grim reaper $\Gamma_{kk'}$ passing through the two pivots. We choose $\tilde \varphi_{k,i}$ so that the grim reaper asymptotic to the $i$th wing of $\scal_k$ matches  $\Gamma_{kk'}$ and define $\tilde \varphi_{k',i'}$ similarly; in other words, $\tilde \varphi_{k,i}$ is the angle formed by  $\Gamma_{kk'}$ with the grim reaper asymptotic to the $i$th wing of $\hcal_k(\Sigma[T_k(\xi), \underline{0}, \tau_k])$ at the pivot. Now that the relevant $\tilde \varphi$'s are fixed, let us denote the piece of $\Gamma_{kk'}$ between the two pivots by $\ecal'$ and the piece between the boundaries $\pd \scal_k$ and $\pd \scal_{k'}$ by $\ecal''$. We do not assign any index to the notation $\ecal'$ or $\ecal''$ to differentiate the various edges because it is not important for what follows.

If we have a ray emanating from $\bar p_k$ and corresponding to the $i$th wing of $\scal_k$, we just translate the ray so that its boundary matches the corresponding pivot. $\tilde \varphi_{k,i}$ is taken so that the asymptotic grim reaper to the $i$th wing matches the translated ray. We denote the translated ray by $\ncal'$ and the piece of the ray starting at $\pd \scal_k$ by $\ncal''$. 

\begin{definition} 
\label{def:xi}
We define the initial surface $M_{\bar \tau}$, the union of the asymptotic surfaces $M'_{\bar \tau}$, and the union of the desingularizing pieces $\scal_{\bar \tau}$:
	\begin{gather*}
	M=M_{\btau} = M(\bar \tau, \xi) 
	:= \bigcup_{k=1}^{N_I} \scal_k(\btau, \xi) \cup \bigcup \ecal'' \cup \bigcup \ncal'', \\
	M'=M'_{\btau} = M'(\bar \tau, \xi)
	:=  \bigcup \ecal'\cup \bigcup \ncal',\\
	 \scal=\scal_{\btau} = \scal(\bar \tau, \xi) :=\bigcup_{k=1}^{N_I} \scal_k(\btau,\xi).
	\end{gather*}
\end{definition}

The metric $g_M$ on $M_{\btau}$ is the metric induced by the immersion of $M_{\btau}$ in $E^3$.

\begin{definition}
\label{def:s}
We push forward the  function $s$ from Section \ref{ssec:desing-surf} by $\hcal_k$ to each $\scal_k$  and extend it to $M$ in the following way:
	\begin{itemize}
	\item   $s =\max_k ( 5 \delta_s /\tau_k)$ on $\ecal''$. 
	\item The coordinate $s$ of a point $p \in \ncal''$ is $d(p,\pd \ncal'')+5 \delta_s/\tau_k$, where $d$ denotes the distance measured with respect to arclength on $\ncal''$, and  $k$ is such that $\scal_k\cap \ncal'' \neq \emptyset$.
	\end{itemize}
	\end{definition}
The function $s$ is continuous at $\scal_{\btau} \cap \ncal''$ but may be discontinuous at $\scal_{\btau} \cap \ecal''$.

To keep track of the  unbalancing of $M$, the displacement of the asymptotic grim reapers, and the bending of $M$, we define 
	\begin{itemize}
	\item $\ubar{\tilde\theta} = \{ \tilde \theta_{k,1}, \tilde \theta_{k,2}\}_{k=1}^{N_I} \in \vcal_{\theta}$ where $\tilde \theta_{k,1} = \theta_1(T_k(\xi))$ and $\tilde \theta_{k,2} = \theta_2(T_k(\xi))$,
	\item $(\tilde {\ubar b}, \tilde {\ubar c}) \in \vcal_{b,c} $ where $(b_n + b_n' +\tilde b_n, c_n + c_n'+\tilde c_n)$ is the position of the center of the grim reaper supporting the $n$th left ray. 
	  \item  $\tilde {\ubar \varphi} \in \vcal_{\varphi} $ as described earlier. 
	\end{itemize}
The proposition below follows from the construction and Proposition \ref{prop:init-conf}. 
		\begin{prop}
		\label{prop:init-surf}
	The surface $M_{\btau}$ constructed above is well defined and has the following properties:
		\begin{enumerate}
		\item $M_{\btau}$ is a complete smooth surface which depends smoothly on $\xi$.
		\item $M_{\btau}$ is periodic in $z$; more precisely, it is invariant under the group generated by the translation $z \to z+2 \pi$.
		\item $M_{\btau}$ is invariant under reflections across the $xy$-plane.
		\item There is a large ball $B\subset \R^2$ such that $M_{\btau}\setminus (B\times \R)$ is the union of grim reaper ends that are in one-to-one correspondence with the ends of our  rescaled initial family $\{ \Gamma_n\}_{n=1}^{N_{\Gamma}}$ and such that  corresponding centers of the left rays differ by $(b'_n + \tilde b_n, c_n' + \tilde c_n)$. 
		\item We have $\tilde \theta_{k,j} =\theta_{k,j}+(\varphi_{k,j+2} - \varphi_{k,j})/2$ for $k=1, \ldots, N_I$ and $j=1,2$. 
		\item $|(\tilde {\ubar b}, \tilde {\ubar c})| \leq C$ and $|\tilde \uvarphi + \uvarphi| \leq C \bar \tau$.
		\item $\hcal_k^{-1}(M_{\btau})$ converges uniformly in $C^j$ norm, for any $j <\infty$, on any compact subset of $E^3$ to a Scherk surface as $\bar \tau \to 0$, for  $k=1, \ldots, N_I$.
		\end{enumerate}
	\end{prop}

	\begin{cor}
	\label{cor:sigma-parameters}
	For $\btau$ small enough so that $\zeta \btau<\delta_{\theta}$, the parameters of $\Sigma[T_k,\tilde{\uvarphi}_k, \tau_k] = \hcal_k^{-1}(\scal_k)$ satisfy the following estimates,
	\begin{enumerate}
	\item $ 25 \delta_{\theta} \leq \theta(T_k) \leq \frac{\pi}{2} - 25 \delta_{\theta}$.
	\item $|\tilde \theta_{k,j}|= |\theta_j(T_k)| \leq 3 \zeta \btau$, $j=1,2$.
	\item $|\tilde{\uvarphi}_k| \leq (\zeta + C)\btau$.
	\end{enumerate}
	\end{cor}

\begin{proof}
This corollary follows from the construction,  Definition \ref{def:delta-theta},   Definition \ref{def:vcal}, and the proof of Proposition \ref{prop:init-conf}.
\end{proof}

As mentioned in the introduction, we divide the study of the linear operator to various pieces of the initial surface, then use cut-off functions and an iteration to solve the equation $\lin v =f$ on the whole surface $M_{\btau}$. In the definition of $M_{\btau}$, the $\scal_k $'s, $\ecal''$'s, and $\ncal''$'s intersect only on the boundary of the $\scal_k$'s. In order to use cut-off functions, we define neighborhoods $\ecal$ and $\ncal$ of $\ecal''$ and $\ncal''$ to have some overlap with $\scal_k$.

\begin{definition}
\label{def:ubara}
Let $\ubar a: = 8 |\log \bar \tau|$. We consider
	\begin{itemize}
	\item $\ncal$ to be  the connected component of $M(\btau, \xi)_{\geq \ubar a}$ that contains $\ncal''$. 
	\item $\ecal$ to be the connected component of $M(\btau, \xi)_{\geq \ubar a}$ that contains $\ecal''$.
	\end{itemize}
\end{definition}

There are $2 N_{\Gamma}$ grim reaper ends, and we denote each of them by $\ncal''$, without any index to differentiate them because they are treated in the same manner in the rest of the article. Therefore, $\ncal''$ is a generic end and $\ncal$ is a neighborhood of $\ncal''$. The same remark applies to an edge $\ecal''$ and its neighborhood $\ecal$ as well.


\section{Linear Operator}
\label{sec:lin-op}


We study the linear operator $\lin = \lin_M:= \Delta_{g_M} + |A_M|^2 + \btau \vec e_y \cdot \nabla$ associated to normal perturbations of $H- \bar \tau \vec e_y \cdot \nu$. The strategy is to first solve the Dirichlet problem associated to $\lin$ with vanishing boundary conditions on the various pieces $\scal_k$, $\ecal$ and $\ncal$. Since $\ecal$ and $\ncal$ are close to being flat, standard theory for the Laplace operator  gives us insight into the Dirichlet problems on $\ecal$ and $\ncal$. On a desingularizing surface $\Sigma$,  Proposition \ref{prop:7-1} shows the existence of an exponentially decaying solution to $\lin v =E$ with vanishing boundary conditions if $E$ is exponentially decaying. We use the functions $w_1$ and $w_2$ to handle the approximate kernel in Lemma \ref{lem:7-4}, then the functions $\{\bar w_i\}_{i=1}^4$ for the asymptotic decay. 
 
We are then ready to solve the equation $\lin v =E$ on the whole surface $M$. We first restrict the support of $E$ to the desingularizing surfaces using a cut-off function $\psi$ and find a solution $u$ to $\lin u = \psi E$ on each $\scal_k$. Extending  $u$ by zero to the rest of $M$ may not produce a smooth function, therefore we need to cut off $u$ also. We then solve the Dirichlet problem $\lin u'' = E - \lin (\psi u)$ on $\ncal$ and $\ecal$ with appropriate boundary conditions. Adding $\psi u$ to a cut-off version of $u''$ gives us a first approximate solution. The error created by the cut-off functions is small  compared to the inhomogeneous term therefore we can iterate to obtain a  sequence of functions converging to an exact solution. The exponential decay of the solutions on the desingularizing surfaces plays an essential role in controlling the error.


\subsection{Linear operator on $\ecal$ and $\ncal$.}
\label{ssec:ecal-ncal}


The domains $\ecal$ and $\ncal$ are graphs of small functions over grim reapers. The grim reapers are isometric to  planes, so $\ecal$ and $\ncal$ are close to being isometric to  planes with an error controlled by Corollary \ref{cor:4-4}. Moreover, $\ecal/G'$ is bounded, where $G'$ is the group generated by $(s,z) \to (s, z + 2\pi)$, so the existence of a unique solution to the  Dirichlet problem with vanishing boundary conditions on $\pd \ecal/G'$ is a standard result from elliptic theory. 
 
The operator $\lin$ on $\ncal$ is a perturbation of the Laplace operator on flat cylinders, which has been well studied. Any result for $\lin$ on $\ncal$ follows directly from a similar one for the Laplace operator on cylinders.

\begin{definition}
\label{def:cylinder}
 We define $(\Omega,g_0)$ to be the cylinder $\Omega=H^{+}_{\leq l}/G'$ equipped with the standard metric $g_0=ds^2 +dz^2$, where $G'$ is the group generated by $(s,z) \to (s,z+2\pi)$, and $l \in (10, \infty)$ is called the length of the cylinder. We have $\pd \Omega=\pd_0 \cup \pd _l$ where $\pd_0$ and $\pd_l$ are the boundary circles $\{ s=0\}$ and $\{s=l\}$ respectively. 
\end{definition}

Let $\lin$ denote an operator on $\Omega$ of the form
	\[
	\lin v = \Delta_{\chi} v + \mathbf{V} \cdot \nabla v+ d v,
	\]
where $\chi$ is a $C^2$ Riemannian metric, $\mathbf{V}$ a $C^1$ vector field, and $d$ a $C^1$ function on $\Omega$. For $\ubar c>0$, we define 
	\begin{equation*}
	N(\lin) := 
	\| \chi -g_0:C^2(\Omega,g_0, e^{-\ubar c s}) \|   + \| \mathbf{V} :C^1(\Omega,g_0, l^{-1})\|  + \| d :C^1(\Omega,g_0,e^{-\ubar c s} + l^{-2})\|.
	\end{equation*}
If the inhomogeneous term $E$ is exponentially decaying,  there is an exponentially decaying solution to the Dirichlet problem $\lin v = E$ in $\Omega$, $v = 0$ on $\pd_l$, and given up to a constant on $\pd_0$.  The following result is a generalization of Proposition A.3 in \cite{kapouleas;embedded-minimal-surfaces}, with an  added  term  $\mathbf V \cdot \nabla$. Its proof is postponed to the Appendix \ref{sec:linop-cylinders}.

\begin{prop}
\label{prop:A3}
Given $\gamma \in (0,1)$ and $\eps>0$, if $N(\lin)$ is small enough in terms of $\ubar c$, $\alpha$, $\gamma$ and $\eps$ (but independently of $l$), 
there is a bounded linear map 
	\[
	\ubar{\mathcal R}: C^{2,\alpha}(\pd_0,g_0) \times C^{0, \alpha}(\Omega, g_{0}, e^{-\gamma s}) \to C^{2, \alpha}(\Omega, g_{0}, e^{-\gamma s})
	\]
such that for $(f,E)$ in $\ubar{\mathcal R}$'s domain and $v = \ubar{\mathcal R}(f,E)$, the following properties are true, where the constants $C$ depend only on $\alpha$ and $\gamma$:
	\begin{enumerate}
	\item   $\lin v = E$ on $\Omega$.
	\item  $v= f -\textrm{\em avg}_{\pd_0} f +B(f,E)$ on $\pd_0$, where $B(f,E)$ is a constant on $\pd_0$ and $\textrm{\em avg}_{\pd_0} f$ denotes the average of $f$ over $\pd_0$. 
	\item  $v\equiv 0$ on $\pd_l$.
	\item $
	\| v :C^{2, \alpha}(\Omega, g_{0}, e^{-\gamma s}) \| $\\
	$\qquad \qquad \leq C \| f-\textrm{\em avg}_{\pd_0} f:C^{2, \alpha}(\pd_0, g_{0})\|+ C \| E: C^{0, \alpha}(\Omega, g_{0}, e^{-\gamma s})\|.
	$
	\item $|B(f,E)|\leq \eps\| f-\text{\em avg}_{\pd_0} f: C^{2, \alpha}(\pd_0, g_{0})\|+ C \| E: C^{0, \alpha}(\Omega, g_{0}, e^{-\gamma s})\|.
	$
	\item If $v' \in C^2(\Omega)$ satisfies $\lin v'=E$ on $\Omega$, and $v=v'$ on $\pd \Omega$, then $v'=v$ on $\Omega$. Moreover, if  $E$ vanishes, then 
	\[
	\| v : C^0(\Omega) \| \leq 2 \| v:C^0(\pd_0)\|.
	\]
	\end{enumerate}
\end{prop}

A similar result is true for infinite cylinders, once  the terms involving $l^{-1}$ and $l^{-2}$ are removed from the definition of $N(\lin)$. 
\begin{prop}
\label{prop:A4}
Given $\gamma \in (0,1)$ and $\eps>0$, if $N(\lin)$ is small enough in terms of $\ubar c$, $\alpha$, $\gamma$ and $\eps$ (but independently of $l$), there is a bounded linear map 
	\[
	\ubar{\mathcal R}: C^{0, \alpha}(\Omega, g_{0}, e^{-\gamma s}) \to C^{2, \alpha}(\Omega, g_{0}, e^{-\gamma s})
	\]
such that for $E$ in $\ubar{\mathcal R}$'s domain and $v = \ubar{\mathcal R}(E)$, the following properties are true, where the constants $C$ depend only on $\alpha$ and $\gamma$:
	\begin{enumerate}
	\item   $\lin v = E$ on $\Omega$.
	\item  $v=B(E)$ on $\pd \Omega$, where $B(E)$ is a constant.
	\item $
	\| v :C^{2, \alpha}(\Omega, g_{0}, e^{-\gamma s}) \| \leq C \| E: C^{0, \alpha}(\Omega, g_{0}, e^{-\gamma s})\|.
	$
	\end{enumerate}
\end{prop}
	\begin{cor}
	\label{cor:lin-op-ncal}
	Given $E \in C^{0,\alpha}(\ncal, g_{M}, e^{-\gamma s})$, there is a unique function $v \in C^{2,\alpha}(\ncal, g_M,e^{-\gamma s})$ such that $\lin v =E$, $v$ is a constant on $\pd \ncal$ and 
		\[
		\| v: C^{2,\alpha}(\ncal, g_M,e^{-\gamma s}) \| \leq C\| E: C^{0,\alpha}(\ncal,g_M, e^{-\gamma s})\|.
		\]
	\end{cor}
\begin{proof}
By Corollary \ref{cor:4-4}, we can apply Proposition \ref{prop:A4} to $\lin=\lin_M$ on $\ncal$ if $\eps$ and $\delta_s$ are small enough. The uniqueness of the solution $v \in C^{2,\alpha}(\ncal, g_M, e^{-\gamma s})$ follows from (iii) in Proposition \ref{prop:A4}.
\end{proof}


\subsection{Linear operator on $\scal_k$.}


We now prove that we can solve a Dirichlet problem with vanishing boundary data on $\scal_k$, modulo  linear combinations of $w$'s and $\bar w$'s. Let us fix $\scal_k$, $\tau  =\tau_k$, and consider the surface $\Sigma= \Sigma[T_k,\tilde{\uvarphi}_k, \tau] = \hcal_k^{-1} (\scal_k)$, the parameters of which are controlled by Corollary \ref{cor:sigma-parameters}. In this section,  the linear operator is $\lin_{\Sigma} := \Delta_{g_\Sigma}+|A_{\Sigma}|^2 + \tau_k \vec e_y \cdot \nabla$, where $\tau_k$ replaces $\btau$.

\begin{prop}
\label{prop:7-1}
Given $E \in C^{0,\alpha}(\Sigma)$, there are $\ubar{\theta}_E := \{ \theta_{E,j}\}_{j=1}^2 \in \R^2$, $ \uvarphi_E : = \{\varphi_{E,i} \}_{i=1}^4 \in \R^4$ and $v_E \in C^{2,\alpha}(\Sigma)$ such that:
	\begin{enumerate}
	\item $\ubar{\theta}_E, \uvarphi_E$ and $v_E$ are uniquely determined by the construction below. 
	\item $\lin_{\Sigma} v_E = E + \sum_{j=1}^2 \theta_{E,j} w_j + \sum_{i=1}^4 \varphi_{E,i} \bar w_i$ on $\Sigma$ and $v_E=0$ on $\pd \Sigma$.
	\item $|\ubar \theta_E| \leq C \| E\|$, where $\|E :C^{0,\alpha}(\Sigma, g_{\Sigma}, e^{-\gamma s/m_k} ) \| $. 
	\item $| \uvarphi_E| \leq C \| E\|$.
	\item $ \| v_E: C^{2,\alpha} (\Sigma, g_{\Sigma}, e^{-\gamma s/m_k})\| \leq C \| E\|$.
	\end{enumerate}
\end{prop}

We first reduce to the case when $E$ is supported on $\Sigma_{\leq 2}$. 

\begin{lemma}
\label{lem:7-2}
If Proposition \ref{prop:7-1} is valid when $E$ is supported on $\Sigma_{\leq 2}$, it is valid in general. 
\end{lemma}

\begin{proof}

We start on the first  wing. 
Let $\Omega = [1, 5 \delta_s /\tau] \times \R/G'$, where $G'$ is the group generated by $(s, z) \to (s, z+2 \pi)$. By equation \eqref{eq:kappa}, the  asymptotic grim reaper $\kappa(H^+_{\geq 1})$ is isometric to $\Omega$.  We consider the operator $\lin_{\Sigma}  = \Delta_{g_{\Sigma}} + |A_{\Sigma}|^2 + \tau \vec e_y \cdot \nabla$ on $\kappa(H^+_{\geq 1})$, where $g_{\Sigma}$ and $A_{\Sigma}$ are as in Notations \ref{ssec:notations2}.

By Corollary \ref{cor:4-4}, we can apply Proposition \ref{prop:A3} to $ \lin_{\Sigma}$ to obtain $v_1 = \ubar {\mathcal R} (0,  E)$ and extend $v_1$, possibly discontinuously, by $0$ on the rest of $\Sigma$. The functions $v_2, v_3$ and $v_4$ are defined similarly for the three other wings. With the usual abuse of notation, we denote their pushforwards on $\Sigma$ by $v_i$ as well. Using the estimates of Proposition \ref{prop:A3} and the fact that $g$ and $g_{\Sigma}$ are equivalent metrics, we get
	\begin{equation}
	\label{eq:v71}
	\| v_i : C^{2,\alpha} (\Sigma, g_{\Sigma}, e^{-\gamma s/m_k} ) \| \leq C \| E\|.
	\end{equation}
On $\Sigma$, we define
	\[
	E' = E - \lin_{\Sigma} (\psi [1,2] \circ s (v_1 + v_2 + v_3 + v_4)).
	\]
Clearly, $E'$ is supported on $\Sigma_{\leq 2}$. We can apply Proposition \ref{prop:7-1} with  $E'$ to obtain $v_{E'}, \ubar \theta_{E'}$ and $\uvarphi_{E'}$, then define $\ubar \theta_E = \ubar \theta_{E'}$, $\uvarphi_E = \uvarphi_{E'}$, and $v_E=v_{E'}+(\psi [1,2] \circ s) (v_1 + v_2 + v_3 + v_4)$. The required estimates are valid thanks to \eqref{eq:v71}.
\end{proof}

In the definition below, the sole purpose of the small constant $\eps_h$ is to ensure that the metric $h$ is non degenerate.

\begin{definition} Let us define a metric $h :=(\frac{1}{2}|A_{\Sigma}|^2 + \eps_h) g_{\Sigma}$, where $\eps_h >0$ is small depending on $\tau$ and $\delta_{\theta}$, 
 and will be determined in the proof of Lemma \ref{lem:7-4}. We also define the $c$-approximate kernel of $L_h:= \Delta_h + 2 |A_{\Sigma}|^2 /(|A_{\Sigma}|^2 + 2\eps_h)$ to be the span of the eigenfunctions of $L_h$ with corresponding eigenvalues in $[-c,c]$. 

\end{definition}

\begin{lemma}
\label{lem:7-4}
There are positive constants $C$ and $c$ such that given $E \in L^2(\Sigma/G', h)$, there is $\ubar \theta_E = (\theta_{E,1}, \theta_{E,2})$ such that $(E- \sum_{j=1}^2 \theta_{E,j} w_j)/(|A_{\Sigma}|^2 + 2\eps_h)$ is $L^2(\Sigma/G', h)$-orthogonal to the $c$-approximate kernel and 
	\[
	|\ubar \theta_E | \leq C \| E/(|A_{\Sigma}|^2 + 2\eps_h) : L^2 (\Sigma/G',h)\|.
	\]
\end{lemma}
The reader familiar with the reference can see that this is Lemma 7.4 from \cite{kapouleas;embedded-minimal-surfaces}. We  sketch the main ideas, paraphrasing the proof from \cite{kapouleas;embedded-minimal-surfaces} in Appendix \ref{sec:proof} for the sake of completeness.

\begin{proof}[Proof of Proposition \ref{prop:7-1}]
We can assume that $E$ has support in $\Sigma_{\leq 2}$ and, by the smooth dependence on the parameters in Proposition \ref{prop:smooth-dependence}, we have uniform control over the geometry of $\Sigma_{\leq 2}$, hence 
	\[
	\| E/(|A_{\Sigma}|^2 + 2 \eps_h) : L^2 (\Sigma/G',h)\| \leq C \| E \|.
	\]
We apply Lemma \ref{lem:7-4} to obtain $c$ and $(\theta'_1, \theta'_2)=(\theta_{E,1}, \theta_{E,2})$. The Dirichlet problem  
	\begin{equation}
	\label{eq:L-h}
	L_h v_E'' = (E - \sum_{j=1}^2 {\theta'_j} w_j)/(|A_{\Sigma}|^2/2+\eps_h) \textrm{ in } \Sigma,\quad 
	 v_E'' = 0 \textrm{ on } \pd \Sigma,
	\end{equation}
 can be solved by using the Lax-Milgram Theorem \cite{gilbarg-trudinger;elliptic-equations}. 
Equation \eqref{eq:L-h} is equivalent to 
	\[
	L_{\Sigma} v_E'' = \Delta_{\Sigma} v_E'' + |A_{\Sigma}|^2 v_E''= E - \sum_{j=1}^2 {\theta'_j} w_j. 
	\]
The non-homogeneous term is $C^{0,\alpha}$, so elliptic regularity, Lemma \ref{lem:7-4} and the control of the geometry of $\Sigma_{\leq 2}$  imply that the solution $v''_E$ is unique and satisfies 
	\[
	\| v_E'': C^{2,\alpha} (\Sigma_{\leq 2}, g_{\Sigma})\| \leq C \| E \|.
	\]

The operator $\lin _{\Sigma}$ is a perturbation of $L_{\Sigma}$ so we can use an iteration process similar to the one in Section 4.1 \cite{mine;tridents} to obtain  $v'_E$ and $\ubar \theta_E = \{ \ubar \theta_{E,j} \}_{j=1}^2$ satisfying
	\begin{gather*}
	\lin_{\Sigma} v'_E = E- \sum_{j=1}^2 {\theta_{E,j}} w_j,\\
	\| v_E': C^{2,\alpha} (\Sigma_{\leq 2}, g_{\Sigma})\| \leq C \| E \|, \quad |\ubar \theta_E| \leq C \| E \|.		\end{gather*}
The solutions $v'_E$ and $\ubar \theta_E$ are unique by construction. 

We now arrange the exponential decay of $v'_E$. Define 
	\[
	v_E = v_E' + \sum_{i=1}^4 \varphi_{E,i} \bar u_i,
	\]
where the $\bar u_i$'s are as in Section \ref{ssec:estimates} and the constants $\varphi_{E,i}$'s will be chosen below. We set up exactly as in Lemma \ref{lem:7-2} and consider the restrictions of  $v_E'$ and $\bar u_i$'s to the component of  $\Sigma_{\geq 2}$ on the $j$th wing pulled back to the asymptotic grim reaper. Let $\Omega$ be a cylinder of length $(5 \delta_s /\tau -2)$. For $i,j=1,\ldots, 4$, we define 
	\[
	a_j = \text{avg}_{\pd_0} v_E' - B(v_E',0),\quad a_{i,j} = \text{avg}_{\pd_0} \bar u_i - B(\bar u_i,0),
	\]
where $\pd_0$, $B$, and $\textrm{avg}_{\pd_0}$ are as in Proposition \ref{prop:A3}. From the uniqueness of the solution in (vi) of Proposition  \ref{prop:A3}, we have $v_E = \ubar{\mathcal R}(v_E,0)$ if and only if 
	\[
	a_j + \sum_{i=1}^4 \varphi_{E,i} a_{i,j}  =0, \quad j=1,\ldots, 4.
	\]
By (v) in Proposition \ref{prop:A3} and (iv) in Lemma \ref{lem:4-17}, we can solve the above system and find unique solutions $\varphi_{E,i}$'s. With such a choice of $\varphi_{E,i}$'s, the function $v_E$ has exponential decay and 
	\begin{gather*}
	|\uvarphi_E | \leq C \|v'_E:C^{0,\alpha} (\Sigma_{\leq 2}, g_{\Sigma})\|  \leq C \|E\|, \\
	\| v_E: C^{2, \alpha}(\Sigma, g_{\Sigma}, e^{-\gamma s/m_k}) \| \leq C \|E \|. \qedhere
	\end{gather*}
\end{proof}

%
%


\subsection{Linear operator on $M$}


Let us recall that we defined $M=M_{\btau}$ in the larger scale, where the mean curvature of the asymptotic grim  reapers is of order $\btau$. 
\begin{definition}
Given $v \in C^{r,\alpha}(M)$, $r =0,2$,  we define the norms
	\[
	\| v \|_r : = \| v:C^{r,\alpha}(M,g_M, e^{-\gamma s} )\|,
	\]
where $s$ is the coordinate given in Definition \ref{def:s}.
\end{definition}

As in Section \ref{sec:initial-surfaces}, we denote the scaling used to fit the desingularizing surface in place of the $k$th intersection by 
	\begin{equation}
	\label{eq:norm-equivalent}
	\hcal_k: \Sigma[T_k, \tilde \uvarphi_{k}, \tau_k] \to \scal_k.
	\end{equation}
Note that $\hcal_k^{\ast} g_M = \frac{1}{m_k^2} g_{\Sigma}$. Moreover, for any  function $v\in C^{r,\alpha}(M)$ supported on $\scal_k$, we have
	\[
	 \frac{1}{C} \| v\|_r \leq  \| \hcal_k^{\ast} v : C^{r,\alpha}(\Sigma[T_k, \tilde \uvarphi_{k}, \tau_k], g_{\Sigma}, e^{-\gamma  s/m_k}) \|  \leq 
C \|  v \|_r ,
	 \]
where the constant $C$ depends only on $r$, $\alpha$ and $m_k$. Since $\{m_k\}_{k=1}^{N_I}$ is a finite set, we will stop mentioning the explicit dependence in $m_k$ and incorporate it in the constants $C$ in the rest of the article.

Let us fix $k \in \{ 1, \ldots, N_I\}$. Given a function  $E$ on $M$ with support in $\scal_k$, we define $E' = \frac{1}{m_k} E \circ \hcal_k$.  By Proposition \ref{prop:7-1}, there exist a function $v'$ and constants $\{\theta_{E',j}\}_{j=1}^2$, $\{ \varphi_{E',i}\}_{i=1}^4$ for which $ \Delta_{g_{\Sigma}} v' + |A_{\Sigma}|^2 v'+ \tau_k \vec e_y \cdot \nabla_{M} v' = E' + \sum_{j=1}^2 \theta_{E',j} w_j + \sum_{i=1}^4  \varphi_{E',i} \bar w_i$. The function $ v = \frac{1}{m_k} v' \circ \hcal_k^{-1} = \frac{1}{m_k} \hcal_{k\ast} v'$ therefore satisfies
	\[
	\lin_{M} v = \Delta_{g_M} v + |A_{M}|^2 v+\bar \tau \vec e_y \cdot \nabla_{g_M} v =E + m_k \hcal_{k\ast} \left( \sum_{j=1}^2 \theta_{E',j} w_j + \sum_{i=1}^4  \varphi_{E',i} \bar  w_i\right).
	\]
To simplify notations, we define the linear map $\Theta: \vcal_{\theta} \times \vcal_{\varphi} \to C^{\infty} (M)$ by 
	\[
	\Theta(\ubar \theta' , \uvarphi') = \sum_{k=1}^{N_I} m_k \hcal_{k \ast} \left(\sum_{j=1}^2 \theta'_{k,j} w_j+ \sum_{i=1}^4 \varphi_{k,i}'\bar w_i \right),
	\]
where $\ubar \theta' = \{  \{ \theta'_{k,j} \}_{j=1}^2 \}_{k=1}^{N_I} \in \vcal_{\theta}$ and $\uvarphi' = \{ \{ \varphi'_{k,i}\}_{i=1}^4\}_{k=1}^{N_I}\in \vcal_{\varphi}$.

\begin{theorem}
\label{thm:linear-invertible}
Given $E \in C^{0,\alpha}(M)$ with finite norm $\| E \|_0$, there exist $v \in C^{2,\alpha}(M)$, $\ubar \theta_E \in \vcal_{\theta}$, and $\uvarphi_E \in \vcal_{\varphi}$ uniquely determined by the construction below, such that
	\[
	\lin_{M} v = E + \Theta(\ubar \theta_E, \uvarphi_E),
	\]
and 
	\[
	\| v \|_2 \leq C \| E\|_0, \quad |\ubar \theta_E |\leq C \| E\|_0, \quad |\uvarphi_E |\leq C \|E\|_0.
	\]
\end{theorem}

\begin{proof} The proof uses an iteration: at the $n$th step, we define the functions $v_n$ and $E_n$;   $u,u'$, and $u''$ are just intermediate functions  and are reset after every step. 

We define two cut-off functions on $M$ by $\psi = \psi[5 \delta_s /\tau, 5 \delta_s /\tau -1]\circ s$ and by $\psi'= \psi[\ubar a ,\ubar a+1 ]\circ s$, where $\ubar a=8 |\log \btau|$ as in Definition \ref{def:ubara}.

We take $E_0:=E$ and proceed by induction. Given $E_{n-1}$, we define $E_n$, $v_n$, $\ubar \theta_n$ and $\uvarphi_n$ in the following way.

For each $k=1, \ldots, N_I$, we consider the function $E'_k = \frac{1}{m_k} (\psi E_{n-1}) \circ \hcal_k$ on $\Sigma_k = \hcal_k^{-1}(\scal_k)$ and apply Proposition \ref{prop:7-1} to  get $v_{E'_k}$, $\ubar \theta_{E'_k}$ and $\ubar \varphi_{E'_k}$. We take $\ubar \theta_{n} = \{\{ \theta_{E'_k, j}\}_{j=1}^2\}_{k=1}^{N_I}$, and $\uvarphi_{n} = \{\{ \varphi_{E'_k,i}\}_{i=1}^4\}_{k=1}^{N_I}$. From the construction, $u = \sum_{k=1}^{N_I} \frac{1}{m_k} \hcal_{k\ast} (u_k)$ satisfies on $\scal$
	\begin{gather*}
	\lin_M u = \psi E_{n-1} + \Theta(\ubar \theta_n, \uvarphi_n),\\
	\| u: C^{2,\alpha} (\scal, g_M, e^{-\gamma  s}) \| \leq C \| E_{n-1} \|_0.
	\end{gather*}

Note that $
	\lin_M (\psi u) =  \psi^2 E_{n-1} +   [\lin, \psi] u + \Theta(\ubar \theta_n, \uvarphi_n), 
	$
where we used the notation $[\lin_M, \psi] u:= \lin_M (\psi) u - \psi(\lin_M u)$. 

From the discussion in Section \ref{ssec:ecal-ncal},  on every edge $\ecal$, there exists a solution $u'$ to  the Dirichlet problem $\lin_M u'= E_{n-1}  - \psi^2 E_{n-1} -   [\lin, \psi] u$ with zero boundary conditions. The function $E_{n-1}  - \psi^2 E_{n-1} -   [\lin, \psi] u$ is supported on $s \geq \frac{5 \delta_s}{\tau} -1$,  therefore
	\begin{align}
	\notag \| u': C^{2,\alpha} (\ecal) \| &\leq C \| E_{n-1}  - \psi^2 E_{n-1} -   [\lin, \psi] u : C^{0,\alpha} (\ecal)\|,\\
	\label{eq:est-u'}					&\leq C e^{- \gamma  (5\delta_s \max_k(m_k)/\btau -1)} \| E_{n-1} \|_0 \leq C e^{-2 \delta_s /\btau} \| E_{n-1} \|_0.
	\end{align}
We will also denote by $u'$ the sum of the solutions $u'$'s on all the edges extended by zero to the rest of $M$.

On every end $\ncal$, there is a solution $u''$  to $\lin_M u''=E_{n-1}  - \psi^2 E_{n-1} -   [\lin, \psi] u$  with exponential decay and given up to a constant on $\pd \ncal$ by Corollary \ref{cor:lin-op-ncal}. As before, we also denote by $u''$ the sum of the solutions on the ends, extended by zero to the rest of $M$.  Since $E \in C^{0,\alpha}(\ncal, g, e^{-\gamma s/2})$, $u''$ satisfies 
	\begin{align}
	 \notag \| u'': C^{2,\alpha} (\ncal, g, e^{-\gamma s/2}) \| &\leq C \| E_{n-1}  - \psi^2 E_{n-1} -   [\lin, \psi] u : C^{0,\alpha} (\ncal, g, e^{-\gamma s/2})\|,\\
	\label{eq:est-u''} 
						&\leq C e^{- \gamma (5\delta_s/ \btau -1)/2} \| E_{n-1}  \|_0 \leq C e^{-2 \delta_s /\btau} \| E_{n-1} \|_0.
	\end{align}

	We choose to focus on the decay of order $e^{-\gamma s/2}$, as opposed to $e^{-\gamma s}$, for the iteration process. The function $u''$ does have a decay of order $e^{-\gamma s}$ and the estimate  $\| u'': C^{2,\alpha} (\ncal, g, e^{-\gamma s}) \| \leq \| E_{n-1}  - \psi^2 E_{n-1} -   [\lin, \psi]u \|_0\leq C \|E_{n-1}\|_0$ is true, although it does not help us with the iteration.

	We define  $v_n = \psi u +\psi'(u'+u'')$. 
	Since the supports of $\psi'$ and $1-\psi^2$ are disjoint, as well as the supports of $\psi'$ and $[\lin, \psi]$, we have
	\begin{align*}
	\lin v_n =  E_{n-1}  +   [\lin, \psi'] (u'+u'')+ \Theta(\ubar \theta_n, \uvarphi_n).
	\end{align*} 

Define   $E_n = - [\lin, \psi'] (u'+u'')$. By \eqref{eq:est-u'}, \eqref{eq:est-u''}, and the fact that $[\lin,\psi']$ is supported on $[\ubar a, \ubar a+1]$, we have, for $\btau$ small enough,
 	\begin{align}
	\notag\| E_n\|_0 &\leq C e^{\gamma (a+1)} \|[\lin, \psi'] (u'+ u'') : C^{0,\alpha} (M, g_M) \| \\
	\notag	&\leq C e^{\gamma (a+1)} ( \| u': C^{2,\alpha} (\ecal) \| + \| u'' : C^{2,\alpha} (\ncal_{ \leq (a+1)}) \|) \\
	\notag	&\leq C e^{\gamma (a+1)} (1+ e^{\gamma (a+1)} )e^{-2 \delta_s /\btau} \| E_{n-1} \|_0 \\
	\label{eq:contraction}	& \leq e^{-\delta_s/\btau} \| E_{n-1}\|_0.
	\end{align}

We define $v_E := \sum_{n=1}^{\infty} v_n$, $\ubar \theta_{E} := \sum_{n=1}^{\infty} \ubar \theta_n$, and $\uvarphi_E := \sum_{n=1}^{\infty} \uvarphi_n$. The three series converge and we have the desired estimates from \eqref{eq:contraction} and Proposition \ref{prop:7-1}. The function $v_E$ is uniquely determined from the construction and satisfies $\lin_M v_E  =E + \Theta(\ubar \theta_E, \uvarphi_E)$. 
\end{proof}


\begin{cor}
\label{cor:7-11}
There are $v_H \in C^{2,\alpha}(M)$ and $\ubar \theta_H,\uvarphi_H$ such that 
	\begin{gather*}
	\lin_M v_H = H_{M} -\bar \tau  \vec e_y \cdot \nu_{M} + \Theta(\utheta_H, \uvarphi_H),\\
	|\utheta_H - \utheta| \leq C \bar \tau, |\uvarphi_H -\uvarphi| \leq C \bar \tau, \| v_H\|_2 \leq C \bar \tau,
	\end{gather*}
where $M = M(\btau, \xi)=M(\btau, \ubar \theta, \btau \ubar b', \btau \ubar c', \ubar \varphi)$. 
\end{cor}

\begin{proof}
From the smooth dependence of $M$ on its parameters, the uniform control of the geometry of $M_{\leq 2} = \scal_{\leq 2}$, and $|\xi|\leq \zeta \btau$, we have
	\[
	\|H_{M} - \btau \vec e_y \cdot \nu_{M}: C^{0,\alpha}(\scal_{\leq 2}, g_M) \| \leq C \btau.
	\]
On the regions  $\ecal''$ and $\ncal''$, $H_{M} -\bar \tau  \vec e_y \cdot \nu_{M} \equiv 0$. 
	We use Proposition \ref{prop:4-20} and Corollary \ref{cor:sigma-parameters} on $\scal_{\geq 1}$ to get
	\[
	\|H_{M} -\bar \tau  \vec e_y \cdot \nu_{M} + \Theta(\ubar \theta + \ubar \theta', \tilde{\uvarphi})\|_0 \leq C \btau,
	\]
where $ \tilde{\uvarphi}$ is as in Section \ref{sec:initial-surfaces} and $\ubar \theta'=\{ \{\theta'_{k,j}\}_{j=1}^2\}_{k=1}^{N_I}$, with $\theta'_{k,j} = (\tilde \varphi_{k,j+2} - \tilde \varphi_{k,j}+\varphi_{k,j+2} - \varphi_{k,j})/2$. We apply Theorem  \ref{thm:linear-invertible} with $E= H_{M} -\bar \tau  \vec e_y \cdot \nu_{M} + \Theta(\ubar \theta + \ubar \theta', \tilde{\uvarphi})$ to obtain $v_E$, $\ubar \theta_E$ and $\uvarphi_E$ and define $v_H= v_E$, $\theta_H = \ubar \theta+\ubar \theta' + \ubar \theta_E$, and $\ubar \varphi_H = \tilde{\uvarphi} + \uvarphi_E$. The estimates follow from  Theorem \ref{thm:linear-invertible}, Propositon \ref{prop:init-surf} (vi) and Corollary \ref{cor:sigma-parameters}. 
\end{proof}

\section{Quadratic Term}
\label{sec:quadratic}

\begin{prop}
\label{prop:8-1}
Given $v \in C^{2,\alpha}(M)$ with $\| v \|_2$ smaller than a suitable constant, the graph $M_v$ of $v$ over $M$ is a smooth immersion, moreover
	\[
	\| H_v - \btau \vec e_y \cdot \nu_v - (H- \btau \vec e_y\cdot \nu) - \lin_M v\|_0 \leq C \| v\|_2^2,
	\]
where $H$ and $H_v$ are the mean curvature of $M$ and $M_v$ pulled back to $M$ respectively, and similarly, $\nu$ and $\nu_v$ are the oriented unit normal of $M$ and $M_v$ pulled back to $M$.
\end{prop}

\begin{proof}
The result follows immediately from equation \eqref{eq:perturb-st}  since we have uniform control of the second fundamental form $A_M$. 
\end{proof}

\section{Fixed Point Argument}
\label{sec:fpt}

We are now ready to prove the main result of this paper. Theorem \ref{thm:main} in the introduction is a rescaled version of the theorem below, where $\wilde{\mcal} = \btau M_v$. The larger scale avoids singularities when $\btau \to 0$.

\begin{theorem} 
\label{thm:8-2}
Given a finite family of grim reapers $\{\wilde{\Gamma}_{n}\}_{n=1}^{N_{\Gamma}}$ in general position as in Section \ref{sec:initial-conf}, let us denote the rescaled family by $\{ \Gamma_n\}_{n=1}^{N_{\Gamma}}$, where $\Gamma_n = \btau^{-1} \wilde{\Gamma}_n$.  There is a $\delta_{\btau}$ depending only on $\max_{k}\{m_k\}$,  $N_{\Gamma}$, $\delta$, and $\delta_{\Gamma}$  from Lemma \ref{lem:cor-cond} such that for every $\btau \in (0,\delta_{\btau})$, there is a $\xi_{\btau} \in \vcal$ and a smooth function $v$ on the smooth initial surface $M=M(\btau, \xi_{\btau})$ with the following properties:
	\begin{enumerate}
	\item The graph $M_{v}$ of $v$ over $M$ is a complete embedded surface in $\R^3$ that is self-translating under mean curvature flow.
	\item $M_{v}$ is invariant under reflection with respect to the $xy$-plane.
	\item  $M_{v}$ is singly periodic with period $2 \pi$ in the $z$-direction.
	\item There is a large ball $B\subset \R^2$ such that $M_v \setminus(B\times \R)$ is the union of ends  in one-to-one correspondence with the ends of  $\{{\Gamma}_{n}\}_{n=1}^{N_{\Gamma}}$. The left ends of $M_{v}$ are exponentially decaying to grim reaper ends in $\{{\Gamma}_{n}\}_{n=1}^{N_{\Gamma}}$ (without any change in the position of the grim reapers). 
For the right ends of $M_v$, the difference between the center of a grim reaper asymptotic to $M_{v}$ and the center of its corresponding end in $\{{\Gamma}_{n}\}_{n=1}^{N_{\Gamma}}$  is at most a constant $C(N_{\Gamma},\delta,  \delta_{\Gamma}, \max_k(m_k))$. 
	\item  If $T_k^{\btau}$ is the translation that moves the $k$th intersection line of $\{\Gamma_{n}\}_{n=1}^{N_{\Gamma}}$ to the $z$-axis, then $T_k^{\btau}(M_v)$ converges uniformly in $C^{j}$ norm, for any $j<\infty$, on any compact set of $E^3$ to a Scherk surface of period $2 \pi/m_k$ as $\btau \to 0$. 

\end{enumerate}
\end{theorem}

\begin{proof}
Let us fix $\alpha' \in (0 ,\alpha)$, and define the Banach space
	\[
	\chi=C^{2, \alpha'} (M(\btau, \ubar 0)).
	\]
Denote by $D_{\btau, \xi}: M(\btau, \ubar 0) \to M(\btau, \xi)$ a family of smooth  diffeomorphisms which depend smoothly on $\xi$ and satisfy the following conditions: for every $f \in \mathcal C^{2, \alpha}(M(\btau, \ubar 0))$ and $f' \in C^{2,\alpha}(M(\btau, \xi))$, we have
	\begin{align*}
	\| f \circ D_{\btau, \xi} ^{-1} \|_2 \leq C\| f \|_2, \quad \| f' \circ D_{\btau, \xi}\|_2 \leq C \|f'\|_2.
	\end{align*}
The diffeomorphisms $D_{\btau, \xi}$ are used to pull back  functions and norms from $M(\btau, \xi)$ to  $M(\btau, \ubar 0)$. 

We fix $\btau$ for now and omit the dependence in $\btau$ in our notations of maps and surfaces from now on. Let
	\[
	\Xi = \{(\xi, u) \in \vcal \times \chi: |\xi| \leq \zeta  \btau, \| u \|_2 \leq \zeta \btau\},
	\]
where $\zeta$ is a large constant to be determined below. 
The map $\ical: \Xi \to \vcal \times \chi$ is defined as follows. Given $(\xi, u) \in \Xi$, let $v=u \circ D_{\xi}^{-1}$, $M = M(\xi)$ and let  $M_v$ be the graph of $v$ over $M$. We define the function $\F:\vcal \times C^{2,\alpha}(M, g_M, e^{-\gamma s}) \to \R$ by 
	\[
	\F(\xi, v) = H_v -\btau \vec e_y \cdot \nu_v,
	\]
where $H_v$ and $\nu_v$ are the mean curvature and the oriented unit normal of $M_v$ respectively pulled back to $M$. Proposition \ref{prop:8-1} asserts that
	\[
	\|\F(\xi, v) - \F (\xi, 0) -  \lin_M v \|_0 \leq C \zeta^2  \btau^2.
	\]
Applying Theorem \ref{thm:linear-invertible} with $E = \F(\xi, v) - \F (\xi, 0) -  \lin_M v $, we obtain $v_E$, $\utheta_E$, and $\uvarphi_E$ such that 
	\begin{gather*}
	\lin_M v_E = E  + \Theta(\utheta_E, \uvarphi_E),\\
	\|v_E \|_2 \leq C \zeta^2  \btau^2, \quad |\utheta_E | \leq C \zeta^2  \btau^2, \quad |\uvarphi_E| \leq C \zeta^2 \btau^2.
	\end{gather*}
Hence,
	\[
	\F(\xi, v) = \F(\xi, 0) + \lin_M v + \lin_M v_E  - \Theta(\utheta_E, \uvarphi_E).
	\]
Corollary \ref{cor:7-11} gives us $v_H$, $\utheta_H$ and $\uvarphi_H$ satisfying 
	$
	\lin_M v_H = \F(\xi, 0) + \Theta(\utheta_H, \uvarphi_H), 
	$
 so
 	\[
	\F(\xi, v) =   \lin_M v + \lin_M v_H+ \lin_M v_E -\Theta(\utheta_E + \utheta_H, \uvarphi_E + \uvarphi_H).
	\]
We define the map $\mathcal I: \Xi \to \vcal \times \chi$ by 
	\[
	\mathcal I(\xi, u) = ((\utheta-\utheta_E - \utheta_H, -\btau \tilde{\ubar b}, -\btau \tilde{\ubar c}, \uvarphi-\uvarphi_E-\uvarphi_H),(-v_E-v_H)\circ D_{\xi} ).
	\] 
Note that we arrange for $\mathcal I (\Xi )  \subset \Xi$ since
	\begin{align*}
	\| - v_E -v_H \|_2 &\leq C( \btau +  \zeta ^2  \btau^2),\\
	|\utheta-\utheta_E - \utheta_H| &\leq  C( \btau +  \zeta ^2  \btau^2),\\
	|\uvarphi-\uvarphi_E-\uvarphi_H| &\leq  C( \btau +  \zeta ^2  \btau^2),
	\end{align*}
and choosing $\zeta > 2C$ and $ \btau < \zeta^{-2}$, we get $C( \btau +  \zeta ^2  \btau^2) < \zeta  \btau$.

The set $\Xi$ is clearly convex. It is a compact set of $\vcal \times \mathcal X$ from the choice of the H\"older exponent $\alpha'<\alpha$ and the imposed exponential decay. The map $\mathcal I$ is continuous by construction, therefore we can apply the Schauder Fixed Point Theorem (p. 279 in \cite{gilbarg-trudinger;elliptic-equations}) to obtain a fixed point $(\xi_{ \btau}, u_{ \btau})$ of $\mathcal I$ for every $ \btau \in (0, \delta_{\btau})$ with $\delta_{\btau}$ small enough. The graph of $v=u_{ \btau} \circ D_{\btau, \xi}^{-1}$ over the surface $M(\btau,\xi_{\btau})$ is  then a self-translating surface. It is a smooth surface by the regularity theory for elliptic equations. The properties (ii) and (iii) follow from the construction. 
\end{proof}

Part (iv) of the theorem quantifies how much the asymptotic grim reapers are perturbed. The construction is slightly more general than advertised since we can impose the position of the left asymptotic grim reapers to be any slight perturbation of the initial family, although we have to surrender precise control of the ones on the right (see also Kapouleas \cite{kapouleas;embedded-minimal-surfaces}).


\appendix
\section{Linear Operator on Long Cylinders}
\label{sec:linop-cylinders}


We give here the proof of Proposition \ref{prop:A3} from Section \ref{sec:lin-op}. Kapouleas has a similar result, without the gradient term $\mathbf V \cdot \nabla $ in Appendix A  of \cite{kapouleas;embedded-minimal-surfaces}. His reasoning applies readily here as well, except  we have to work a little more for the uniqueness in (vi), as shown below. 

Let us recall  that $\lin$ denotes an operator on the cylinder $(\Omega, g_0)$ of the form
	\[
	\lin v = \Delta_{\chi} v + \mathbf{V} \cdot \nabla v+ d v,
	\]
where $\chi$ is a $C^2$ Riemannian metric, $\mathbf{V}$ a $C^1$ vector field, and $d$ a $C^1$ function on $\Omega$. For $\ubar c >0$, we define 
	\begin{equation*}
	N(\lin) := 
	\| \chi -g_0:C^2(\Omega,g_0, e^{-\ubar c s}) \|   + \| \mathbf{V} :C^1(\Omega,g_0, l^{-1})\|  + \| d :C^1(\Omega,g_0,e^{-\ubar c s} + l^{-2})\|,
	\end{equation*}
where $l$ is the length of the cylinder (see Definition \ref{def:cylinder}).

\begin{varprop}
Given $\gamma \in (0,1)$ and $\eps>0$, if $N(\lin)$ is small enough in terms of $\ubar c$, $\alpha$, $\gamma$ and $\eps$ (but independently of $l$), 
there is a bounded linear map 
	\[
	\ubar{\mathcal R}: C^{2,\alpha}(\pd_0,g_0) \times C^{0, \alpha}(\Omega, g_{0}, e^{-\gamma s}) \to C^{2, \alpha}(\Omega, g_{0}, e^{-\gamma s})
	\]
such that for $(f,E)$ in $\ubar{\mathcal R}$'s domain and $v = \ubar{\mathcal R}(f,E)$, the following properties are true, where the constants $C$ depend only on $\alpha$ and $\gamma$:
	\begin{enumerate}
	\item   $\lin v = E$ on $\Omega$.
	\item  $v= f -\textrm{\em avg}_{\pd_0} f +B(f,E)$ on $\pd_0$, where $B(f,E)$ is a constant on $\pd_0$ and $\textrm{\em avg}_{\pd_0} f$ denotes the average of $f$ over $\pd_0$. 
	\item  $v\equiv 0$ on $\pd_l$.
	\item $
	\| v :C^{2, \alpha}(\Omega, g_{0}, e^{-\gamma s}) \| $\\
	$\leq C \| f-\textrm{\em avg}_{\pd_0} f:C^{2, \alpha}(\pd_0, g_{0})\|+ C \| E: C^{0, \alpha}(\Omega, g_{0}, e^{-\gamma s})\|.
	$
	\item $|B(f,E)|\leq \eps\| f-\text{\em avg}_{\pd_0} f: C^{2, \alpha}(\pd_0, g_{0})\|+ C \| E: C^{0, \alpha}(\Omega, g_{0}, e^{-\gamma s})\|.
	$
	\item If $v' \in C^2(\Omega)$ satisfies $\lin v'=E$ on $\Omega$, and $v=v'$ on $\pd \Omega$, then $v'=v$ on $\Omega$. Moreover, if  $E$ vanishes, then 
	\[
	\| v : C^0(\Omega) \| \leq 2 \| v:C^0(\pd_0)\|.
	\]
	\end{enumerate}
\end{varprop}

\begin{proof} 
The proposition is valid for the standard Laplacian with respect to the flat metric, with a vanishing $\eps$ in (v). One can prove this fact by separating variables, and using Fourier series with coefficients depending on $s \in (0,l)$. 
In the case of vanishing boundary conditions, the operator $\lin: C^{2,\alpha}(\Omega,g_0,e^{-\gamma s}) \to C^{0,\alpha}(\Omega,g_0,e^{-\gamma s})$ is a  perturbation of $\Delta$, therefore the statements (i)-(v) follow. If the boundary conditions are not zero, we can find the solution $u$ to the Laplace equation with the given boundary conditions, then solve $\lin v =E - \lin u$ with vanishing boundary data. 

We now prove uniqueness of the solution by showing that the smallest eigenvalue of $\lin$ is bounded  away from $0$. 
Let $\phi$ be a $C^{2,\alpha}$ function that vanishes on $\pd \Omega$. We denote the average of $\phi$ by $\bar \phi (s)=(2 \pi)^{-1} \int_0^{2\pi} \phi(s,\theta) d\theta$.  From the Poincar\'e inequality
	$
	\int_0^{2 \pi} |\phi(s, \theta)-\bar \phi(s)|^2 d\theta \leq 4\pi^2 \int_0^{2 \pi} |\pd_{\theta} \phi(s,\theta)|^2 d\theta,
	$
 we get
	\begin{align}
	|\bar \phi(s)|^2 
		\notag &= \left(\int_0^s \bar \phi'(t) dt \right)^2 
		\leq s \int_0^s |\bar \phi'(t)|^2 dt
		 \leq s \int_0^{s}\!\!\! \int_{0}^{2\pi} |\nabla \phi|^2 d\theta dt\\
		 \label{eq:bar-phi-2} &\leq s \int_0^{l}\!\!\! \int_{0}^{2\pi} |\nabla \phi|^2 d\theta dt.
	\end{align}
Since 
	$
	\int_0^{2 \pi} (\phi-\bar \phi)^2 d\theta = \int_0^{2 \pi} \phi^2 d\theta -2 \pi \bar \phi^2,
	$
we have 
	\begin{align*}
	\int_0^{l}\!\!\!  \int_{0}^{2\pi} |\phi|^2 d\theta ds
		 & = \int_0^{l}\!\!\!  \int_{0}^{2\pi} (\phi-\bar\phi)^2 d\theta ds + 2 \pi\int_0^{l} |\bar \phi|^2 ds   \\
		 &\leq (4 \pi^2 + \pi l^2)\int_0^{l}\!\!\! \int_{0}^{2\pi} |\nabla \phi|^2 d\theta ds.
	\end{align*}	
The inequality above proves uniqueness for $\Delta$. For $\lin$, we write
	\[
	\int_0^l \!\!\!\int_0^{2\pi} e^{-\ubar cs}\phi^2(s,\theta) d\theta ds = \int_0^l \!\!\!\int_0^{2\pi} e^{-\ubar cs}(\phi-\bar \phi)^2 d\theta ds + 2\pi \int_0^l e^{-\ubar cs} \bar \phi^2 ds =:I+I\!I,
	\]
where $\ubar c$ is as in the definition of $N(\lin)$. 
$I$ can be estimated using Poincar\'e's inequality,
	\begin{align*}
	\int_0^l\!\!\! \int_0^{2\pi} e^{-\ubar cs} (\phi-\bar \phi)^2 d\theta ds &\leq C \int_0^l \left(\int_0^{2\pi} |\nabla \phi|^2 d\theta \right) e^{-\ubar cs} ds \\
		& \leq C \| e^{-\ubar cs}\|_{L^{\infty}} \int_0^l\!\!\!\int_0^{2\pi} |\nabla \phi|^2 d\theta ds = C \int_0^l \!\!\!\int_0^{2\pi} |\nabla \phi|^2 d\theta ds.
	\end{align*}
$I\!I$ is estimated using \eqref{eq:bar-phi-2},
	\begin{align*}
	I\!I &\leq 2 \pi \left(\int_0^l e^{-\ubar cs} s ds\right) \int_0^{l}\!\!\! \int_{0}^{2\pi} |\nabla \phi|^2 d\theta ds
	\leq \frac{2\pi}{\ubar c^2} \int_0^{l}\!\!\! \int_{0}^{2\pi} |\nabla \phi|^2 d\theta ds.
		\end{align*}
Hence, we have
	$
	\int_{\Omega} e^{-\ubar cs} \phi^2 \leq C(\ubar c) \int_{\Omega} |\nabla \phi|^2,
	$
and
	\begin{align*}
	-\int_{\Omega} ( \lin \phi) \phi =& -\int_{\Omega} (\Delta_{g_0} \phi) \phi - \int_{\Omega}d \phi^2 -\int_{\Omega}\phi \mathbf V\cdot \nabla_{g_0} \phi \\
					&  + \int_{\Omega} (\Delta_{g_0} \phi-\Delta_{\chi} \phi) \phi +\int_{\Omega}\phi \mathbf V\cdot (\nabla_{g_0} \phi-\nabla_{\chi} \phi )  \\
					\geq &\int_{\Omega}|\nabla \phi|^2 - C N(\lin) \left(\int_{\Omega} (C(\ubar c)+1)  |\nabla \phi|^2 +\frac{1}{l^2} \int_{\Omega}\phi^2\right)\\
					\geq & \frac{1}{2 (4 \pi^2 + \pi l^2) }\int_{\Omega} \phi^2,
	\end{align*}
if $N(\lin)$ is small enough depending on $C(\ubar c)$, but independently of $l$. Therefore, if $ \lin \phi=0$ and $\phi$ vanishes on the boundary of $\Omega$, $\phi \equiv 0$. 

For the estimate in (iv), we quote Kapouleas (Appendix A of \cite{kapouleas;embedded-minimal-surfaces}): The desired estimate reduces to the case where $v=f\equiv 1$ on $\pd_0$, because otherwise we can produce a subdomain of $\Omega$ with a vanishing eigenvalue. If $\lin =\Delta$, the solution is $v = (l-s)/l$. This can be corrected to the solution to $\lin$ of the form $\Delta +\mathbf V \cdot \nabla+d$ with $\| \mathbf V:C^1(\Omega, l^{-1})\|$ and $\| d:C^1(\Omega, l^{-2})\|$ appropriately small, by scaling the length of the cylinder to unit, while leaving the meridian unchanged. We thus can have the estimate for $v$ with a constant $3/2$ for example instead of $2$. By using now the earlier proven parts of the proposition we can correct this $v$ to a $v$ for a general $\lin$ while establishing the estimate at the same time. 
\end{proof}

Proposition \ref{prop:A4} is proved by treating $\lin$ as a perturbation of the Laplace operator $\Delta$, for which the results are standard.


\section{Proof of Lemma \ref{lem:7-4}}
\label{sec:proof}


Let us recall that $h :=(\frac{1}{2}|A_{\Sigma}|^2 + \eps_h) g_{\Sigma}$, where $\eps_h >0$ is small depending on $\tau$ and $\delta_{\theta}$. The $c$-approximate kernel of $L_h:= \Delta_h + 2 |A_{\Sigma}|^2 /(|A_{\Sigma}|^2 + 2\eps_h)$ is the span of the eigenfunctions of $L_h$ with corresponding eigenvalues in $[-c,c]$, and  $G'$ is the group generated by $(s,z) \to (s, z + 2 \pi)$.

\begin{varlemma}
There are positive constants $C$ and $c$ such that given $E \in L^2(\Sigma/G', h)$ there is $\ubar \theta_E = (\theta_{E,1}, \theta_{E,2})$ such that $(E- \sum_{j=1}^2 \theta_{E,j} w_j)/(|A_{\Sigma}|^2 + 2\eps_h)$ is $L^2(\Sigma/G', h)$-orthogonal to the $c$-approximate kernel and 
	\[
	|\ubar \theta_E | \leq C \| E/(|A_{\Sigma}|^2 + 2\eps_h) : L^2 (\Sigma/G',h)\|.
	\]
\end{varlemma}

	\begin{proof}
First, we prove that all eigenfunctions $f$ of $L_h$ on $\Sigma$ corresponding to low eigenvalues, say less than $10$, satisfy
	\begin{equation}
	\label{eq:last}
	\| f : C^0 (\Sigma) \| \leq C \| f: L^2 (\Sigma /G', h) \|.
	\end{equation}
Indeed, we control the geometry of  $\Sigma_{\leq 2}$ uniformly so we can  use standard elliptic theory to bound $\| f:C^0(\Sigma_{\leq 2})\|$ in terms of $ \| f: L^2 (\Sigma_{\leq 2} /G', h) \|$. In the region $s \geq 1$, we consider $f$ as a solution to the equation $ \Delta_{\Sigma} f +  |A_{\Sigma}|^2 f  + \frac{\lambda}{2} (|A_{\Sigma}|^2 + 2 \eps_h) f =0$. Using Corollary \ref{cor:4-4} and the fact that $\eps_h$ is small, we can apply the result (vi) of Proposition \ref{prop:A3} to obtain $\| f : C^0(\Sigma_{\geq 1}) \| \leq 2 \| f: C^0 (\pd \Sigma_{\leq 1})\|$, which establishes \eqref{eq:last}.

Since $\Sigma(\theta(T))$ is minimal without umbilical points, the Gauss map is conformal. It maps $\Sigma(\theta(T))/G'$  to a sphere minus the four points $(\pm \sin \theta(T), \pm \cos \theta(T), 0)$. By the Weierstrass-Ennerper representation of Scherk surfaces from \cite{karcher;minimal-surfaces}, the standard metric on the sphere pulled back to $\Sigma(\theta(T))$ is $h_0 = \frac{1}{2}|A|^2 g_{\Sigma(\theta(T))}$.

Let $\rho: \mathbf S^2 \to \R$ denote the distance from  $\{ (\pm \sin \theta(T), \pm \cos \theta(T), 0)\}$. We define a logarithmic function $\psi_{\mathbf S^2}: \mathbf S^2 \to [0,1]$ by 
	\[
	\psi_{\mathbf S^2}(p) = \psi [2,1] (\log \rho(p) / \log \delta_h),
	\]
where $\delta_h$ is a small positive constant to be determined in the course of the proof. Notice that $\psi_{\mathbf S^2}$ vanishes at distance $\leq \delta_h^2$ from these points, and $\psi_{\mathbf S^2} \equiv 1$ at distance $\geq \delta_h$.

For a function $f$ on $\Sigma(\theta(T))/G'$, we define $\fcal_1(f)$ to be the pushforward to $\Sigma/G'$ by $Z=Z[T,\uvarphi, \tau]$ of the function $f \psi_{\mathbf S^2}\circ \nu$.  Similarly, for a function $f'$ on $\Sigma/G'$, $\fcal_2 (f')$ is a function on $\Sigma(\theta(T))/G'$ defined by $\fcal_2(f')= (f' \circ Z) (\psi_{\mathbf S^2} \circ \nu)$. 
The region on $\Sigma(\theta(T))/G'$ where $\psi_{\mathbf S}^2 \circ \nu \neq 1$ is contained in four disks of radius $\delta_h$, therefore its $h$-area is small for $\delta_h$ small. On $\Sigma/G'$, the region where $\psi_{\mathbf S}^2 \circ \nu \neq 1$ has small area also. Indeed, the points at distance $\delta_h$ of $(\pm \sin \theta(T), \pm \cos \theta(T), 0)$ all have to same $s$-coordinate on $\Sigma(\theta(T))$ (see \cite{mine;part1} Section 4.3). Let use call this value $s_0$. We  use Lemma \ref{lem:4-3} and Corollary \ref{cor:4-4} to bound the area by
	\begin{align*}
	\int_{(\Sigma_{\geq s_0}) /G'}  dV_h &= \frac{1}{2} \int_{(\Sigma_{\geq s_0}) /G'}  (|A_{\Sigma}|^2 + 2 \eps_h) dV_{g_{\Sigma}}\\
	 	&\leq C \int_{s_0}^l (\eps e^{-s} + \tau^2 +  \eps_h) ds\leq C (e^{-s_0} + (\tau^2+ \eps_h) (l-s_0) ),
	\end{align*}
where $dV_g$ is the volume form associated to the metric $g$ and $l = 5 \delta_s/\tau$. Taking $\eps_h$ small with respect to $\sqrt \tau$ so that $A_{\Sigma}$ is not singular, and taking $\delta_h$ small enough so that $s_0$ is large, we obtain a small $h$-area for the region where $\psi_{\mathbf S}^2 \circ \nu \neq 1$ on $\Sigma$.

Using Rayleigh quotients, one can prove that if two manifolds $(M_1, h_1)$ and $(M_2, h_2)$ are close to being isometric, except on a set of small area, and if there exist $\eps_1>0$ and maps $\fcal_1: C_0^{\infty} (M_1) \to C_0^{\infty}(M_2)$  and $\fcal_2: C_0^{\infty} (M_2) \to C_0^{\infty}(M_1)$ such that for all $f, g \in C_0^{\infty}(M_i)$, $i,j=1,2$, $j\neq i$,
	\begin{gather*}
	\| \fcal_i f \|_{\infty} \leq 2 \| f \|_{\infty},\\
	|\langle f,g\rangle - \langle \fcal_i(f), \fcal_i(g) \rangle | \leq \eps_1 \| f\| _{\infty} \| g \|_{\infty},\\
	\| \nabla (\fcal_i (f)) \|_{L^2} \leq (1 +\eps_1) \| \nabla f\|_{L^2} + \eps_1 \| f \|_{\infty},\\
	\| f - \fcal_j \circ \fcal_i (f) \|_{L^2} \leq \eps_1\| f \|_{\infty}, 
	\end{gather*}
where $\langle \cdot, \cdot \rangle$ stands for the $L^2$ inner product and $\| \cdot \|_{\infty}$ is the $L^{\infty}$ norm, then the eigenvalues of the operator $L_{h_1}$ on $M_1$ and $L_{h_2}$ on $M_2$ are close. Moreover, if $f$ is an eigenfunction of $L_{h_i}$ on $M_i$ with corresponding eigenvalue $\lambda$, $\fcal_i(f)$ is close to a linear combination of eigenfunctions of corresponding eigenvalues close to $\lambda$ on $M_j$. All the  ``closeness" depends on $\eps_1$ and can be estimated. The  reader can find more detail in Appendix B of \cite{kapouleas;surfaces-euclidean} and note that the first condition in B.1.6 was corrected by Kapouleas  (see \cite{kapouleas-yang} page 281) and should be replaced by the one given here. 

One can generate functions close the approximate kernel of $L_h$ on $\Sigma$ by cutting off functions in the approximate kernel of $\Delta_{h_0} +2$ on $\Sigma(\theta(T))$. The eigenfunctions corresponding to  eigenvalues less than $1$ of the operator $\Delta_{h_0} +2 $ on $\Sigma(\theta(T))$ are exactly $ \vec e_x \cdot \nu$ and $ \vec e_y \cdot \nu$. In Lemma \ref{lem:w}, we proved that we can generate any combination of $ \vec e_x \cdot \nu$ and $ \vec e_y \cdot \nu$ with a linear combination of $w_1$ and $w_2$. The bound on $|\ubar\theta_E|$ follows from the fact that $P^{-1}$ in Lemma \ref{lem:w} is bounded.
\end{proof}

\bibliographystyle{siam}
\bibliography{../../mybiblio}

\def\cprime{$'$}
\begin{thebibliography}{10}

\bibitem{altschuler-wu;translating-surfaces}
{\sc S.~J. Altschuler and L.~F. Wu}, {\em Translating surfaces of the
  non-parametric mean curvature flow with prescribed contact angle}, Calc. Var.
  Partial Differential Equations, 2 (1994), pp.~101--111.

\bibitem{angenent;formation-singularities}
{\sc S.~Angenent}, {\em On the formation of singularities in the curve
  shortening flow}, J. Differential Geom., 33 (1991), pp.~601--633.

\bibitem{angenent-velazquez;cusp-singularities}
{\sc S.~B. Angenent and J.~J.~L. Vel{\'a}zquez}, {\em Asymptotic shape of cusp
  singularities in curve shortening}, Duke Math. J., 77 (1995), pp.~71--110.

\bibitem{angenent-velazquez;degenerate-neckpinches}
\leavevmode\vrule height 2pt depth -1.6pt width 23pt, {\em Degenerate
  neckpinches in mean curvature flow}, J. Reine Angew. Math., 482 (1997),
  pp.~15--66.

\bibitem{gilbarg-trudinger;elliptic-equations}
{\sc D.~Gilbarg and N.~S. Trudinger}, {\em Elliptic partial differential
  equations of second order}, vol.~224 of Grundlehren der Mathematischen
  Wissenschaften [Fundamental Principles of Mathematical Sciences],
  Springer-Verlag, Berlin, second~ed., 1983.

\bibitem{huisken;asymptotic-behavior}
{\sc G.~Huisken}, {\em Asymptotic behavior for singularities of the mean
  curvature flow}, J. Differential Geom., 31 (1990), pp.~285--299.

\bibitem{huisken-sinestrari}
{\sc G.~Huisken and C.~Sinestrari}, {\em Mean curvature flow singularities for
  mean convex surfaces}, Calc. Var. Partial Differential Equations, 8 (1999),
  pp.~1--14.

\bibitem{kapouleas;surfaces-euclidean}
{\sc N.~Kapouleas}, {\em Constant mean curvature surfaces in {E}uclidean
  spaces}, in Proceedings of the International Congress of Mathematicians,
  Vol.\ 1, 2 (Z\"urich, 1994), Basel, 1995, Birkh\"auser, pp.~481--490.

\bibitem{kapouleas;embedded-minimal-surfaces}
\leavevmode\vrule height 2pt depth -1.6pt width 23pt, {\em Complete embedded
  minimal surfaces of finite total curvature}, J. Differential Geom., 47
  (1997), pp.~95--169.

\bibitem{kapouleas-yang}
{\sc N.~Kapouleas and S.-D. Yang}, {\em Minimal surfaces in the three-sphere by
  doubling the {C}lifford torus}, Amer. J. Math., 132 (2010), pp.~257--295.

\bibitem{karcher;minimal-surfaces}
{\sc H.~Karcher}, {\em Embedded minimal surfaces derived from {S}cherk's
  examples}, Manuscripta Math., 62 (1988), pp.~83--114.

\bibitem{korevaar-kusner-solomon}
{\sc N.~J. Korevaar, R.~Kusner, and B.~Solomon}, {\em The structure of complete
  embedded surfaces with constant mean curvature}, J. Differential Geom., 30
  (1989), pp.~465--503.

\bibitem{mine;part1}
{\sc X.~H. Nguyen}, {\em Construction of complete embedded self-similar
  surfaces under mean curvature flow. {P}art {I}}, Trans. Amer. Math. Soc., 361
  (2009), pp.~1683--1701.

\bibitem{mine;tridents}
\leavevmode\vrule height 2pt depth -1.6pt width 23pt, {\em Translating
  tridents}, Comm. Partial Differential Equations, 34 (2009), pp.~257--280.

\bibitem{traizet;surfaces-minimales}
{\sc M.~Traizet}, {\em Construction de surfaces minimales en recollant des
  surfaces de {S}cherk}, Ann. Inst. Fourier (Grenoble), 46 (1996),
  pp.~1385--1442.

\end{thebibliography}

\end{document}